\documentclass[11 pt,reqno]{amsart}

\usepackage[latin1]{inputenc}
\usepackage{amsmath,amsthm,amssymb,amscd,mathrsfs,mathtools,xcolor,accents}
\usepackage[shortlabels]{enumitem}
\usepackage[perpage]{footmisc}
\usepackage[normalem]{ulem}
\usepackage{graphicx}
\definecolor{darkgreen}{rgb}{0.0, 0.6, 0.13}

\usepackage{bbm} 
\usepackage{pifont} 
\usepackage{hyperref}

\newtheorem{thm}{Theorem}[section]
 \newtheorem{cor}[thm]{Corollary}
  
 \newtheorem{lem}[thm]{Lemma}
 \newtheorem{prop}[thm]{Proposition}
\newtheorem{claim}[thm]{Claim}
 \theoremstyle{definition}
 \newtheorem{df}[thm]{Definition}
 \theoremstyle{remark}
 \newtheorem{rem}[thm]{Remark}
 
 \numberwithin{equation}{section}
\def\sqw{\hbox{\rlap{\leavevmode\raise.3ex\hbox{$\sqcap$}}$
\sqcup$}}
\def\findem{\ifmmode\sqw\else{\ifhmode\unskip\fi\nobreak\hfil
\penalty50\hskip1em\null\nobreak\hfil\sqw
\parfillskip=0pt\finalhyphendemerits=0\endgraf}\fi}

\newcommand{\mb}{\medskip\noindent}
\newcommand{\gb}{\bigskip\noindent}
\newcommand{\noi}{\noindent}
\newcommand{\R}{\mathbb R}
\newcommand{\M}{\mathcal M}

\newcommand{\N}{\mathbb N}
\newcommand{\Z}{\mathbb Z}
\newcommand{\C}{\mathbb C}

\newcommand{\F}{\mathcal F}

\newcommand{\T}{{\mathbb T}}

\newcommand{\Nc}{\mathcal{N}}

\newcommand{\Rc}{\mathcal{R}}
\newcommand{\Qc}{\mathcal{Q}}
\newcommand{\ind}{\mathbbm 1}

\newcommand{\jb}[1]
{\langle #1 \rangle}
\newcommand{\jbb}[1]
{[\hspace{-0.6mm}[ #1 ]\hspace{-0.6mm}]}
\DeclareMathOperator{\med}{med}

\allowdisplaybreaks[4]

\begin{document}
\title[Invariant Gibbs measures and  global strong solutions for 2D fNLS]{Invariant Gibbs measures and global dynamics for fractional cubic Schr\"{o}dinger equations on the torus}

\author[Y. Wang]{Yuzhao Wang}
	\address{School of Mathematics, University of Birmingham,  Edgbaston,
Birmingham, B15 2TT, United Kingdom}
	\email{y.wang.14@bham.ac.uk}
\author[H. Yue]{Haitian Yue}
	\address{Institute of Mathematical Sciences, ShanghaiTech University, Pudong, Shanghai,  China}
	\email{yuehaitian@shanghaitech.edu.cn}
	\author[C. Zhang]{Chenyuan Zhang}
	\address{School of Mathematical Sciences,
		University of Science and Technology of China, Hefei 230026, Anhui, China}
	\email{roundround@mail.ustc.edu.cn}
	\author[L. Zhao]{Lifeng Zhao}
	\address{School of Mathematical Sciences,
		University of Science and Technology of China, Hefei 230026, Anhui, China}
	\email{zhaolf@ustc.edu.cn}


\dedicatory{}

\begin{abstract}
We consider the defocusing Wick-ordered cubic fractional nonlinear
Schr\"odinger equation on the two-dimensional torus with dispersion relation
\(\omega(k)=|k|^\alpha\). In the weakly dispersive regime
\(\frac{29}{15}<\alpha<2\), we construct global dynamics for almost
every initial datum with respect to the associated Gibbs measure as the limit of the finite-dimensional truncated flows and
prove the invariance of the Gibbs measure. The core of the proof is an
almost sure local theory based on the method of random averaging operators
\cite{DNY2024}. The main new ingredients are fractional lattice counting
estimates and localized random tensor bounds, which exploit the geometric
structure of the fractional phase in place of the classical number-theoretic
tools available for quadratic dispersion.
\end{abstract}

\maketitle

\section{Introduction} 

\subsection{Overview}

We study the defocusing fractional nonlinear Schr\"{o}dinger equation
\begin{equation}
    \label{fNLS}
\begin{cases}
    (i\partial_t- {\rm D}^\alpha)u=|u|^2u,\\
    u(0)=u_{\rm in},
\end{cases}
\qquad (t,x)\in \R\times \T^2,
\end{equation}
where $\T^2=(\R/2\pi\Z)^2$ and $\alpha\in(1,2]$. The operator ${\rm D}^\alpha$ is defined by $\F_x({\rm D}^\alpha f)(k)=|k|^\alpha \F_x f(k)$ for $k\in\Z^2$.
Here and below, $\F_x$ denotes the spatial Fourier transform on the torus. 
The equation \eqref{fNLS} is Hamiltonian and may be viewed as a weakly dispersive variant of the standard cubic NLS. The parameter $\alpha$ measures the strength of dispersion: the larger $\alpha$ is, the stronger the dispersive effect becomes. 
When $\alpha= 2$, fNLS  \eqref{fNLS} reduces to the well-known cubic nonlinear Schr\"odinger equation (NLS), which is a canonical
model for the envelope dynamics of weakly nonlinear dispersive wave propagation; see \cite{SS99} for more details. 
When $\alpha\neq 2$, fNLS \eqref{fNLS} arises in fractional quantum mechanics {\cite{FQM}}. 
Fractional quantum mechanics is a generalization of quantum mechanics by replacing the Brownian motions with L\'evy processes \cite{Laskin00}. 
Moreover, it is shown in {\cite{KLS2012limit}} that fNLS \eqref{fNLS} describes the continuum limit of lattice interactions when $\alpha\in (1, 2]$. Specifically, the case when $1 < \alpha < 2$ corresponds to the long-range lattice interactions, while the case when $\alpha= 2$ corresponds to the short-range or quick-decaying interactions. 

In this paper, we focus on the fractional nonlinear Schr\"odinger equation
\eqref{fNLS} with \(\alpha\in(1,2)\), for which the dispersion is weaker than
that of the standard cubic nonlinear Schr\"odinger equation, corresponding to
the endpoint case \(\alpha=2\) in \eqref{fNLS}. We are particularly interested
in probabilistic well-posedness on the support of the associated Gibbs measure,
as well as the invariance of this Gibbs measure under the resulting dynamics.

\subsubsection{Deterministic well-posedness}
For smooth solutions, \eqref{fNLS} formally conserves both the Hamiltonian
\begin{equation}
    \label{Hamiltonian}
    H(u)=\frac12\int_{\T^2}\big| {\rm D}^{\frac{\alpha}{2}}u\big|^2dx
    +\frac14\int_{\T^2}|u|^4dx
\end{equation}
and the mass
\begin{equation}
    \label{Mass}
    M(u)=\int_{\T^2}|u|^2dx.
\end{equation}
These conservation laws suggest global control of solutions to \eqref{fNLS}, if exists, at the energy regularity level $H^{\frac{\alpha}{2}}$ and $L^2$, which are relatively low when $\alpha$ gets small. 

At lower regularity, however, the well-posedness theory of \eqref{fNLS} is much more delicate. 
The scaling symmetry
$u(t,x)\mapsto u_\lambda(t,x):=\lambda^{\frac{\alpha}{2}}u(\lambda^\alpha t,\lambda x)$
leaves \eqref{fNLS} invariant on $\R^2$ and gives the scaling critical Sobolev index $s_c=\frac{2-\alpha}{2}$.
Thus, deterministic ill-posedness is expected below the scaling threshold, while
deterministic well-posedness above, or near, this threshold depends on the
availability of suitable dispersive estimates. On the torus \(\mathbb T^2\),
although there is no exact scaling symmetry, the scaling heuristic remains a
useful guide.
However, due to the difficulty of obtaining sufficiently strong Strichartz
estimates in the fractional setting, the well-posedness theory for \eqref{fNLS} on
\(\mathbb T^2\) remains far from satisfactory. 
In a recent work \cite{Wang2026fNLSsobolev},
by establishing new Strichartz estimates,
the author proved deterministic local
well-posedness for \eqref{fNLS} with initial data in \(H^s(\mathbb T^2)\), with $s > 1-\frac{\alpha}{4}$ for $1<\alpha<\frac43$, and $s > \frac{\alpha}{2}$ for $ \frac43\le \alpha<2$.
This result improves on the elementary energy/continuity threshold \(H^{1+}(\mathbb T^2)\) in the weakly dispersive regime but is still far from the threshold predicted by the scaling heuristic.

\subsubsection{Probabilistic well-posedness}
The main purpose of this work is to study \eqref{fNLS} from a probabilistic
point of view, with particular emphasis on almost sure global well-posedness
and invariance of the associated Gibbs measure. We defer the precise formulation
of the problem and the construction of the relevant objects to
Subsection~\ref{SUB:setup}.

The study of invariant Gibbs dynamics for nonlinear dispersive equations has attracted considerable attention in recent years.
This line of
research was initiated by the pioneering works of Lebowitz-Rose-Speer
\cite{LRS}, and Bourgain \cite{B94,B96}, which established the foundational
framework for the construction of Gibbs measures and the globalization argument
in probabilistic settings. 
Following these classical paradigms, a substantial
body of work has developed both the structural understanding of such measures
and the construction of their associated dynamics  \cite{B97,BBNLS2D,BBNLS3D,BBNLW3D,DNY2022,DNY2024,OhPL21,Oh2015APA,OTW2020,Sun2019GibbsMD,Tzvetkov2006InvariantMF,Tzvetkov2008InvariantMF,Wang2021,Wang2024}. 
We do not attempt to give a comprehensive account of the literature here.

Since Bourgain's re-centering method \cite{B94,B96}, several powerful strategies have emerged. Prominent examples include the theory of regularity structures by Hairer \cite{Hairer13,Hairer14,Hairer14reg,Hairer15}, the para-controlled calculus by Gubinelli-Imkeller-Perkowski \cite{GIP15,GIP16}, and the Wilsonian renormalization-group method of
Kupiainen \cite{Kup16}. These approaches have achieved remarkable success
for parabolic singular SPDEs, but do not directly transfer to dispersive
problems, where the available smoothing mechanisms are fundamentally
different. In a recent breakthrough, Deng-Nahmod-Yue \cite{DNY2024}
introduced the random averaging operator framework and resolved almost sure
global well-posedness and uniqueness for the defocusing NLS on
\(\mathbb T^2\) with arbitrary polynomial nonlinearities. This framework
was further developed into the random tensor theory \cite{DNY2022}, which
captures more intricate multilinear structures and provides a refined
mechanism for the propagation of randomness along the dynamics.


For the fractional model \eqref{fNLS}, motivated in part by its role in
fractional quantum mechanics \cite{SS99,Laskin00,KLS2012limit}, invariant
measure problems have attracted increasing attention in recent years. In
particular, Oh-Tzvetkov-Wang \cite{OTW2020} proved invariance of the white
noise measure for the fourth-order NLS. The Gibbs measure for fractional NLS
has been constructed in \cite{LO2023fwave,LTW23,Wang2021,Sun2021}, and in the one-dimensional
weakly dispersive regime \(\alpha\in(1,2)\), the invariance of the Gibbs
measure is now relatively well understood; see, for example,
\cite{Sun2021,Wang2024} and the references therein.
In higher dimensions, however, the corresponding dynamical theory and the
invariance of the Gibbs measure remain poorly understood. To the best of our
knowledge, in the weakly dispersive regime \(\alpha\in(1,2)\), no higher
dimensional invariance result is currently available. This is due to several
additional difficulties, which we discuss in Subsection~\ref{SUB:ideas}. The
main purpose of this work is to fill this gap, at least for some range of \(\alpha\).

\subsection{Setup and main results}
\label{SUB:setup}

In this subsection, we set up the main objects associated with the fractional
NLS \eqref{fNLS}, including the Gibbs measure, the Wick ordering, and the gauge transform. 
We then state our main result, namely the invariance of the Gibbs
measure under the dynamics of the Wick ordered equation \eqref{fNLS}.

\subsubsection{Gibbs measure}
The Gibbs measure problem starts from the finite-dimensional Hamiltonian analogy. For a Hamiltonian ODE, Liouville's theorem and the conservation of the Hamiltonian imply the invariance of measures of the form $Z^{-1}e^{-H}dpdq$. Applying this principle formally to the Hamiltonian PDE \eqref{fNLS}, as in the pioneering works \cite{LRS,B94,B96}, one is led to the expression
\begin{equation}\label{formalGibbs}
d\rho_\alpha=Z^{-1}e^{-H(u)-\frac12M(u)}du,
\end{equation}
where $H$ and $M$ are the Hamiltonian and mass defined in \eqref{Hamiltonian} and \eqref{Mass}, respectively.
The additional mass term is harmless for the dynamics, since $M(u)$ is conserved, and it replaces the massless Gaussian free field by the massive Gaussian free field associated with the quadratic energy; see also \cite{Oh2015APA} for this standard massive formulation in the two-dimensional NLS setting.
This motivates one to define the Gibbs measure as an absolutely continuous probability measure with respect to the massive Gaussian free field
\begin{equation}\label{massiveGFFform}
    ``d\mu_\alpha= Z_\alpha^{-1}
    e^{-\frac12\int_{\T^2} |\jbb{\nabla}^{\frac{\alpha}{2}}u|^2dx}du\text{''},
\end{equation}
where $ \jbb{\cdot}=(1+|\cdot|^{\alpha})^{\frac1\alpha}$.
In view of the conservation of mass, we still expect the resulting Gibbs measure to be invariant once it is properly constructed.
Equivalently, this Gaussian field is represented by the random Fourier series
\begin{equation}\label{GFFintro}
    u^\omega (x) =\sum_{k\in\Z^2}\frac{g_k(\omega)}{\jbb{k}^{\frac{\alpha}{2}}}e^{ik\cdot x}, \end{equation}
where $(g_k)_{k\in\Z^2}$ are independent standard complex Gaussian random variables.
Namely,
\[
    \mu_\alpha=\mathbb P\circ (u^{\omega})^{-1}.
\]
The random function in \eqref{GFFintro} belongs to $H^{\frac{\alpha-2}{2}-}(\T^2)$ and not to $H^{\frac{\alpha-2}{2}}(\T^2)$ almost surely. In particular, throughout the weakly dispersive range $1<\alpha<2$, the Gaussian field is supported below $L^2(\T^2)$ as well as the critical Sobolev space $H^{\frac{2-\alpha}2} (\T^2)$.

The next step is to make sense of the quartic part of the Gibbs density. Let us first recall the one-dimensional situation. In dimension one, the analogue of \eqref{GFFintro} has finite $L^2$ mass almost surely when $\alpha>1$, and the defocusing Gibbs measure can be constructed as a weighted Gaussian measure without Wick renormalization. The corresponding fractional NLS Gibbs dynamics in the full range $\alpha\in(1,2)$ were studied in \cite{Sun2019GibbsMD,Sun2021,Wang2024}.
The situation in dimension two is entirely different. If $\Pi_{ N}$ denotes the Fourier projection to $\{\jb{k}\leq N\}$ with $\jb{k} = (1+|k|^2)^{\frac12}$, and $u_N=\Pi_{ N} u$, then
\begin{equation}\label{L2expectation}
 \sigma_N : =    \frac{1}{(2\pi)^2}\mathbb{E}\|u_N\|_{L^2(\T^2)}^2
    =\sum_{\jb{k}\leq N}\frac{1}{\jbb{k}^{\alpha}}
    \sim N^{2-\alpha}
\end{equation}
for $1<\alpha<2$. Hence the $L^2$ mass diverges as $N\to\infty$. Consequently,
\[
    \int_{\T^2}|u^\omega (x) |^4dx=\infty
\]
almost surely in the unrenormalized sense, and one cannot directly construct a probability measure from the formal expression \eqref{formalGibbs}.
Thus, as in the two-dimensional NLS setting \cite{B96,Oh2015APA}, one is required to perform a Wick renormalization on the nonlinear part of the Hamiltonian.

\subsubsection{Wick ordering}

The preceding divergence is precisely the reason for Wick ordering. The role of the renormalization is to subtract the divergent self-contractions created by the Gaussian field before passing to the limit. This renormalization is classical in Euclidean quantum field theory and stochastic PDE; see \cite{Simon.B,GlimmJaffe,PLWick,Oh2015APA}. 

Given $N \in \N$, we define the Wick ordered quartic power $:|u_N|^4:$ by
\begin{align}
\label{Wick_q}
:|u_N|^4\!:  =  |u_N|^4-4\sigma_N|u_N|^2+2\sigma_N^2,
\end{align}
where $\sigma_N$ is given by \eqref{L2expectation}.
Let 
\begin{align}\label{Wick_poten}
G_N (u) = \frac14 \int_{\T^2} :|u_N|^4\!:  dx.
\end{align}
The usual Wiener-chaos computation (see, for instance,
\cite[Proposition 1.1]{Oh2015APA}) shows that $\{G_N (u)\}_{N \in \N}$ is Cauchy in $L^p (\mu_\alpha)$ for every $1 \leq p < \infty$. Therefore, we can define the limit $G(u)$ as
\[
G(u) =  \frac14 \int_{\T^2} :|u|^4\!:  dx =  \lim_{N \to \infty} G_N(u) = \frac14  \lim_{N \to \infty} \int_{\T^2} :|u_N|^4\!:  dx,
\]
in $L^p(\mu_\alpha)$ for any finite $p \ge 2$.
This identifies the formal Wick ordered Hamiltonian:
\begin{align}
\label{Wick_Hamil}
H_{\rm Wick} (u) = \frac12\int_{\T^2} |{\rm D}^{\frac{\alpha}{2}}u|^2dx +  \frac14 \int_{\T^2} :|u|^4\!:  dx.
\end{align}
The quartic term is infinite on the support of $\mu_\alpha$, so
\eqref{Wick_Hamil} is not itself a finite random variable. The rigorous
Gibbs measure is instead defined below by taking the renormalized potential
as a density with respect to the Gaussian measure, whose covariance already
encodes the quadratic part.

Let us then define 
\begin{align}\label{RN}
R_N (u) = e^{-G_N(u)} = e^{-\frac14 \int_{\T^2} :|u_N|^4\!:  dx}
\end{align}

\begin{prop}[Theorem 1.2 in \cite{LO2023fwave}]
\label{PROP:Gibbs}
    Let $\alpha\in(\frac{3}{2},2)$ and $R_N(u)$ be as in \eqref{RN}. 
    Then $R_N(u)$ converges to a limit $R(u)$ in $L^p(\mu_\alpha)$ for every $1\leq p<\infty$. 
\end{prop}

In the following, we denote the limit in Proposition \ref{PROP:Gibbs} as
\begin{align}
    \label{G}
    R(u) = e^{-\frac14 \int_{\T^2} :|u|^4\!:  dx}.
\end{align}
Then, as a consequence of Proposition \ref{PROP:Gibbs}, we can rigorously define the Gibbs measure $d \rho_\alpha$ associated with the Wick ordered fractional Hamiltonian \eqref{Wick_Hamil} by
\begin{align}
\label{Gibbs_rig}
d \rho_\alpha = Z^{-1} R(u) d\mu_\alpha = Z^{-1} e^{-\frac14 \int_{\T^2} :|u|^4\!:  dx} d\mu_\alpha. 
\end{align}
Furthermore, $d\rho_{\alpha}$ is absolutely continuous with respect to $d\mu_{\alpha}$.

\begin{rem}
In \cite{LO2023fwave}, Proposition~\ref{PROP:Gibbs} is proved for the
real-valued fractional Gibbs measure. We remark that the complex-valued version
can be constructed in the same way.
Indeed, recall that the Wick ordered quartic power can be written as
\[
\begin{aligned}
:|u_N|^4:
&=
:\big( (\operatorname{Re} u_N)^2+(\operatorname{Im} u_N)^2 \big)^2:  \\
&=
:(\operatorname{Re} u_N)^4:
+
2:(\operatorname{Re} u_N)^2:\,:(\operatorname{Im} u_N)^2:
+
:(\operatorname{Im} u_N)^4: .
\end{aligned}
\]
The convergence of the left-hand side can be justified rigorously by the same
argument as in \cite{LO2023fwave}; the same applies to $R_N$ in
\eqref{RN}. This yields the construction of the complex-valued Gibbs measure
$\rho_\alpha$ in \eqref{Gibbs_rig}, and hence proves Proposition~\ref{PROP:Gibbs}.

We also point out that, although the measure can be constructed in the range
$\alpha>\frac32$, the corresponding dynamical problem is solved here only in
the more restrictive range $\alpha>\frac{29}{15}$. We refer to
Theorem~\ref{LWP} and Theorem \ref{GWP} for the precise statement.
\end{rem}

To define the Wick ordered dynamics associated with \eqref{fNLS}, we first
introduce the truncated Wick ordered Hamiltonian
\begin{equation}\label{Hamil_wick^N}
    H_{\rm Wick}^N(u)
    =
    \frac12\int_{\mathbb T^2}|{\rm D}^{\frac{\alpha}{2}}u|^2\,dx
    +
    \frac14\int_{\mathbb T^2}:|u_N|^4:\,dx,
\end{equation}
where $u_N=\Pi_Nu$. The (nonlinear part of the) corresponding truncated Wick ordered dynamics is
\begin{equation}\label{truncfNLS}
\begin{cases}
 (i\partial_t-{\rm D}^\alpha)u_N
    =
    \Pi_N\bigl(:|u_N|^2u_N:\bigr),\\
    u_N(0)=\Pi_Nu^\omega,
\end{cases}
\end{equation}
where the truncated Wick ordered cubic nonlinearity is given by
\begin{equation}\label{Wick_cubic_N}
    \Pi_N\bigl(:|u_N|^2u_N:\bigr)
    =
    \Pi_N\bigl(|u_N|^2u_N-2\sigma_Nu_N\bigr).
\end{equation}
Here $\sigma_N$ is defined in \eqref{L2expectation}.

It can be shown (see \cite[Proposition 1.3]{Oh2015APA} for $\alpha = 2$ case)  that for general $\alpha $, the sequence
    $\Pi_N\bigl(:|u_N|^2u_N:\bigr)$
is Cauchy in $L^p(\mu_\alpha;H^s(\mathbb T^2))$, for every finite $p\ge1$
and every $s<\frac{3\alpha}2-3$. Hence it admits a limit, denoted by $:|u|^2u:$.
Passing formally to the limit $N\to\infty$, the Wick ordered version of the fractional NLS \eqref{fNLS} takes the form
\begin{equation}\label{WfNLS}
\begin{cases}
    (i\partial_t-{\rm D}^\alpha)u
    =  \ :|u|^2u: \ \ ,\\
    u(0)=u^\omega.
\end{cases}  
\end{equation}

\begin{rem}\label{truncatedGibbs}
Consider the truncated Gibbs measure
\begin{equation}\label{truncateGibbsmeas}
    d\rho_{\alpha,N}^{\circ}
    =
    Z_{\alpha,N}^{-1}
    \exp\left(
        -\frac14\int_{\mathbb T^2}:|u_N|^4:\,dx
    \right)d\mu_{\alpha,N},
\end{equation}
associated with the truncated Wick ordered Hamiltonian
\eqref{Hamil_wick^N}, where $\mu_{\alpha,N} = \mathbb{P}\circ \left(\Pi_N u^{\omega}\right)^{-1}$ is the truncated Gaussian measure. Then $d\rho_{\alpha,N}^{\circ}$ is invariant under the finite-dimensional system \eqref{truncfNLS}. Furthermore, if we define $ d\rho_{\alpha,N} =  d\rho_{\alpha,N}^{\circ}\times d\mu_{\alpha, N}^{\perp}$ with $\mu_{\alpha,N}^{\perp} = \mathbb{P}\circ \left(\Pi_N^{\perp} u^{\omega}\right)^{-1}$, it is invariant under the product flow
\[
    u^N(t)=u_N(t)+u_N^\perp(t),
\]
where $u_N(t)$ solves the finite-dimensional Wick ordered equation
\eqref{truncfNLS}, while $u_N^\perp(t)$ evolves according to the linear flow
\[
(i\partial_t-{\rm D}^\alpha)u_N^\perp=0.
\]
Moreover, as $N\to\infty$, the measures $\rho_{\alpha,N}$ converge weakly
to the Gibbs measure $\rho_\alpha$.
\end{rem}

\subsubsection{Gauge transform}

We next introduce a gauge transform which removes the remaining resonant
mass-type contribution after Wick ordering. Recall that the Wick ordering nonlinearity \eqref{Wick_cubic_N}
subtracts the Gaussian contraction $\sigma_N$. However, in the
Fourier expansion of the cubic term $|u_N|^2 u_N$, there remains a diagonal contribution
involving the actual truncated mass. As in the one-dimensional fractional
problem \cite{Wang2024} and the two-dimensional NLS setting
\cite{DNY2024}, this contribution can be removed by a phase rotation.

Let us first explain this process at the truncated level. 
For the truncated system \eqref{truncfNLS}, define
\[
    v_N(t)
    =
    e^{2it(m_N-\sigma_N)}u_N(t),
\]
where
\begin{equation}
\label{eq:truncated-mass}
    m_N
    =
    \frac{1}{(2\pi)^2}\|u_N\|_{L^2(\mathbb T^2)}^2
    =
    \sum_{\langle k\rangle\le N}|(u_N)_k|^2
    =
    \sum_{\langle k\rangle\le N}|(u_{\rm in})_k|^2
\end{equation}
is conserved under the flow of \eqref{truncfNLS}. Here $\sigma_N$ is given in \eqref{L2expectation}, which denotes the
Wick contraction appearing in the definition of the Wick ordered nonlinearity.  Then $v_N$ solves
\begin{equation}
\label{truncfNLSgauge}
\begin{cases}
(i\partial_t- {\rm D}^{\alpha})v_N
=
\Pi_N\bigl(|v_N|^2u_N-2m_Nv_N\bigr),
\\
v_N(0)=\Pi_N u^\omega .
\end{cases}
\end{equation}
Note that the explicit Wick constant $\sigma_N$ no longer appears in the gauged
equation, although Wick ordering remains essential in the construction of the
Gibbs measure and in identifying the correct renormalized dynamics.

By \cite[Proposition~2.4]{DNY2024}, the truncated Gibbs measure
$d\rho_{\alpha,N}^{\circ}$ in \eqref{truncateGibbsmeas} is also invariant under the
flow of \eqref{truncfNLSgauge}. The advantage of \eqref{truncfNLSgauge} is that
the gauged nonlinearity admits a multilinear decomposition in which the most
singular diagonal pairings have been removed.
In particular,
\[
\begin{split}
\mathcal F_x & \bigg(
|v|^2v-\frac{2}{(2\pi)^2}\|v\|_{L^2(\mathbb T^2)}^2v
\bigg)(k) =
\sum_{\substack{k=k_1-k_2+k_3\\ k_2\neq k_1,k_3}}
v_{k_1}\overline{v_{k_2}}v_{k_3}-|v_k|^2v_k .
\end{split}
\]
Thus the nonlinearity in \eqref{truncfNLSgauge} can be written as
\begin{equation}\label{multilinearW^3}
 |v_N|^2v_N-2m_Nv_N
 =\mathcal N_3(v_N,v_N,v_N)+\mathcal Q_3(v_N,v_N,v_N),
\end{equation}
where
\begin{equation}\label{trilinearforms}
\begin{aligned}
 \big(\mathcal N_3(u,v,w)\big)_k
 &=
 \sum_{\substack{k=k_1-k_2+k_3\\ k_2\neq k_1,k_3}}
 u_{k_1}\overline{v_{k_2}}w_{k_3},\\
 \big(\mathcal Q_3(u,v,w)\big)_k
 &=-u_k\overline{v_k}w_k.
\end{aligned}
\end{equation}
We write $\mathcal N_3(v)=\mathcal N_3(v,v,v)$ and similarly for
$\mathcal Q_3$. This is the form used throughout the proof.

\subsubsection{Main results}

We are now ready to state our main results. We first address almost sure local well-posedness for the Wick ordered fractional NLS \eqref{WfNLS} with Gaussian
random initial data given by \eqref{GFFintro}.

\begin{thm}\label{LWP}
    Let $\alpha\in(\frac{29}{15},2)$ and fix sufficiently small
    $0<\varepsilon<\min\{15\alpha-29,2-\alpha\}$. There exist
    $\tau_0=\tau_0(\alpha,\varepsilon)>0$ and
    $\theta=\theta(\alpha,\varepsilon)>0$ such that, for every
    $0<\tau\leq\tau_0$, there is a measurable set
    $\Omega_{\tau,\varepsilon}\subset\Omega$ satisfying
    $\mathbb P(\Omega_{\tau,\varepsilon}^c)
    \leq C_{\alpha,\varepsilon,\theta}e^{-c_{\alpha,\varepsilon}
    \tau^{-\theta}}$.
    For every $\omega\in\Omega_{\tau,\varepsilon}$, the solutions $v_N$
    of \eqref{truncfNLSgauge} converge in
    $C([-\tau,\tau];H^{\frac{\alpha-2}{2}-\varepsilon}(\mathbb T^2))$
    to a distributional solution $v$ of
    \begin{equation}\label{Eqn:gaugedlimit}
       (i\partial_t-{\rm D}^{\alpha})v
       =\mathcal N_3(v)+\mathcal Q_3(v)
    \end{equation}
    with initial data $u^\omega$. Moreover, the solutions $u_N$ of the Wick ordered equation
    \eqref{truncfNLS} converge in the same
    space to
    \[
        u(t)=e^{-2itX(u^{\omega})}v(t),
        \qquad
        X(u^\omega)=\lim_{N\to\infty}
        \left(\frac{1}{(2\pi)^2}\|\Pi_Nu^\omega\|_{L^2}^2-\sigma_N\right),
    \]
    where the displayed limit defining $X(u^\omega)$ exists. 
    $u$ solves \eqref{WfNLS} with $u(0)=u^\omega$ in the distributional sense. 

\end{thm}

\begin{rem}
    By slightly shrinking the admissible set, which keeps the exceptional set exponentially small, the gauged limit $v$ with initial data $u^\omega$ belongs to and is unique in the structured solution class $\mathcal{U}_{u^\omega}([-\tau,\tau])$ defined in
    Definition~\ref{structured-class}, and $u$ is unique in its inverse image under the gauge transformation above. See more discussion in Remark \ref{unique2}.
\end{rem}

    We call Theorem \ref{LWP} almost sure local well-posedness in the
    structured class. Taking a sequence $\tau_j\downarrow0$ sufficiently
    rapidly and applying the Borel-Cantelli lemma shows that the solution
    exists on a nontrivial time interval $[-\tau_\omega,\tau_\omega]$ and is unique in $\mathcal{U}_{u^\omega}([-\tau_\omega,\tau_\omega])$ for
    $\mathbb P$-almost every $\omega$, and hence for
    $\mu_\alpha$- and $\rho_\alpha$-almost every initial datum.

\smallskip

\begin{rem}
    In fact, the local result (Theorem \ref{LWP}) also holds for the focusing case. However, a standard invariant argument to globalize the local solution needs the construction of the corresponding Gibbs measure. So we only consider the defocusing case here.
\end{rem}

\begin{rem}
In \cite{MX22}, probabilistic well-posedness for the fractional NLS on
\(\mathbb T^d\), \(d\ge1\), was established for initial data in
\(H^s(\mathbb T^d)\), with either \(s>\frac d2\), or
\(s\in[\frac{\alpha}{2},1]\) when \(d\le3\). These results, however, all concern
data of positive regularity. In contrast, in our setting the Gibbsian initial
data \eqref{GFFintro} almost surely belongs to
$H^{\frac{\alpha-2}{2}-}(\mathbb T^2)$,
which is of negative regularity for every \(\alpha<2\), and becomes increasingly
rough as \(\alpha\) moves away from \(2\). While such a difficulty already arises
for the classical NLS in dimensions \(d\ge3\), it appears in two dimensions for
the fractional NLS due to the weaker dispersion.
\end{rem}

\begin{rem}
We do not expect the range of \(\alpha\) in Theorem~\ref{LWP} to be optimal.
The main purpose of this work is instead to develop a flexible framework for
fractional NLS in higher dimensions, including fractional counting estimates,
fractional tensor localization, and fractional \(\Gamma\)-conditions. We
therefore do not pursue the optimal range of \(\alpha\) in this paper.
The main difficulty in extending the range of \(\alpha\) comes from the
fractional counting estimates; see Subsection~\ref{SUB:ideas} for further
discussion. In the spirit of the probabilistic scaling heuristic argument from
\cite{DNY2022}, the counting estimates in Section~\ref{count} yield the
probabilistic regularity threshold
$s_{\mathrm{pr}}
=
\frac16-\frac{\alpha}{4}$.
This suggests that, in principle, Theorem~\ref{LWP} may be extendable to the
range
$\alpha>\frac{14}{9}$,
for instance by adapting the random tensor argument of \cite{DNY2022}. Further
improvement of the range of \(\alpha\) would likely require sharper fractional
counting estimates. For this reason, we leave the question of the optimal range
open.
\end{rem}

\begin{rem}
In the one-dimensional setting, the full weakly dispersive range
\(\alpha\in(1,2)\) for the fractional NLS \eqref{fNLS} has been settled \cite{Sun2019GibbsMD,Sun2021,Wang2024}. 
In that case, the
Gibbsian initial data have positive Sobolev regularity, and hence Wick ordering
is not needed. Moreover, the relevant Strichartz estimates and fractional
counting estimates are better understood; see, for example, \cite{Sun2021}.
By contrast, in the two-dimensional problem considered in this paper, the
initial data are substantially more singular. In addition, both the relevant
Strichartz estimates and the fractional counting estimates are much less
developed. These issues constitute some of the main new difficulties in the
two-dimensional fractional setting.
\end{rem}

\begin{rem}
Compared with the standard NLS on \(\mathbb T^2\), where Bourgain
\cite{B96} reduced the relevant counting problems to counting intersections
between lattice points and algebraic curves, together with divisor-counting
arguments, the presence of a non-integer dispersion exponent \(\alpha<2\) makes
the counting problems in our setting considerably more delicate. In particular,
one is led to count lattice points in thickened regions rather than on exact
algebraic curves.

This feature is reminiscent of the irrational-torus setting; see
\cite{FOSW2021}. However, in \cite{FOSW2021}, the dispersion relation remains
quadratic, and in certain cases algebraic divisor-counting arguments are still
available. In the present fractional setting, such algebraic structure is no
longer available because the resonance function involves non-integer powers.
As a result, classical number-theoretic counting arguments are not directly
applicable. We therefore develop a new framework based on the geometric
structure of the associated thickened resonance domains.
We refer to Subsection \ref{SUB:ideas} for further discussion.


\end{rem}

\smallskip

\gb

Based on the almost sure local well-posedness, we can extend the solution to arbitrary long time and get the invariance of the Gibbs measure.

\begin{thm}\label{GWP}
    Let $\alpha\in(\frac{29}{15},2)$ and let $\rho_{\alpha}$ be the Gibbs
    measure defined in \eqref{Gibbs_rig}. There exists a rotation-invariant
    Borel set
    $\Sigma\subset\bigcap_{\varepsilon>0}
    H^{\frac{\alpha-2}{2}-\varepsilon}(\mathbb T^2)$ with
    $\rho_\alpha(\Sigma^c)=0$ such that the following holds. For every
    $u_{\mathrm{in}}\in\Sigma$,
    let $u_N$ solve \eqref{truncfNLS}. Then, for every $T>0$ and every
    $\varepsilon>0$, $u_N$ converges in
    $C([-T,T];H^{\frac{\alpha-2}{2}-\varepsilon}(\mathbb T^2))$ to a unique
    limit $u$, which solves \eqref{WfNLS} in the distributional sense. The flow map
    $\Phi_tu_{\mathrm{in}}:=u(t)$ is a Borel map from $\Sigma$ to itself,
    satisfying the usual group property,
    $\Phi_0=\mathrm{Id}$ and
    $\Phi_{t_2}\circ\Phi_{t_1}=\Phi_{t_1+t_2}$. Moreover, the Gibbs measure $\rho_\alpha$ is invariant in the sense that
    $(\Phi_t)_\#\rho_\alpha=\rho_\alpha$ for every $t\in\mathbb R$.
\end{thm}

  The globalization argument originates from Bourgain \cite{B94,B96}. We also
 refer to \cite[Section 6]{DNY2024} and \cite[Section 3]{Bring23}. The details
 needed in the present fractional setting are given in
 Section~\ref{MainProof}.

\begin{rem}\label{unique2}
The limit $u$ is unique in the propagated structured class in Section \ref{MainProof}. The uniqueness here is not an unconditional uniqueness statement among all distributional solutions. It is uniqueness in the admissible structured class, which encodes the high-frequency tail stability inherited from the canonical truncations, aligns with the uniqueness notion used in \cite[Remark 1.6(3)]{DNY2024}. In particular, any limit of solutions to canonical truncated Wick-ordered fNLS with initial data (truncated in the same way) converges to the solution in the propagated structured class.
\end{rem}

\begin{rem}
     There have been many results for the fractional nonlinear wave equation (fNLW) for the two-dimensional case. In \cite{weakuniversality}, they study the so-called weak universality result and obtain the convergence of both the measure and dynamics. \cite{LO2023fwave} consider the fNLW Gibbs dynamics with a general polynomial nonlinearity. Compared to these cases, the main challenge for fNLS is that our Duhamel formula has no smoothing effect. However, by carefully discussing the fNLS dynamics with higher-order nonlinearity, we may also consider the weak universality result under suitable assumptions.
\end{rem}

\subsection{Main difficulties and ideas}
\label{SUB:ideas}

We now discuss the main points of the proof. The two basic difficulties are
the roughness of the Gibbsian initial data and the loss of the quadratic
arithmetic structure when \(\alpha<2\). The latter affects the argument in
three related places: the lattice counting estimates, the localization of
random tensor bounds, and the \(\Gamma\)-condition used in the high--low
analysis.

\subsubsection{Singular data}

The Gaussian field \(u^\omega\) in \eqref{GFFintro} belongs almost surely to
\(H^{\frac{\alpha-2}{2}-}(\mathbb T^2)\). Thus, for every \(\alpha<2\), the
initial data lie below \(L^2\) and well below the regularity covered by the
available deterministic theory; see \cite{Wang2026fNLSsobolev}. For the
standard equation \(\alpha=2\), the Gibbs measure is supported only
\(\varepsilon\)-below the critical space \(L^2\); see \cite{B96}. Here the gap
becomes larger as \(\alpha\) decreases.

We use the random averaging operator method introduced in \cite{DNY2024}; see
also \cite{Wang2024} for the one-dimensional fractional NLS. The worst
high-low-low term is regarded as a random operator, built from the low
frequencies, acting on the high-frequency linear solution. This leads to the
schematic ansatz
\begin{equation}
    u=u_{\rm lin}+\mathcal P(u_{\rm lin})+\mathcal Z,
\end{equation}
\noi
where \(u_{\rm lin}=e^{-it{\rm D}^\alpha}u_0\), \(\mathcal P\) denotes the
random averaging term, and \(\mathcal Z\) is the remainder. The dyadic version
of this decomposition and its a priori bounds are given in
Section~\ref{structuresol}. The term \(\mathcal P(u_{\rm lin})\) is controlled
through operator and Hilbert--Schmidt norms, whereas \(\mathcal Z\) is treated
by contraction. We refer to \cite{DNY2022,DNY2024} for the general framework.

\subsubsection{Fractional counting}

The multilinear estimates require lattice counts adapted to the fractional
phase. After localizing the frequencies and the modulation, we are led to
subsets of the set $S$ defined in \eqref{Eqn:S} in
Subsection~\ref{count}. The key restriction is the localized phase condition
\[
 |k_1|^\alpha-|k_2|^\alpha+|k_3|^\alpha-|k|^\alpha
 =m+\mathcal O(1),
\]
where \(m\) is the localized modulation parameter.

For \(\alpha=2\), Bourgain's argument \cite{B96} uses the arithmetic structure
of the quadratic resonance. The modulation is integer-valued, and identities
such as
\[
    |k_1|^2-|k_2|^2=(k_1-k_2)\cdot(k_1+k_2)
\]
lead to divisor bounds and lattice-point estimates on circles. For
\(1<\alpha<2\), the modulation is no longer connected to integer arithmetic
and there is no corresponding factorization. The same distinction arises when
one compares our problem with the irrational-torus setting of
\cite{FOSW2021}: although part of the lattice structure is lost there, the
dispersion is still quadratic. In our case the counting problem is genuinely
geometric. The available one-dimensional fractional estimates, used directly,
are not strong enough at the Gibbsian regularity.

After fixing two frequencies and using momentum conservation, the remaining
problem becomes a geometric lattice-point count in a unit-thickened level set
of one of the functions
\[
    \phi_{b,\pm}(x)=|x|^\alpha\pm|x-b|^\alpha.
\]
The shape of these level sets changes with the frequency scales, the vector
\(b\), and the height of the level. Their thickness and curvature also behave
differently near the singular points \(0\) and \(b\). For
\(\phi_{b,-}\), we use coordinate slices together with a lower bound on one
directional derivative; this gives Lemma~\ref{2vec1}. The level sets of
\(\phi_{b,+}\) are convex away from their minimum. After separating the region
near the minimum, we apply the convex-arc estimate in
Lemma~\ref{LEM:strip} and obtain Lemma~\ref{2vec2}. The remaining counting
bounds are built from these two arguments.

\subsubsection{Localized random tensors}

We also use the random tensor estimates of \cite{DNY2022}, which reduce the
probabilistic multilinear estimates to deterministic tensor bounds. These
bounds are then estimated by the lattice counts above.

The ordinary random tensor estimate
\cite[Proposition~4.14]{DNY2022} carries a small loss at the largest frequency.
This loss cannot be afforded in some of the high-low-low interactions. For quadratic
dispersion, the localized estimate \cite[Proposition~4.15]{DNY2022} replaces
the largest-frequency scale in the tensor windows by a lower-frequency scale.
This relies on the identity
\[
    |b_0+f|^2-|c_0+f|^2
    =2f\cdot(b_0-c_0)+|b_0|^2-|c_0|^2.
\]
After the other variables are localized, the dependence on \(f\) is therefore
linear.

For fractional dispersion, however,
\[
    |b_0+f|^\alpha-|c_0+f|^\alpha
\]
is nonlinear in \(f\). Its leading behavior varies with the scale and the
relative positions of the frequencies. To recover a localized tensor bound
with the same effective lower-frequency window,
we divide the domain according to these configurations. In the most delicate
high-low-low case, this amounts to classifying the possible types of
\[
    |x+f|^\alpha-|y+f|^\alpha,\qquad f\in\mathbb Z^2.
\]
Taylor expansion determines the leading term for each piece, and the type and
the corresponding region are counted together.
This is the content of Lemma~\ref{HLLtensor}.

\subsubsection{Fractional \(\Gamma\)-condition}

The last point concerns the \(\Gamma\)-condition. An exact frequency
projection may place the output and a largest input frequency on opposite
sides of a radial threshold. In the high--low analysis, this forces the
largest frequencies to lie in a thin annular region. For the quadratic problem in two dimensions,
this gives an arithmetic gain because \(|k|^2\) is integer-valued; see
\cite[Proposition~4.5]{DNY2024}. In the one-dimensional fractional problem,
the same role is played by the integer-valued quantity \(|k|\); see
\cite[Lemma~2.11]{Wang2024}.

In two dimensions, \(|k|\) is not integer-valued, while
\(|k|^\alpha\) is neither integer-valued nor quadratic. We pass instead from
the fractional annulus to a quadratic one. If \(|k|\sim N\) and the
fractional width is \(N_{\med}^{\alpha}\), then
\[
    \bigl||k|^2-\Gamma^2\bigr|
    \lesssim N_{\med}^{\alpha}N^{2-\alpha}.
\]
Since \(|k|^2\) is integer-valued, divisor counting becomes available again,
at the cost of \(N^{2-\alpha+}\). This loss contributes to the lower bound on
\(\alpha\).

These estimates give the local convergence of the truncated gauged solutions
for \(\alpha\in(\frac{29}{15},2)\), and hence Theorem~\ref{LWP}. The inverse
gauge transform gives the Wick ordered dynamics. Theorem~\ref{GWP} then follows
from invariance of the finite-dimensional Gibbs measures and the commutator
and stability estimates in Section~\ref{MainProof}.

\subsection{Notations and choice of parameters}\label{notations} 
In this section, we collect some notations and conventions that will be followed in the proof. For example, we always assume that $N$ or $M$ is a dyadic number. And we use $N_{\max}$, $N_{\med}$, $N_{\min}$ to denote the largest, second largest and smallest number among $N_1$, $N_2$, $N_3$. Let $k_i\in\Z^2$ satisfy  $\jb{k_i}\leq N_i$ or $\frac{N_i}{2}<\jb{k_i}\leq N_i$, $i=1,2,3$, we denote $k_{\max}$, $k_{\med}$, $k_{\min}$ to be the integer vectors that correspond to $N_{\max}$, $N_{\med}$, $N_{\min}$ respectively. We also define $a\wedge b = \min{\{a,b\}}$ and $a\vee b = \max{\{a,b\}}$ for any $a,b\in\R$.

Most of the time, we are working in the spatial Fourier space so we will simply write

\begin{equation}
    u_k = \left(\mathcal{F}_x u\right)(k) = \frac{1}{(2\pi)^2}\int_{\mathbb{T}^2}e^{-ik\cdot x}u(x)dx.
\end{equation}

As for the temporal Fourier transform, we write

\begin{equation}
    \widehat{u}(\lambda) = \left(\mathcal{F}_t u\right)(\lambda) = \frac{1}{2\pi}\int_{\mathbb{R}}e^{-i\lambda t}u(t)dt.
\end{equation}

In particular, $\widehat{u}_k(\lambda) = \left(\mathcal{F}_{t,x}u\right)(k,\lambda)$. Recall that $\jbb{k}:= (1 + | k |^{\alpha})^{\frac{1}{\alpha}}$, then $\jbb{k} \sim \jb{k}$ for fixed $\alpha$. $\ind_{E}$ is the indicator function of a set $E$. We will use smooth cutoff function $\eta(t)$ that equals 1 for $|t|\leq 1$ and equals 0 for $|t|\geq2$. For any Schwartz function $\chi$, we write $\chi_{\tau}(t) = \chi(\frac{t}{\tau})$ for any $0<\tau\ll1$. We also have the Duhamel formula of the equation \eqref{truncfNLSgauge}

\begin{equation}\label{Duhamel}
    v_N(t) = e^{-it{\rm D}^{\alpha}}\Pi_Nu^\omega
    +\mathcal{I}\Pi_N\mathcal{N}_3(v_N)(t)
    +\mathcal{I}\Pi_N\mathcal{Q}_3(v_N)(t),
    \quad \mathcal{I}v(t) = -i\int_{0}^{t}e^{-i(t-t'){\rm D}^{\alpha}}v(t')dt'.
\end{equation}

In the proof, we will fix $\frac{29}{15}<\alpha<2$, $0<\varepsilon\ll1$  in Theorem \ref{LWP} and Theorem \ref{GWP}. For such $\alpha,\varepsilon$, the parameters are chosen to satisfy:

\begin{equation}\label{parameters1}
    0<\kappa_0^{-\frac{1}{2}} \ll \delta \ll \gamma_0 \ll\gamma \ll \delta_0 \ll\varepsilon\ll \min\{15\alpha-29,2-\alpha\}
\end{equation}
\noi
and $\theta\ll \delta^{30}$, $b>b'>\frac{1}{2}$, $\kappa_0\gg1$ so that $\kappa_0^{-\frac{1}{2}}\sim \delta^3\sim b'-\frac{1}{2}\sim b-b'$. Theorem \ref{GWP} follows from a countable intersection over $\varepsilon$.

\noi

The rest of the paper is organized as follows. In Section \ref{preliminaries}, we gather some basic definitions and lemmas. In Section \ref{structuresol}, we introduce the random averaging operator and state our ansatz in detail. In Section \ref{prep}, we make some preparations for the proof of Proposition \ref{local2}, which is the key to our main results. To be precise, we give several counting estimates and tensor estimates in this section.  Then we prove Proposition \ref{local2} in Section \ref{esti} and finish the proof of our main results in Section \ref{MainProof}.

\subsection{Acknowledgment} 

Y.Wang was supported by the EPSRC Mathematical Sciences Small Grant
(grant no. UKRI1116). H. Yue was supported by the NSF grant of China
(No. 12301300). L. Zhao was supported by the National Natural Science
Foundation of China (No. 12271497 and No. 12341102).

\smallskip

\section{Preliminaries}\label{preliminaries}

 To clarify the independence of some random variables in the following proof, we first define the Borel $\sigma$-algebra.

\begin{df} \label{borel}
 For any dyadic $N$ , we denote by $\mathcal{B}_{\leq N}$  the $\sigma$-algebra generated by the random variables $g_k$ for $\langle k\rangle\leq N$, and by $\mathcal{B}_{\leq N}^+$  the smallest $\sigma$-algebra containing both $\mathcal{B}_{\leq N}$ and the $\sigma$-algebra generated by the random variables $|g_k|^2$ for $k\in\mathbb{Z}^2$.
\end{df} 

\subsection{Tensors}\label{tensordef}

Here, we introduce the norms that appear in Section \ref{prep}. These norms are standard when applying the method of random averaging operator. See \cite{DNY2022,DNY2024} for more discussions.

\begin{df}[{{\cite[Definition 2.1]{DNY2022}}}]\label{tensornorms} Let $A$ be a finite set and denote $k_A=(k_j)_{j\in A}$. A tensor $h_{k_A} = h(k_A)$ is a function $(\Z^2)^A\to \C$. The support of $h$ is the set of $k_A$ such that $h_{k_A}\neq 0$. These tensors may depend on $\omega$ which belongs to the ambient probability space $\Omega$, though we may omit this dependence. A partition of $A$ is a pair of sets $(B,C)$ such that $B\cup C=A$ and $B\cap C=\varnothing$. For such $(B,C)$, we define the norm $\|\cdot\|_{k_B\to k_C}$ by
 
 \begin{equation}\label{tensornorm}
     \|h\|_{k_B\to k_C}^2=\sup\bigg\{\sum_{k_C}\bigg|\sum_{k_B}h_{k_A}\cdot z_{k_B}\bigg|^2:\sum_{k_B}|z_{k_B}|^2=1\bigg\}.
 \end{equation}

 By duality, we have that 
 \begin{equation}\label{duality}\|h\|_{k_B\to k_C}=\sup\bigg\{\bigg|\sum_{k_B,k_C}h_{k_A}\cdot z_{k_B}\cdot y_{k_C}\bigg|:\sum_{k_B}|z_{k_B}|^2=\sum_{k_C}|y_{k_C}|^2=1\bigg\},
 \end{equation}
 \noi
 hence $\|h\|_{k_B\to k_C}=\|h\|_{k_C\to k_B}=\|\overline{h}\|_{k_B\to k_C}$. If $B=\varnothing$ or $C=\varnothing$, we get the norm $\|\cdot\|_{k_A}$ defined by
 \[\|h\|_{k_A}^2=\sum_{k_A}|h_{k_A}|^2.\] 

  \end{df}
 
 Note that $\|h\|_{k_B\to k_C}\leq \|h\|_{k_A}$ by definition. Then we give some basic definitions to simplify the statements of the following lemmas and propositions.

\begin{df}
\label{DEF:certain}
Let $\theta\ll1 $ be the small constant we have chosen. If some statement $S$ holds with probability $\mathbb{P} (S) \ge 1- C_{\theta} e^{-cA^{\theta}}$ for some $A > 0$,
we say that this statement $S$ is $A$-certain.
\end{df}

\begin{df}[{{\cite[Definition 2.3]{DNY2022}}}]
    \label{DEF:pair}
    For a complex number $z$, we define $z^+ = z$ and $z^- = \bar z$.
For a finite index set $A$,
we also use the notation $z^{\zeta_j}$ for 
$\{ \zeta_j \}_{j \in A}$ with $\zeta_j \in \{
\pm \}$.  We say that $(k_i,k_j)$ for $i,j \in A$ is a pairing in $k_A$,
    if $(k_i - k_j, \zeta_i + \zeta_j) = (0,0)$.
    We say a pairing is over-paired if $k_i = k_j = k_\ell$ for some $\ell \in A \backslash \{i,j\}$.
\end{df}

Now we record some tensor estimates from \cite{DNY2021,DNY2022,DNY2024,Wang2024}.

\begin{prop}[{{\cite[Proposition 4.11]{DNY2022}}} ]
  \label{PROP:puretensor}
  Consider two tensors $h_{k_{A_1}}^{(1)}$ and
  $h_{k_{A_2}}^{(2)}$, where $A_1 \cap A_2 = C$. Let $A_1 \Delta A_2 = A$ and
  define the semi-product
  \[ 
    H_{k_A} = \sum_{k_C} h_{k_{A_1}}^{(1)} h_{k_{A_2}}^{(2)} . 
  \]
  For any partition $(X, Y)$ of $A$, let $X \cap A_i = X_i$, $Y \cap A_i
  = Y_i$ for $i = 1, 2$. Then, we have
  \[ 
    \| H \|_{k_X \to  k_Y} \le  \| h^{(1)} \|_{k_{X_1 \cup C} \to 
    k_{Y_1}} \cdot \| h^{(2)} \|_{k_{X_2} \to  k_{C \cup Y_2}} . 
  \]
\end{prop}

\begin{lem}[{{\cite[Proposition 2.18]{Wang2024}}} ]
  \label{LEM:mixedtensor}
  The following holds.
  \[ 
    \| h^{(1)}_{k_1 k_1'} h^{(2)}_{k_2 k_1'} \|_{k_1' \to  k_1 k_2}
    \le  \| h^{(1)}_{k_1 k_1'} \|_{k_1' \to  k_1} \| h^{(2)}_{k_2
    k_1'} \|_{k'_1 \to  k_2} . 
  \]
\end{lem}

In fact, \cite{Wang2024} treat the case $k_j\in\mathbb{Z}$, but the proof is the same for $k_j\in\mathbb{Z}^2$.

\begin{prop} [Weighted bounds, \cite{DNY2021,DNY2024,Wang2024}]
  \label{Prop:weightedZb} 
  Suppose that matrices $h = h_{k k''}$, $h^{(1)} = h^{(1)}_{k
  k'}$ and $h^{(2)} = h_{k' k''}^{(2)}$ satisfy that
  \[ 
    h_{k k''} = \sum_{k'} h_{k k'}^{(1)} h_{k' k''}^{(2)}, 
  \]
  and $h_{k k'}^{(1)}$ is supported in $\{| k - k' | \lesssim L\}$, then we have
  \[ 
    \Big\| \Big( 1 + \frac{| k - k'' |}{L} \Big)^{\kappa_0} h_{k k''}
    \Big\|_{\ell_{k k''}^2} \lesssim \| h^{(1)}_{kk'} \|_{k \to  k'}
    \cdot \Big\| \Big( 1 + \frac{| k' - k'' |}{L} \Big)^{\kappa_0}
    h^{(2)}_{k' k''} \Big\|_{\ell_{k' k''}^2} . 
  \]
\end{prop}

We also recall the Wiener chaos estimate (\cite[Theorem I.22]{Simon.B}), which is useful when constructing the Gibbs measure. 

\begin{lem}
  \label{LEM:wc}
  Let $\{ g_n \}_{n \in \mathbb{N}}$ be a sequence of
  independent standard real-valued Gaussian random variables. Given $l \in
  \mathbb{Z}_{\ge  0}$, let $\{ P_j \}_{j \in \mathbb{N}}$ be a sequence
  of polynomials in $\bar{g} = \{ g_n \}_{n \in \mathbb{N}}$ of degree at most
  $l$. Then, for $p \ge  2$, we have
  \[ 
    \bigg\| \sum_{j \in \mathbb{N}} P_j (\bar{g}) \bigg\|_{L^p (\Omega)} 
    \le  (p - 1)^{\frac{l}{2}} \bigg\| \sum_{j \in \mathbb{N}} P_j
     (\bar{g}) \bigg\|_{L^2 (\Omega)} . 
  \]
\end{lem}

Last, we record the large-deviation-type result, which is a special case of \cite[Lemma 4.1]{DNY2024}.

\begin{lem}
\label{LEM:LDE}
  Let $E \subset \mathbb{Z}^2$ be a finite set, and $a = a_{k_1 \cdots k_r}
  (\omega)$ be a random tensor such that the collection $\{ a_{k_1 \cdots k_r}
  \}$ is independent of the collection $\{ g_k (\omega) ; k \in E \}$. Let
  $\zeta_j \in \{ \pm \}$ and assume that in the support of $a_{k_1 \cdots
  k_r}$ there is no pairing in $\{ k_1, \cdots, k_r \}$ associated with the
  signs $\zeta_j$. Define the random variable $X$ as
  \begin{align*} 
    X (\omega) := \sum_{k_1, \cdots, k_r\in E} a_{k_1 \cdots k_r} \prod_{j = 1}^r
    \eta_{k_j} (\omega)^{\zeta_j}, 
  \end{align*}
  where $\eta_{k_j} \in \{ g_{k_j} , |g_{k_j}|^2-1 \}$.
  Then, for any $A > |E|$, we have $A$-certainly that
  \[
  \begin{aligned} 
    | X (\omega) |^2 \le  A^{\theta} \cdot \sum_{k_1, \cdots, k_r\in E} |
    a_{k_1 \cdots k_r} (\omega) |^2 . 
  \end{aligned}
  \]
\end{lem}

\noi

\subsection{Norms}\label{normdef}

In this section, we define the function and operator norms that appear in the proof.

\begin{df}
For a space-time function $u \in C(\R; L^2 (\mathbb{T}^2))$, or the kernel $h_{kk'}(t)$ of a time-dependent operator $\mathcal{H}$, we denote
\[ 
\widetilde{u_k}(\lambda) = \widehat{u_k} (\lambda-|k|^\alpha),\quad \widetilde{h_{kk'}}(\lambda) = \widehat{h_{kk'}}(\lambda-|k|^\alpha). 
\]

Then we can define the $X^b$ norm as
\[
\| u \|_{X^b} := \bigg( \int_{\mathbb{R}} \langle \lambda \rangle^{2 b} \|
    \widetilde{u_k} (\lambda) \|_{\ell^2_k}^2 d \lambda \bigg)^{\frac12}. 
\]
\end{df}

\begin{df}
Also we can define the $Y^b$ norm as
\begin{equation}
    \label{Yb}
\| \mathcal{H} \|_{Y^b} :=  \bigg( \int_{\mathbb{R}} \langle \lambda \rangle^{2
    b} \| \widetilde{h_{k k'}} (\lambda) \|_{\ell^2_k\to \ell^2_{k'}}^2 d \lambda \bigg)^{\frac12} 
\end{equation}
\noi
and the $Z^b$ norm (Hilbert-Schmidt norm) of the operator $\mathcal{H}$,

\begin{equation}
\begin{aligned}
    \label{Zb}
    \| \mathcal{H} \|_{Z^b} := \bigg( \int_{\mathbb{R}} \langle \lambda \rangle^{2
    b} \| \widetilde{h_{k k'}} (\lambda) \|_{\ell^2_k \ell^2_{k'}}^2 d \lambda \bigg)^{\frac12}
\end{aligned}    
\end{equation}
\end{df}

More generally, consider an operator $\mathcal P$ over space-time functions given by  
\[
\mathcal P (u) (t,x) = \frac1{(2\pi)^4} \sum_{k \in \Z^2} \int_{\R} \bigg( \sum_{k'\in \Z^2} \int_\R \widetilde {P_{k k'}} (\lambda,\lambda') \widehat u_{k'} (\lambda') d\lambda' \bigg)  e^{i(t\lambda + k\cdot x)} d\lambda,
\]

\noi
where $\widetilde {P_{k k'}} (\lambda,\lambda')$ is similarly given by
\[
\widetilde{P_{kk'}}(\lambda,\lambda') = \widehat {P_{k k'}} (\lambda-|k|^\alpha,|k'|^\alpha-\lambda') .
\]

We can define the operator norm of $\mathcal P$ by
\begin{align}
    \label{Ybb}
    \begin{split}
    \| \mathcal{P} \|_{Y^{b, b'}} & := \big\| \langle \lambda \rangle^{ b}
    \langle \lambda' \rangle^{-  b'} \widetilde{P_{k k'}} (\lambda,\lambda')
    \big\|_{\ell_{k'}^2 L_{\lambda'}^2 \to  \ell_{k}^2 L_{\lambda}^2} \\ 
    & = \sup_{\|u\|_{X^{b'}} = 1} \frac1{(2\pi)^4} \bigg( \int_{\mathbb{R}} \langle \lambda \rangle^{2 b} \sum_{k} \bigg|  \sum_{k'} \int_\R  \widetilde{P_{k k'}} (\lambda,\lambda') \widehat u_{k'} (\lambda') d\lambda' \bigg|^2 d \lambda \bigg)^{\frac12} \\
     & = \sup_{\|u\|_{X^{b'}} = 1} \big\| \mathcal P (u) \big\|_{X^b},
    \end{split}
\end{align}

\noi
and the Hilbert-Schmidt norm of $\mathcal P$ 

\begin{align} 
\label{Zbb} 
\begin{split} 
    \| \mathcal{P} \|_{Z^{b, b'}} 
    & = \bigg( \int_{\mathbb{R}^2} \langle \lambda
    \rangle^{2 b} \langle \lambda' \rangle^{- 2 b'} \left\|  \widetilde{P_{k k'}} (\lambda,
    \lambda') \right\|^2_{\ell^2_k \ell^2_{k'}} d \lambda d \lambda' \bigg)^{\frac12}.
    \end{split}
\end{align}

Given the above definitions, we define the corresponding localized norms for any finite interval $I$
\[ 
  \| u \|_{X^b (I)} := \inf \left\{ \| v \|_{X^b} : v = u \text{ on } I
  \right\} 
\]
\noi
and similarly define $Y^{b, b'} (I)$ and $Z^{b, b'} (I)$. We will use $X^b$ to denote $X^b (I)$ for simplicity unless otherwise specified.

\gb

\section{Random averaging operator} \label{structuresol} 

\subsection{Structure of the solution}

We focus on the local well-posedness. Before the estimates, we may assume

\begin{equation}\label{preofgauss}
    |g_{k}|\lesssim \tau^{-\theta}|k|^\theta,\quad |m_N-\sigma_N|\lesssim \tau^{-\theta},
\end{equation}
\noi
by removing a set of probability less than $C_{\theta}e^{-\tau^{-\theta}}$ since by definition

\begin{equation}\label{gaugephase}
 m_N-\sigma_N = \sum_{\jb{k}\leq N}\frac{|g_k|^2-1}{\jbb{k}^{\alpha}}.    
\end{equation}

Let $v_N$ denote the solution to \eqref{truncfNLSgauge}. Our goal here is
to obtain a quantitative estimate for the difference
$y_N:=v_N-v_{\frac{N}{2}}$, which satisfies the equation

\begin{equation}\label{fNLSdiff}
\begin{cases}
(i\partial_t - {\rm D}^\alpha) y_N =  \Pi_N \left( \Nc_3\left(y_N + v_{\frac{N}{2}}\right) - \Nc_3\left(v_{\frac{N}{2}}\right) \right)
+\Delta_N \Nc_3\left(v_{\frac{N}{2}}\right) \\
\hspace{2.7cm}  +\Pi_N \Qc_3(v_N) - \Pi_{\frac{N}{2}} \Qc_3\left(v_{\frac{N}{2}}\right) \\
y_N(0) = \Delta_N u^\omega.
\end{cases}
\end{equation}
where $\Delta_N = \Pi_N-\Pi_{\frac{N}{2}}$. Consider the high-low-low interaction part

\begin{equation}\label{rao1}
\begin{cases}
    (i\partial_t - {\rm D}^\alpha) \psi_{N,L} = 2\Pi_N \Nc_3(\psi_{N,L}, v_L, v_L)\\
\psi_{N,L}(0) = \Delta_N u^\omega.
\end{cases}
\end{equation}
\noi
where $(N,L)$ is in the set

\begin{equation}\label{setNL}\mathcal{K}_0:=\{(N,L)\in(2^{\mathbb{Z}})^2:2^{-1}\leq L< N^{1-\delta}\}.
\end{equation}

By linearity we have,
\begin{equation}\label{RAOstructure}
(\psi_{N,L})_k
=\sum_{\frac N2<\langle k'\rangle\leq N}H_{kk'}^{N,L}
\frac{g_{k'}(\omega)}{\jbb{k'}^{\frac{\alpha}{2}}}.
\end{equation}

\noindent where for $\frac{N}{2}<\langle k'\rangle\leq N$ and $\langle k\rangle\leq N$, $H_{kk'}^{N,L}=\varphi_k$ is the $k$-th mode of the solution $\varphi$ to the equation
\begin{equation}\label{defpsi2}
\left\{
\begin{aligned}
&(i\partial_t-{\rm D}^\alpha)\varphi=2\Pi_N \Nc_3(\varphi, v_L, v_L),\\
&\varphi(0)=e^{ik'\cdot x}.
\end{aligned}
\right.
\end{equation} 

\noindent We set $H_{kk'}^{N,L}=0$ outside
\(\{\langle k\rangle,\langle k'\rangle\leq N\}\). By definition these $H_{kk'}^{N,L}$, as well as the $h_{kk'}^{N,L}$ defined below, are $\mathcal{B}_{\leq N}$ measurable and $\mathcal{B}_{\leq L}^+$ measurable in the sense of Definition \ref{borel}.

\smallskip
For any dyadic $N,L$, we further define
\begin{equation}\label{matrices}\zeta_{N,L}: =\psi_{N,L} -\psi_{N,\frac{L}{2}},\qquad h^{N,L}: =H^{N,L}-H^{N,\frac{L}{2}}.
\end{equation} 
At the bottom scale we use the convention
\[
 \zeta_{N,\frac12}:=\psi_{N,\frac12},
 \qquad h^{N,\frac12}:=H^{N,\frac12}.
\]
Note that $\psi_{N,\frac{1}{2}}=e^{-it{\rm D}^\alpha}(\Delta_Nu^{\omega})$, and that $H_{kk'}^{N,\frac{1}{2}} =e^{-i|k|^{\alpha}t}\mathbf{1}_{k=k'}$ is restricted to the frequency band $\frac{N}{2}<\langle k\rangle\leq N$. Let $L_N = \max \{ L \in 2^\mathbb{Z} : L < N^{1-\delta} \} $ be the largest $L$ such that $(N,L)\in\mathcal{K}_0$. We can define the remainder as $z_N = y_N- \psi_{N,L_N}$, so it satisfies the equation

\begin{equation}\label{rao2}
\begin{cases}
    (i\partial_t - {\rm D}^\alpha) z_{N} = \Rc_N\\
z_{N}(0) = 0
\end{cases}
\end{equation}
\noi
where

\begin{equation}\label{remainder}
\begin{aligned}
    \Rc_N = & \quad \Pi_N \left( \Nc_3\left(z_N + \psi_{N,L_N}+ v_{\frac{N}{2}}\right) - \Nc_3\left(v_{\frac{N}{2}}\right)  - 2\Nc_3(\psi_{N,L_N}, v_L, v_L)\right)\\
    & \quad +\Delta_N \Nc_3\left(v_{\frac{N}{2}}\right)+  \Pi_N \Qc_3(v_N) - \Pi_{\frac{N}{2}} \Qc_3\left(v_{\frac{N}{2}}\right)\\
    = & \sum_{N_{\max} = N_{\med} = N}\Pi_N\Nc_3(y_{N_1},y_{N_2},y_{N_3})+ \sum_{N_1,N_3\leq\frac{N}{2}}\Pi_N\Nc_3(y_{N_1},z_N+\psi_{N,L_N},y_{N_3})\\
    & \sum_{L_N\leq N_{\med}\leq\frac{N}{2}}2\Pi_N\Nc_3(\psi_{N,L_N},y_{N_2},y_{N_3})+ \sum_{N_2,N_3\leq\frac{N}{2}}2\Pi_N\Nc_3(z_{N},y_{N_2},y_{N_3})\\
    &\quad +\Delta_N \Nc_3\left(v_{\frac{N}{2}}\right)+  \Pi_N \Qc_3(v_N) - \Pi_{\frac{N}{2}} \Qc_3\left(v_{\frac{N}{2}}\right).
\end{aligned}
\end{equation}
Moreover, $z_N$ is $\mathcal{B}_{\leq N}$ measurable.

\begin{rem} With the above construction, we have the following structure for
the gauged truncated solution to \eqref{truncfNLSgauge}:

\begin{equation}\label{ansatzN}
    v_N = e^{-it{\rm D}^\alpha}\Pi_Nu^{\omega}+ \sum_{\substack{(M,L)\in\mathcal{K}_0\\M\leq N}}\zeta_{M,L} + \sum_{M\leq N}z_M.
\end{equation}

If $v=\lim\limits_{N\to\infty}v_N$ denotes the gauged limit, then

\begin{equation}\label{ansatz1}
v = e^{-it{\rm D}^\alpha}u^{\omega}+\sum_{(N,L)\in\mathcal{K}_0}\zeta_{N,L}+z,\quad \text{where} \quad z=\sum_N z_N.
\end{equation} 
The Wick ordered solution is recovered from $v$ by the inverse gauge
transform. This is the full ansatz, where
$\zeta_{N,L}$ can be viewed as a random averaging operator, whose kernel is
essentially given by $h^{N,L}$, applied to the Gaussian free field
$e^{-it{\rm D}^\alpha}u^{\omega}$.
\end{rem}

\subsection{The a priori bounds}

We now present the bound on $J=[-\tau,\tau]$ where $\tau\ll 1$. Recall the relevant parameters defined in \eqref{parameters1}, we aim to show

\begin{equation}\label{pribounds}
\begin{cases} 
\|h^{N,L}\|_{Y^b(J)} \leq L^{-\delta_0}, \\[2pt]
\|h^{N,L}\|_{Z^b(J)} \leq N^{\frac{5-2\alpha}{2}+\gamma_0} L^{-\frac{\gamma_0}{4}}, \\[4pt]
\left\| \left(1 + \dfrac{|k - k'|}{L}\right)^\kappa_0 h_{kk'}^{N,L} \right\|_{Z^b(J)} \leq N, \\[4pt]
\|z_N\|_{X^b(J)} \leq N^{\frac{10}{3}-2\alpha + \gamma}.
\end{cases}
\end{equation}

To prove these bounds, we should restrict $v_N$,$y_N$,$z_N$ and $h^{N,L}$ to the interval $J$ and construct the extensions of these restrictions to all time. See \cite[Section 3.2.1]{DNY2024} for a detailed argument. We still denote them by the original symbols for simplicity. Then it suffices to prove the following proposition:

\begin{prop}\label{local2} Recall the relevant constants defined in \eqref{parameters1}, and assume that $0<\tau\ll 1$. Consider the following statement which we call  $ \mathtt{Loc}(M) $ for $M\geq 1$: for any $(N,L)\in\mathcal{K}_0$ with $L<M$, we have
\begin{align}
\label{induct1}
&\|h^{N,L}\|_{Y^b}\leq L^{-\delta_0},\\
\label{induct2}
&\|h^{N,L}\|_{Z^b}\leq N^{\frac{5-2\alpha}{2}+\gamma_0}L^{-\frac{\gamma_0}{4}},\\
\label{induct3}&\bigg\|\bigg(1+\frac{|k-k'|}{L}\bigg)^{\kappa_0}h_{kk'}^{N,L}\bigg\|_{Z^{b}}\leq N.
\end{align} 

Define the operators

\begin{align}\label{defop1}\mathcal{P}^+(z) = \mathcal{P}_{N,L_2,L_3}^+(z)&:=\eta_{\tau}(t)\cdot\mathcal{I}\Pi_N\mathcal{N}_{3}(z,y_{L_2},y_{L_3}),\\
\label{defop2}\mathcal{P}^-(z) = \mathcal{P}_{L_1,N,L_3}^-(z)&:=\eta_{\tau}(t)\cdot\mathcal{I}\Pi_N\mathcal{N}_{3}(y_{L_1},z,y_{L_3}),
\end{align} then for any $(N,L_{\max})\in\mathcal{K}_0$ as defined in \eqref{setNL} with $L_{\max}<M$ where $L_{\max} = \max\{L_2, L_3\}$ for $\mathcal{P}^{+}$ and $L_{\max} = \max\{L_1, L_3\}$ for $\mathcal{P}^{-}$, we have

\begin{equation}
\| \mathcal{P}^{\pm} \|_{Y^{b,b}} = \|\mathcal{P}^{\pm}\|_{X^{b}\to X^{b}}\leq \tau^{\theta}L^{-2\delta_0}_{\max},\label{rao-OP}
\end{equation}
\noi
and
\begin{equation}\label{rao_HS}
\| \mathcal{P}^{+} \|_{Z^{b,b}}  \lesssim \tau^{\theta} N^{\frac{5-2\alpha}{2}+ \frac{3 \gamma_0}{4} } L^{ -\frac{\gamma_0}{3}}_{\max}.
\end{equation}

Finally for any $N\leq M$ we have
\begin{equation}\label{induct6}\|z_N\|_{X^b}\leq N^{\frac{10}{3}-2\alpha+\gamma}.
\end{equation}

Now suppose that the statement $\mathtt{Loc}(M)$ is true for $\omega\in\Omega_M$,  then the statement $\mathtt{Loc}(2M)$ is true for $\omega\in \Omega_{2M}$ where $\Omega_{2M}$ is another set such that $\mathbb{P}(\Omega_{M}\backslash\Omega_{2M})\leq C_{\theta}e^{-(\tau^{-1}M)^\theta}$. In particular, apart from a set of $\omega$ with probability $\leq C_{\theta}e^{-\tau^{-\theta}}$, the statement $\mathtt{Loc}(M)$ is true for all $M$.
\end{prop}

This is the key ingredient to the proof of the local well-posedness result, Theorem \ref{LWP}. We will present the details in Section \ref{MainProof}. As for the proof of this proposition, we will turn it to multilinear estimates in Section \ref{esti}, which requires various counting and random tensor estimates in Section \ref{prep}.

Before ending this subsection, we record a unitary property of $H^{N,L}_{kk'}$ originating from the $L^2$ conservation of the linear equation \eqref{rao1}, which was first observed by Bourgain \cite{B97} in the context of Hartree-type NLS. 

\begin{lem}\label{unitary}
  Assume $\mathtt{Loc}(M)$ given in Proposition \ref{local2} holds.
  For $L < \min (M, N^{1-\delta})$, there exists $\tau \ll 1$ such that for each $|t| \le \tau$, the operator (kernel)
  $\{ H_{k k'}^{N, L}(t) \}_{k k'}$ is unitary, i.e.
  \begin{align}
    \label{uni}
    \sum_{\langle k\rangle\leq N}
    H_{k_1 k}^{N,L}(t)\overline{H_{k_2 k}^{N,L}(t)}
    =
    \mathbf 1_{\langle k_1\rangle,\langle k_2\rangle\leq N}
    \delta_{k_1k_2}.
  \end{align}
\end{lem}

\begin{cor}[Cancellation after summing dyadic increments]
\label{cor:unitary-cancellation}
Under the assumptions of Lemma \ref{unitary}, let \(L\) be dyadic and
\(L<\min(M,N^{1-\delta})\). For any function \(a:\mathbb Z^2\to\mathbb C\)
and any distinct \(\langle k_1\rangle,\langle k_2\rangle\leq N\),
\begin{equation}
\label{eq:unitary-cancellation}
\sum_{\substack{L_1,L_2\leq L\\L_1,L_2\ {\rm dyadic}}}
\sum_{\langle k'\rangle\leq N}
h_{k_1k'}^{N,L_1}\overline{h_{k_2k'}^{N,L_2}}a(k') =
\sum_{\substack{L_1,L_2\leq L\\L_1,L_2\ {\rm dyadic}}}
\sum_{\langle k'\rangle\leq N}
h_{k_1k'}^{N,L_1}\overline{h_{k_2k'}^{N,L_2}}
\bigl(a(k')-a(k_1)\bigr).
\end{equation}
\end{cor}

The proof of the above lemma and corollary is the same as \cite[Lemma 3.6]{Wang2024} and \cite[Corollary 3.7]{Wang2024} respectively. See also \cite{DNY2021}.

\begin{rem}\label{rem:unitary-redistribution}
The cancellation in Corollary \ref{cor:unitary-cancellation} is an identity
for the full dyadic sum, not for a fixed pair \((L_1,L_2)\). In the pairing
estimates below, we first sum over all dyadic low scales and add the zero
identity in \eqref{eq:unitary-cancellation}; we then redistribute the
result among the dyadic pairs and estimate a representative pair. The
resulting logarithmic multiplicity is absorbed by the \(L^\theta\) losses
already allowed in Remark \ref{RMK:absorb}.
\end{rem}

\smallskip

\section{Counting and tensor estimates}\label{prep}

\subsection{Counting estimates}\label{count}

In the proof of Proposition \ref{local2}, we need counting estimates for
\begin{equation}
 \label{Eqn:S}
 S=\left\{(k,k_1,k_2,k_3)\in(\mathbb Z^2)^4:\quad
 \begin{aligned}
  &k_1-k_2+k_3-k=0,\qquad k_2\notin\{k_1,k_3\},\\
  &|k_1|^\alpha-|k_2|^\alpha+|k_3|^\alpha-|k|^\alpha
  =m+\mathcal O(1),\\
  &|k|\leq N,\qquad |k_j|\leq N_j
  \quad\text{for }j\in\{1,2,3\}
 \end{aligned}
 \right\}.
\end{equation}
Here \(m\) is the localized modulation parameter. We also estimate subsets
obtained by fixing some of the variables, denoted by $S_{k_j}$,
$S_{kk_j}$, and so on. We begin with the two-vector counting. Momentum
conservation leaves only one variable to count after two suitable
frequencies are fixed. More precisely, if we fix two $k_{j'}$'s or
$k$ together with one $k_{j'}$, the remaining set can be written as

\[
    \{k_j\in\Z^2: |k_j|\leq N_j,\, \phi_{b,\pm}(k_j)\in J_0 \},
\]
\noi
where $J_0$ is an interval of length $\mathcal{O}(1)$  and 

\begin{equation}\label{phi-def}
    \phi_{b,\pm}(x) = |x|^{\alpha} \pm |x - b|^{\alpha}, \quad x\in\R^2
\end{equation}

\noindent where $\alpha\in(1,2)$ and $b\in\Z^2$ are fixed. Note that $k_j=0$ and $k_j=b$ are two lattice points which can be counted separately. Hence, we focus on the number of lattice points in the domain

\begin{equation}\label{originaldomain}
    \tilde{S}^{\pm}_\alpha := \left\{x\in\R^2:  |x|\leq N_j \,,\,\, |x|\wedge|x-b|\geq\frac{1}{2}, \,\, \phi_{b,\pm}(x)\in J_0 \right\},
\end{equation}
\noi
on which $\phi_{b,\pm}(x)$ is smooth.

\subsubsection{Two-vector counting}

For the classical case $\alpha=2$, $\tilde{S}^{-}_2$ consists of several line segments of length $\mathcal{O}(N_j)$ while $\tilde{S}^{+}_2$ consists of several circles of radius $\mathcal{O}(N_j)$. Hence, the number theory implies that

\begin{equation}\label{classicalcounting}
    \#(\tilde{S}^{-}_2\cap\mathbb{Z}^2)\lesssim N_j, \quad \#(\tilde{S}^{+}_2\cap\mathbb{Z}^2)\lesssim_{\theta} N_j^{\theta}
\end{equation}

\noi
hold for any sufficiently small $\theta>0$. Taking this into account, we expect that $\tilde{S}^{-}_\alpha$ is a slightly curved strip, whose edge curves are not convex. Moreover, some long line segments are contained in the domain, so we can hardly expect an essentially better estimate for $\alpha\in(1,2)$
than \eqref{classicalcounting} for $\alpha=2$.  While $\tilde{S}^{+}_\alpha$ is roughly an elliptic annulus (See Figure \ref{Figure 2}) with strongly convex edge curves. We start from the $\phi_{b,-}$ case, where the counting estimates are basically reduced to dimension 1. Let us introduce the one-dimensional lattice counting here.

\begin{lem}[Basic lattice counting -1d]
\label{LEM:1d}
Let $I,J\subset\mathbb R$ be intervals, and let $F\in C^1(I)$. Suppose that
$I$ can be decomposed into at most $A$ subintervals $I_\nu$, with
$A=\mathcal{O}(1)$, such that on each $I_\nu$ one has
\[
|F'(t)|\ge \mu>0 .
\]
Then
\[
\#\{k\in I\cap\mathbb Z:\ F(k)\in J\}
\lesssim_A
1+\frac{|J|}{\mu}.
\]
\end{lem}

\begin{proof}
On each interval $I_\nu$, the condition $|F'|\ge \mu$ implies that $F$
is monotone and
$|F(I_\nu\cap E)|\ge \mu |I_\nu\cap E|$
for any interval $E\subset I_\nu$. Hence
\[
|\{t\in I_\nu:\ F(t)\in J\}|
\le \frac{|J|}{\mu}.
\]
The number of integers in this set is therefore
\[
\mathcal{O}\left(1+\frac{|J|}{\mu}\right).
\]
Summing over $\mathcal{O}_A(1)$ intervals $I_\nu$ gives the claim.
\end{proof}

 Note that $|\nabla\phi_{b,-}(x)|$ enjoys a uniform lower bound on $S_{\alpha}^-$ (see \eqref{phi-lowerbound}), so the two partial derivatives can never be quite small at the same time. We may fix the first or second variable of $k_j$ depending on whether the other partial derivative is bounded from below, and count the possible choice of the other by Lemma \ref{LEM:1d}. Regarding the number of connected components $I_{\nu}$, we rely on the following lemma:

\begin{lem}[Bounded complexity]
\label{lem:one-dimensional-structure}
Let $1<\alpha<2$. For $a\in\mathbb R$ and $c\ge0$, define
\[
q_{a,c}(t)
:=
(t-a)\big((t-a)^2+c^2\big)^{\frac{\alpha-2}{2}}.
\]
Then, for any $a_1,a_2\in\mathbb R$ and $c_1,c_2\ge0$, the function
\[
Q(t):=q_{a_1,c_1}(t)-q_{a_2,c_2}(t)
\]
has only $\mathcal{O}_\alpha(1)$ monotonic intervals. Consequently, for every
$\mu>0$, the set
\[
\{t\in\mathbb R:\ |Q(t)|\ge \mu\}
\]
has $\mathcal{O}_\alpha(1)$ connected components.
\end{lem}

\begin{rem}
This lemma shows that the function $Q$ has bounded complexity, a useful property in the counting arguments below. See also \cite[Section~4.1]{DWWZ} for a related use of bounded-complexity sets in the context of semi-algebraic geometry.
\end{rem}

\begin{proof}
A direct computation gives
\[
q'_{a,c}(t)
=
\big((\alpha-1)(t-a)^2+c^2\big)
\big((t-a)^2+c^2\big)^{\frac{\alpha-4}{2}},
\]
and
\[
q''_{a,c}(t)
=
(\alpha-2)(t-a)
\big((\alpha-1)(t-a)^2+3c^2\big)
\big((t-a)^2+c^2\big)^{\frac{\alpha-6}{2}}.
\]
Thus $q'_{a,c}\ge0$, and $q'_{a,c}$ is increasing on $(-\infty,a]$ and
decreasing on $[a,\infty)$. In particular, each individual function $q_{a,c}$ is
increasing and has exactly one change of convexity.

The preceding explicit formulae show that the family
\(\{q_{a,c}\}_{a\in\mathbb R,c\ge0}\) has uniformly bounded one-dimensional
complexity: each \(q_{a,c}\) is increasing, and its derivative has exactly one
maximum. Consequently, the difference
\(Q=q_{a_1,c_1}-q_{a_2,c_2}\) has only \(\mathcal{O}_\alpha(1)\) critical points. Equivalently,
\(Q\) has only \(\mathcal{O}_\alpha(1)\) monotonicity intervals. 

Moreover, the level sets $Q=\mu$ and $Q=-\mu$ have $\mathcal{O}_\alpha(1)$ many points, so
that $\{|Q|\ge \mu\}$ has $\mathcal{O}_\alpha(1)$ connected components.
\end{proof}

With the above two lemmas, we can count the two-dimensional set $\tilde{S}^{-}_\alpha$ defined in \eqref{originaldomain}. Recall the original set $S$ in \eqref{Eqn:S}, this counting applies to $S_{kk_1}$, $S_{kk_3}$, $S_{k_1k_2}$ and $S_{k_2k_3}$. Also note that the non-pairing condition implies that the fixed parameter $b\in\mathbb{Z}^2\backslash\{0\}$.

\begin{lem}[Basic lattice counting -2d]
\label{LEM:2d}
Let $b\in\mathbb Z^2\setminus\{0\}$, $N_j\ge1$. Recall $\tilde{S}^{-}_\alpha$ in \eqref{originaldomain} with 
$J_0\subset\mathbb R$ an interval of length $\mathcal{O}(1)$ and $\phi_{b,-}(x)$ defined in \eqref{phi-def}. Then
\begin{equation}
\label{count:2d}
\#\tilde{S}^{-}_\alpha
\lesssim
\frac{N_j^{3-\alpha}}{\langle b\rangle}+N_j.
\end{equation}
\end{lem}

\begin{proof}
 If $x$ shares the same direction as $b$, or rather $x$ belongs to the set
 \begin{equation}
     E_0: = \tilde{S}^-_\alpha \cap \{x\in\mathbb{Z}^2: \exists\, n_0\in \mathbb{Z},\,\, s.t.\,  x=n_0 b \ \text{or}\ b=n_0 x\},
 \end{equation}
then the number of choices of $x$ is bounded by $\mathcal{O}\left(\frac{N_j}{|b|}+ N_j\wedge|b|\right)$, which is acceptable. In particular, we may assume that the line segment connecting $x$ and $x-b$ does not pass $(0,0)\in\mathbb{R}^2$ for any $x\in\tilde{S}^-_\alpha\backslash E_0$.

Set $G(x):=x|x|^{\alpha-2}$.
Then
$\nabla\phi_{b,-}(x)
=
\alpha\bigl(G(x)-G(x-b)\bigr)$. Note that the derivative $DG(x)$ is the Hessian of $|x|^\alpha$:

\begin{align}\label{eq:H2}
\alpha DG(x) = H(|x|^\alpha)
=
\alpha |x|^{\alpha-2} I
+
\alpha(\alpha-2)|x|^{\alpha-4}x\otimes x.
\end{align}
Since \(1<\alpha<2\), the Hessian \(H(|x|^\alpha)\) is symmetric and positive definite on
\(\mathbb R^2\setminus\{0\}\), with eigenvalues
\begin{align}\label{eq:H3}
\alpha(\alpha-1)|x|^{\alpha-2}
\qquad\text{and}\qquad
\alpha |x|^{\alpha-2}.
\end{align}
In particular, for any $x\neq0$, $\xi\in\mathbb{R}^2$, the vector norm can be estimated as
\begin{equation}\label{symmetricmatrix}
   |DG(x)\xi| \sim |x|^{\alpha-2} |\xi|.
\end{equation}
Now with $\gamma_0(t)=x+tb$, $-1\leq t\leq0$ not passing through the origin, we have the uniform lower bound on $\tilde{S}^-_\alpha\backslash E_0$:

\begin{equation}\label{phi-lowerbound}
    |\nabla\phi_{b,-}(x)|  = \alpha\left| \int_{-1}^0 DG(\gamma_0(t))b\, dt \right|\sim |\gamma_0(t)|^{\alpha-2}|b|\gtrsim (N_j+|b|)^{\alpha-2}|b|.
\end{equation}

We now split the set according to which coordinate derivative is large. Let
\[
E_i
= \left(\tilde{S}^{-}_\alpha \backslash E_0 \right) \bigcap
\left\{
x\in\mathbb Z^2:\ 
|\partial_{x_i}\phi_{b,-}(x)|\ge c (N_j+|b|)^{\alpha-2}|b|
\right\}, \qquad i=1,2,
\]
where $c>0$ is a sufficiently small constant. Since
\[
|\nabla\phi_{b,-}|
\le
|\partial_{x_1}\phi_{b,-}|+|\partial_{x_2}\phi_{b,-}|,
\]
\eqref{phi-lowerbound} implies that
\[
\tilde{S}^-_\alpha 
\subset E_0\cup E_1\cup E_2 .
\]

We estimate $E_1$, and the estimate for $E_2$ is identical. Fix
$n\in\mathbb Z$ with $|n|\lesssim N_j$, and define
\begin{equation}
    F_{n}(t):=\phi_{b,-}\left(t,n\right).
\end{equation}
Then with $b=(b_1,b_2)\neq 0$,
\[
\begin{aligned}
F_{n}'(t)
  =
\alpha\left[
(t+b_1)\left((t+b_1)^2+\left(n+b_2\right)^2\right)^{\frac{\alpha-2}{2}}
-
t\left(t^2+n^2\right)^{\frac{\alpha-2}{2}}
\right],
\end{aligned}
\]
which is of the form
\[
F_{n}'(t)
=
\alpha\left(
q_{-b_1,\,\left|n+b_2\right|}(t)-q_{0,\,\left|n\right|}(t)
\right).
\]
By Lemma~\ref{lem:one-dimensional-structure}, the set
\[
\{t\in\mathbb{R}:\ |F_{n}'(t)|\ge c(N_j+|b|)^{\alpha-2}|b|)\}
\]
has $\mathcal{O}_\alpha(1)$ connected components.

On each such component, we may apply Lemma~\ref{LEM:1d}.
Since $|F_{n}'(t)|\ge c(N_j+|b|)^{\alpha-2}|b|$ and $|J_0|=\mathcal{O}(1)$, we obtain
\[
\#\left\{t\in\mathbb Z:\ \left(t,n\right)\in E_1\right\}
\lesssim
1+\frac{1}{(N_j+|b|)^{\alpha-2}|b|} .
\]
There are $\mathcal{O}(N_j)$ possible values of $n$. Hence
\[
|E_1|
\lesssim
N_j\left(1+\frac{1}{(N_j+|b|)^{\alpha-2}|b|}\right).
\]
The same estimate holds for $E_2$.
Therefore
\[
\#\tilde{S}^-_\alpha\lesssim \sum_{j=0}^2  |E_j|
\lesssim
N_j\left(1+\frac{1}{(N_j+|b|)^{\alpha-2}|b|}\right).
\]

If $|b|\le N_j$, then $(N_j+|b|)^{\alpha-2}|b|\gtrsim N_j^{\alpha-2}|b|$.
Therefore
\[
\#\tilde{S}^-_\alpha
\lesssim
N_j+\frac{N_j^{3-\alpha}}{|b|}
\lesssim
N_j+\frac{N_j^{3-\alpha}}{\langle b\rangle}.
\]
Otherwise $|b|\ge N_j$, then
$(N_j+|b|)^{\alpha-2}|b|\gtrsim |b|^{\alpha-1}$, and hence
\[
\#\tilde{S}^-_\alpha\lesssim N_j.
\]
Combining the two cases proves \eqref{count:2d}.
\end{proof}

\gb

Now we are ready to prove the two-vector counting estimates related to $\phi_{b,-}$. This can be viewed as a two-dimensional extension of the one-dimensional estimate in \cite[Lemma~2.5]{Wang2024}. Although this counting estimate is not sufficient for all the cases needed later, it remains useful in several parts of the argument.

\begin{lem}[Two-vector counting I]
\label{2vec1}
     Let $1\leq N_j\leq N$ for $j\in\{1,2,3\}$ and $\alpha\in(1,2)$, we have the following counting estimates:

\begin{equation}
\begin{aligned}
&|S_{kk_1}| \lesssim \dfrac{(N_2 \wedge N_3)^{3-\alpha}}{\langle k - k_1 \rangle} + N_2 \wedge N_3, \\[3pt]
&|S_{kk_3}| \lesssim \dfrac{(N_1 \wedge N_2)^{3-\alpha}}{\langle k - k_1 \rangle} + N_1 \wedge N_2, \\[3pt]
&|S_{k_1 k_2}| \lesssim \dfrac{N_3^{3-\alpha}}{\langle k_1 - k_2 \rangle} + N_3, \\[3pt]
&|S_{k_2 k_3}| \lesssim \dfrac{N_1^{3-\alpha}}{\langle k_2 - k_3 \rangle} + N_1. \\
\end{aligned}
\end{equation}
\end{lem}

\begin{proof}
We prove the estimate for $S_{kk_3}$, while the remaining estimates follow by relabeling. Fix $k,k_3$. From the relation
$k=k_1-k_2+k_3$, and let $b=k-k_3\neq 0$, we have $k_2=k_1-b$. The phase condition becomes
\[
|k_1|^\alpha-|k_2|^\alpha
=
|k_1|^\alpha-|k_1-b|^\alpha
\in J_0,
\]
where $J_0$ is an interval of length $\mathcal{O}(1)$, depending on the fixed
frequencies $k,k_3$ and on the modulation parameter $m$.

Counting in $k_1$ (with $k_1=0,b$ counted separately), and dropping the additional restriction $|k_2|\le N_2$, it reduces to show:

\begin{equation}\label{eq:2d}
    \#\bigl\{x\in\mathbb Z^2:\ |x|\le N_1,\,\, |x|\wedge|x-b|\geq\frac{1}{2},\,\, \phi_{b,-}(x)\in J_0\bigr\}
\lesssim
\frac{N_1^{3-\alpha}}{\langle b\rangle}+N_1.
\end{equation}
which is exactly the bound implied by Lemma \eqref{LEM:2d}. On the other hand, counting in $k_2$, we write $k_1 = k_2-k_3+k$. Then the phase condition reads
\[
|k_2|^\alpha-|k_2-b'|^\alpha\in J_0', \, \quad b'=k_3-k\neq 0,
\]
with $|J_0'|=\mathcal{O}(1)$. Then Lemma \ref{LEM:2d} implies the bound
\[
|S_{kk_3}|
\lesssim
\frac{N_2^{3-\alpha}}{\langle k-k_3\rangle}+N_2.
\]
Taking the smaller of these two bounds yields
\[
|S_{kk_3}|
\lesssim
\frac{(N_1\wedge N_2)^{3-\alpha}}{\langle k-k_3\rangle}
+N_1\wedge N_2.
\]
This completes the proof.

\end{proof}

\bigskip

For the two-vector counting estimates related to $\phi_{b,+}$, we may show that for $1<\alpha<2$,
\begin{equation}
    |S_{k k_2}|\lesssim  (N_1\wedge N_3)^{2-\frac{\alpha}{2}},\quad
    |S_{k_1 k_3}|\lesssim  N_2^{2-\frac{\alpha}{2}},
\end{equation}
in the spirit of the above proof. Nevertheless, these bounds are far from enough to reach the Gibbsian level. 

A key observation here is that the associated level curve $\phi_{b,+}$ is strongly convex, which suggests the potential for better counting estimates. However, capturing these improvements requires a more precise understanding of the geometric structure, alongside number-theoretic arguments for counting lattice points along curves where curvature comes into play.

\begin{df}
Let $\mathcal{C}$ be a curve of differentiability class $C^2$ whose
curvature $\kappa$ is positive,  then the total curvature
of $\mathcal{C}$ is defined as
\begin{equation}
    \mathscr{K}(\mathcal{C}):=\int_\mathcal{C} \kappa\,ds
\end{equation}
 where $s$ is arclength along
 $\mathcal{C}$ and the radius of curvature of $\mathcal{C}$ is $\rho = 1/\kappa$.
\end{df}

We begin by recalling the following curved strip counting result, which serves as our starting point.

\begin{lem}[Curved strip counting near a convex arc]
\label{LEM:strip}
Let \(\mathcal C\subset\mathbb R^2\) be a convex \(C^2\) arc of length \(L_{\mathcal C}\), whose
total curvature is at most \(\mathcal{O}(1)\). Assume that its curvature satisfies
\[
 \kappa_{\mathcal C}\sim R^{-1}
\]
for some \(R\ge1\). Let \(0< \delta \ll \min\{1,L_{\mathcal C}^{-1}\}\). All hidden constants above are independent of $R, L_\mathcal{C}$. Then
\begin{equation}
\#\{k\in\mathbb Z^2:\operatorname{dist}(k,\mathcal C)\le  \delta \}
\lesssim
1+\frac{L_{\mathcal C}}{R^{\frac{1}{3}}}.
\end{equation}

In particular, if a region \(\tilde{S}\) can be covered by \(K\) such curved
strips, with arc lengths bounded by \(L\) and curvature radii comparable to
\(R\), then
\begin{equation}
\#(\tilde{S}\cap\mathbb Z^2) \lesssim  K\left(1+\frac{L}{R^{\frac{1}{3}}}\right).    
\end{equation}

\end{lem}

\begin{proof}
This is a direct application of the standard lattice-point estimate for
neighborhoods of convex arcs \cite[Theorem~1.6]{Count23}, applied to the
standard lattice \(\mathbb Z^2\), for which \(d_{\mathcal L}=A_{\mathcal L}=1\). We may divide the curve into $\mathcal{O}(1)$ arcs to restrict the total curvature of each arc to less than $\pi$ . Since $\delta$ is chosen so that $\delta L_{\mathcal {C}} \ll 1$, all the hypotheses of \cite[Theorem~1.6]{Count23} are satisfied. Then we reach the desired estimate.

\end{proof}

\begin{rem}
    In fact, \cite[Theorem~1.6]{Count23} is a quantitative result on the counting estimates near a convex arc. However, we only care about the order of the length $L$ and radius of curvature $\rho\sim R$ appearing in the counting bound when $L$ and $R$ are sufficiently large. A large constant independent of $L,R$ is harmless to our analysis.
\end{rem}

    It is easy to see $\phi_{b,+}$ admits a unique minimal point at $x=\frac{b}{2}$ with $\phi_{b,+}\left(\frac{b}{2}\right)=2^{1-\alpha}|b|^{\alpha}$. In particular, the curve $\phi_{b,+}=m$ exists if and only if $m\geq 2^{1-\alpha}|b|^\alpha$. When $m\in( 2^{1-\alpha}|b|^\alpha,\infty)$, as we will see below, $\phi_{b,+}=m$ is a simple closed convex curve, which allows us to apply the above counting lemma. Based on this fact and a careful discussion for singularities, we have the following two-vector counting estimates.

\begin{lem}[Two-vector counting II]
\label{2vec2}
        Let $1\leq N_j\leq N$, for $j\in\{1,2,3\}$ and $\alpha\in(1,2)$, we have the following counting estimates:
\begin{equation}
    |S_{k k_2}|\lesssim (N_1\wedge N_3)^{\frac{8}{3}-\alpha},\quad
    |S_{k_1 k_3}|\lesssim  N_2^{\frac{8}{3}-\alpha}.
\end{equation}

\end{lem}

\begin{proof}
We prove the second estimate, and the first is identical after relabeling.
Fix \(k_1,k_3\). Then the remaining variable is \(k_2\), and the counting
problem is reduced to estimating the number of lattice points in
\[
\tilde{S}
=
\left\{
x\in\mathbb Z^2:\ |x|\le N_2,\ 
\phi_{b,+}(x)\in J_0
\right\},
\]
where
\[
b=k_1+k_3,\qquad
\phi_{b,+}(x)=|x|^\alpha+|x-b|^\alpha,
\]
and \(J_0\) is an interval of length \(\mathcal{O}(1)\). 

We decompose dyadically in the radial scale
\[
\tilde{S}_M
=
\left\{
x\in\mathbb R^2:
\frac M2<|x|\le M,\ 
\phi_{b,+}(x)\in J_0
\right\},\qquad 1\le M\le N_2.
\]
It suffices to prove
\begin{equation}\label{dyadicbound}
    \#(\tilde{S}_M\cap\mathbb Z^2)
\lesssim
M^{\frac83-\alpha}
\end{equation}
for each dyadic \(M\), since summing over \(M\le N_2\) then gives
\[
|S_{k_1k_3}|
\lesssim
\sum_{M\le N_2}M^{\frac83-\alpha}
\lesssim
N_2^{\frac83-\alpha}.
\]

On \(\tilde{S}_M\), the level set
\[
\phi_{b,+}(x)=m
\]
is a simple convex curve, apart from the harmless singular points \(x=0,b\), which are
counted separately. The total curvature is always bounded by $2\pi$. We consider three geometric regimes.

\medskip

\noindent
\textbf{Case 1: \(M\gg |b|\).}

\begin{figure}[ht]
\centering
\includegraphics[width=0.6\textwidth]{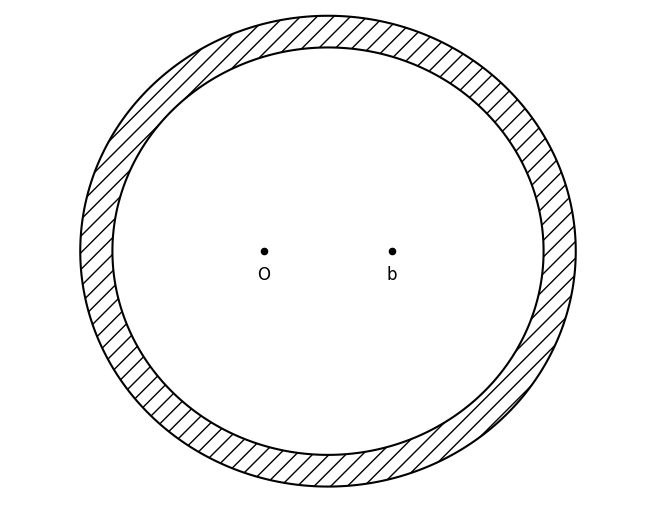}
\caption{The thickness of $\tilde{S}_M$ with $M\gg|b|$ are larger in scale than the admissible width of neighborhood in Lemma \ref{LEM:strip}, so a finer partition is necessary}
\label{Figure 2}
\end{figure}

In this case, we consider the dyadic region $\tilde{S}_M$ with \(J_0\) being an interval of length \(\mathcal{O}(1)\) and $ M\gg |b|$, which implies that     
$|x|\sim |x-b|\sim M$.
Since
\[
\nabla\phi_{b,+}(x)
=
\alpha x|x|^{\alpha-2}
+
\alpha (x-b)|x-b|^{\alpha-2}= \alpha\left[ x(|x|^{\alpha-2}+|x-b|^{\alpha-2}) - b|x-b|^{\alpha-2}\right],
\]
and  \(M\gg |b|\),
we have
\begin{equation}
\label{eq:N1}
|\nabla\phi_{b,+}(x)|
\sim M^{\alpha-1}.
\end{equation}

Next, the Hessian satisfies
\begin{align}\label{eq:H1}
H(\phi_{b,+})(x)
=
H(|x|^\alpha)+H(|x-b|^\alpha),
\end{align}
where $H(|x|^\alpha)$ is defined in \eqref{eq:H2}. In particular, in the sense of quadratic forms\footnote{Equivalently, for every \(\xi\in\mathbb R^2\),
\[
\alpha(\alpha-1) |x|^{\alpha-2}|\xi|^2
\le
\xi^T H(|x|^\alpha)\xi
\le
\alpha |x|^{\alpha-2}|\xi|^2,
\]}, the Rayleigh quotient theorem implies that
\[
\alpha(\alpha-1) |x|^{\alpha-2} I
\le H(|x|^\alpha)
\le \alpha |x|^{\alpha-2} I .
\]
and therefore 

\begin{equation}\label{Hessianequal}
    \alpha(\alpha-1) (|x|^{\alpha-2}+|x-b|^{\alpha-2}) I
\le H(\phi_{b,+})(x)
\le \alpha (|x|^{\alpha-2}+|x-b|^{\alpha-2}) I.
\end{equation}

In the regime \(M\gg |b|\), we have \(|x|\sim |x-b|\sim M\). Hence
\begin{equation}
\label{eq:H11}
c_\alpha M^{\alpha-2} I
\le H(\phi_{b,+})(x)
\le C_\alpha M^{\alpha-2} I.
\end{equation}

The standard formula for the curvature of a level curve $\phi_{b,+}(x)=m$ is
\begin{align}\label{eq:curv}
\kappa
=
\frac{
|(\nabla^\perp \phi_{b,+})^T
H(\phi_{b,+})
\nabla^\perp \phi_{b,+}|
}
{|\nabla\phi_{b,+}|^3},
\end{align}
where $\nabla^{\perp}\phi_{b,+} = (-\partial_{x_2}\phi_{b,+}\, ,\,\partial_{x_1}\phi_{b,+})^T$ is the tangent vector of the curve. By using \eqref{eq:N1} and \eqref{eq:H11}, 
we obtain
\[
\big|
(\nabla^\perp\phi_{b,+})^T
H(\phi_{b,+})
\nabla^\perp\phi_{b,+}
\big|
\sim
M^{\alpha-2} | \nabla^\perp\phi_{b,+} |^2
\sim M^{3\alpha-4},
\]
where we used $| \nabla^\perp\phi_{b,+}(x) | = | \nabla \phi_{b,+}(x) | \sim M^{\alpha -1}$.
Therefore, 
\begin{equation}
\label{eq:K1}
\kappa
=
\frac{
\big|
(\nabla^\perp\phi_{b,+})^T
H(\phi_{b,+})
\nabla^\perp\phi_{b,+}
\big|
}
{|\nabla\phi_{b,+}|^3}
\sim
\frac{M^{3\alpha-4}}{M^{3\alpha-3}}
=
M^{-1}.
\end{equation}
Equivalently, the radius of curvature satisfies $\rho\sim M$.

Now fix a level value \(
\ell\in J_0\), and consider a connected component \(\mathcal C_\ell\)
of the level curve
\[
\phi_{b,+}(x)=\ell
\]
inside the dyadic annulus \(\frac{M}{2}<|x|\le M\). By \eqref{eq:K1}, each
such component is a convex arc with radius of curvature \(\rho\sim M\). Its
length is bounded by
$\operatorname{Length}(\mathcal C_\ell)\lesssim M$, which can be divided into $\mathcal{O}(1)$ subarcs of length $\ll M$ and total curvature $\ll \pi$.

The modulation window has size \(\mathcal{O}(1)\). Since
$|\nabla\phi_{b,+}(x)|\sim M^{\alpha-1}$,
the geometric thickness of the region
\[
\{x:\phi_{b,+}(x)\in J_0\}
\]
in the normal direction to the level curve is
$\mathcal{O}(M^{1-\alpha})$.
In order to apply the curved strip lattice counting lemma, we decompose this
strip into thinner strips of width comparable to \(M^{-1}\), the reciprocal of
the curvature radius. 
More precisely, decompose \(J_0\) into subintervals \(J_\ell\) of length $|J_\ell|\sim M^{\alpha-2}$.
Since \(|J_0|=\mathcal{O}(1)\), the number of such subintervals is
$\mathcal{O}(M^{2-\alpha})$.
For each  \(J_\ell\), the corresponding subregion
\[
\tilde{S}_{M,\ell}
=
\{x\in\tilde{S}_M:\phi_{b,+}(x)\in J_\ell\}
\]
has normal thickness
\[
\mathcal{O}\left(\frac{|J_\ell|}{|\nabla\phi_{b,+}|}\right)
=
\mathcal{O}\left(\frac{M^{\alpha-2}}{M^{\alpha-1}}\right)
= C_{0} M^{-1}.
\]
Therefore \(\tilde{S}_{M,\ell}\) is contained in a $C_{0}M^{-1}$-neighborhood of
a convex arc \(\mathcal C_\ell\) of length \(\mathcal{O}(M)\) and curvature radius \(\rho\sim M\).

By  Lemma \ref{LEM:strip}, the number of lattice points in a
$\delta =C_{0}M^{-1}$-neighborhood of such an arc is bounded by
\[
\mathcal{O}\left(1+\frac{M}{M^{\frac{1}{3}}}\right)
=
\mathcal{O}(M^{\frac{2}{3}}).
\]
If necessary, we subdivide each connected component of the level curve $\mathcal C_{\ell}$ into
\(\mathcal{O}(1)\) subarcs $\mathcal C_{\ell,j}$ whose length $L_{\mathcal{C}_{\ell, j}}\ll M$ so that $\delta L_{\mathcal{C}_{\ell,j}}\ll 1$ . Summing over $j$ gives the counting estimate for $\tilde{S}_{M,\ell}$. This only changes the
implicit constants in the final lattice counting bound.
Hence
\[
\#(\tilde{S}_{M,\ell}\cap\mathbb Z^2)
\lesssim M^{\frac{2}{3}}.
\]
Summing over the \(\mathcal{O}(M^{2-\alpha})\) subintervals \(J_\ell\), we obtain
\begin{equation}
\#(\tilde{S}_M\cap\mathbb Z^2)
\lesssim
M^{2-\alpha}M^{\frac{2}{3}}
=
M^{\frac83-\alpha}.
\end{equation}
This proves the desired dyadic bound \eqref{dyadicbound} in the case \(M\gg |b|\).


\medskip

\noindent
\textbf{Case 2: \(M\ll |b|\).}

\begin{figure}[ht]
\centering
\includegraphics[width=0.5\textwidth]{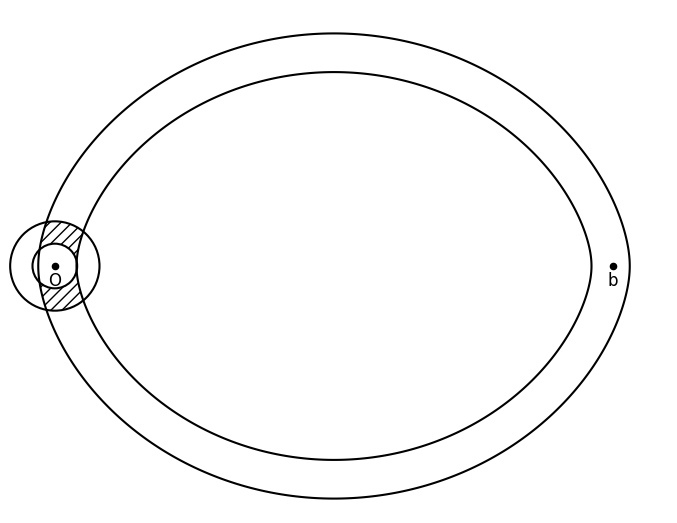}
\caption{The condition $|x|\sim M$ restrict $\tilde{S}_M$ to the neighborhood of one singular point $O$ for $ M\ll|b|$, where the curvature behaves distinctly depending on $M$}
\label{Figure 3}
\end{figure}

We now consider the regime
$M\ll |b|$.
Since \(x\in\tilde{S}_M\), we have \(|x|\sim M\). On the other hand,
$|x-b|\sim |b|$.
Thus the two terms in
\[
\phi_{b,+}(x)=|x|^\alpha+|x-b|^\alpha
\]
live at different scales. The region \(\tilde{S}_M\) is close to one of the singular
points, namely \(x=0\), but it is far away from the other singular point \(x=b\).

We first estimate the gradient. Since
\[
\nabla\phi_{b,+}(x)
=
\alpha x|x|^{\alpha-2}
+
\alpha (x-b)|x-b|^{\alpha-2},
\]
we have
\[
|x|^{\alpha-1}\sim M^{\alpha-1},
\qquad
|x-b|^{\alpha-1}\sim |b|^{\alpha-1}.
\]
Because \(M\ll |b|\) and \(\alpha>1\), the second term dominates. Therefore
\begin{equation}
\label{eq:N2}
|\nabla\phi_{b,+}(x)|
\sim |b|^{\alpha-1}.
\end{equation}

Next, we estimate the Hessian $H$. Because \(\alpha-2<0\) and \(M\ll |b|\), we have
$M^{\alpha-2}\gg |b|^{\alpha-2}$.
Thus \eqref{Hessianequal} implies that the Hessian is dominated by the contribution of the nearby singularity
\(x=0\), and
\begin{equation}
\label{eq:H22}
c_\alpha M^{\alpha-2}I
\le H(\phi_{b,+})(x)
\le C_\alpha M^{\alpha-2}I.
\end{equation}

Using the curvature formula \eqref{eq:curv} for a level curve, and combining \eqref{eq:N2} and \eqref{eq:H22}, we obtain
\[
\kappa
=
\frac{
\big|(\nabla^\perp\phi_{b,+})^T
H(\phi_{b,+})
\nabla^\perp\phi_{b,+}\big|
}
{|\nabla\phi_{b,+}|^3} \sim \frac{M^{\alpha-2}
|\nabla^\perp\phi_{b,+}|^2}{|b|^{3\alpha-3}} \sim \frac{1}{M^{2-\alpha}|b|^{\alpha-1}}.
\]
Equivalently, the radius of curvature satisfies
\begin{equation} 
\rho \sim M^{2-\alpha}|b|^{\alpha-1}.
\end{equation}
Notice that \(\rho\gg1\) in the present regime \(|b| \gg M\gg1\).

We also need to understand the length of the relevant level curve inside the
dyadic region \(\tilde{S}_M\). Since \(|x|\sim M\), the set is contained in an
annulus of radius \(M\). Hence every connected component of a level curve $\mathcal{C}_{\ell}$
\[
\phi_{b,+}(x)=\ell
\]
inside \(\tilde{S}_M\) has length
\begin{equation*}
L_{\mathcal C_\ell}\lesssim M.
\end{equation*}

We may decompose \(J_0\) into subintervals \(J_\ell\) of length
\[
|J_\ell|
\sim
|\nabla\phi_{b,+}|\rho^{-1}
\sim
|b|^{\alpha-1}
\cdot
\frac{1}{M^{2-\alpha}|b|^{\alpha-1}}
=
M^{\alpha-2}.
\]
Since \(|J_0|=\mathcal{O}(1)\), the number of such subintervals is again
$\mathcal{O}(M^{2-\alpha})$.
For each \(J_\ell\), the corresponding region
\[
\tilde{S}_{M,\ell}
=
\{x\in\tilde{S}_M:\phi_{b,+}(x)\in J_\ell\}
\]
has normal thickness
\[
\mathcal{O}\left(\frac{|J_\ell|}{|\nabla\phi_{b,+}|}\right)
=
\mathcal{O}\left(
\frac{M^{\alpha-2}}{|b|^{\alpha-1}}
\right)
=
\mathcal{O}(\rho^{-1}).
\]
Thus \(\tilde{S}_{M,\ell}\) is contained in an \(\mathcal{O}(\rho^{-1})\)-neighborhood of a
convex arc \(\mathcal C_\ell\) of length \(\mathcal{O}(M)\) and curvature radius \(\rho\).

Note that $\rho^{-1}L_{\mathcal{C_\ell}} =\mathcal{O}\left( \frac{M^{\alpha-1}}{|b|^{\alpha-1}}\right)\ll1$. By Lemma~\ref{LEM:strip}, each such strip contains at most
\[
\mathcal{O}\left(1+\frac{M}{\rho^{\frac{1}{3}}}\right)
\]
lattice points. Since \(\rho\sim M^{2-\alpha}|b|^{\alpha-1}\), this gives
\[
\#(\tilde{S}_{M,\ell}\cap\mathbb Z^2)
\lesssim
1+\frac{M}{(M^{2-\alpha}|b|^{\alpha-1})^{\frac{1}{3}}}.
\]
Summing over the \(\mathcal{O}(M^{2-\alpha})\) slices, we obtain
\begin{equation}
\#(\tilde{S}_M\cap\mathbb Z^2)
\lesssim
M^{2-\alpha}
\left(
1+
\frac{M}{(M^{2-\alpha}|b|^{\alpha-1})^{\frac{1}{3}}}
\right) 
 \lesssim M^{\frac{8}{3} - \alpha},
\end{equation}
where we used \(1 \le M\ll |b|\).
This proves the desired dyadic estimate \eqref{dyadicbound} in the case \(M\ll |b|\).


\medskip

\noindent
\textbf{Case 3: \(M\sim |b|\).}

We now consider the transition regime \(M\sim |b|\). This is the most delicate
case, since the dyadic annulus \(|x|\sim M\) may interact both with the singular
points \(0,b\) and with the unique minimum point \(\frac{b}{2}\) of \(\phi_{b,+}\).

Fix a sufficiently small constant \(\eta>0\). We decompose
\[
\tilde{S}_M
=
\tilde{S}_M^{(0)}
\cup
\tilde{S}_M^{(b)}
\cup
\tilde{S}_M^{(\mathrm{min})}
\cup
\tilde{S}_M^{(\mathrm{reg})},
\]
where
\begin{equation}
\begin{aligned}
    \tilde{S}_M^{(0)}
&:=
\tilde{S}_M\cap\{|x|\le \eta M\},\\
\tilde{S}_M^{(b)}
&:=
\tilde{S}_M\cap\{|x-b|\le \eta M\},\\
\tilde{S}_M^{(\mathrm{min})}
&:=
\tilde{S}_M\cap\left\{\left|x-\frac b2\right|\le \eta M\right\},\\
\end{aligned}
\end{equation}
and \(\tilde{S}_M^{(\mathrm{reg})}\) is the remaining part of \(\tilde{S}_M\). We
estimate these four regions separately.

\smallskip

\noindent
\underline{\it The singular region near \(0\).}
We first consider \(\tilde{S}_M^{(0)}\). Note that the curvature of the level curve differs depending on its distance to the singularity. See Case 2 and Figure \ref{Figure 4}. Hence, we decompose \(\tilde{S}_M^{(0)}\) dyadically according to

\begin{figure}[ht]
\centering
\includegraphics[width=0.8\textwidth]{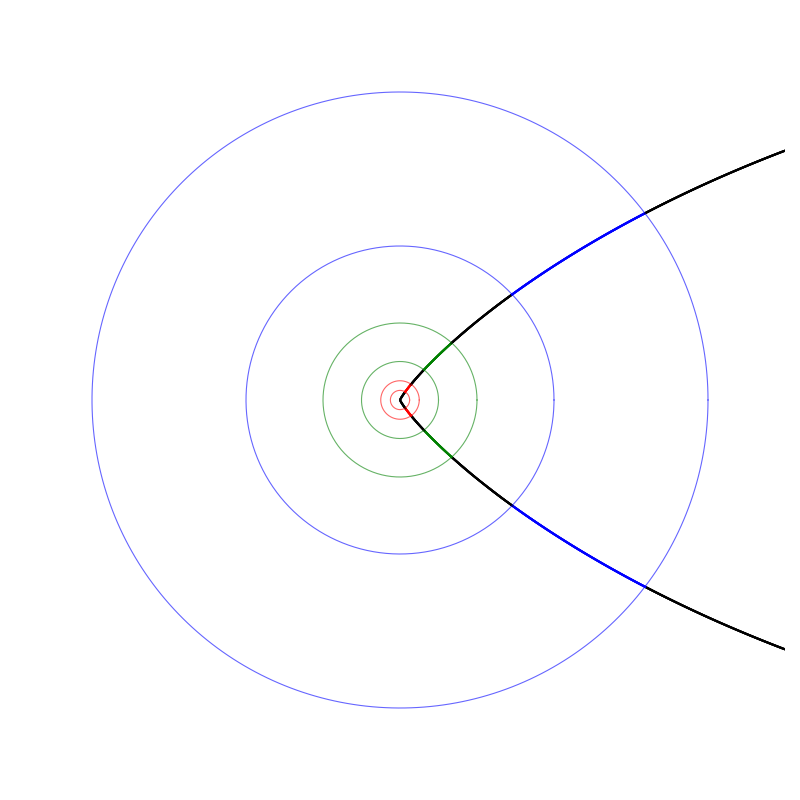}
\caption{Curvature differs acoording to the dyadic scale}
\label{Figure 4}
\end{figure}

\[
R\le |x|\le 2R,
\qquad
1\lesssim R\le \eta M.
\]
The region \(|x|\lesssim1\) contains only \(\mathcal{O}(1)\) lattice points and is
harmless, so we assume \(R\gg1\). On this dyadic piece we have
\[
|x|\sim R,
\qquad
|x-b|\sim |b|\sim M.
\]
Therefore
\begin{equation}
    |\nabla\phi_{b,+}(x)|=\Big|\alpha x|x|^{\alpha-2}+\alpha(x-b)|x-b|^{\alpha-2}\Big|\sim M^{\alpha-1},
\end{equation}
since the second term dominates when \(R\le \eta M\), provided \(\eta>0\) is
chosen sufficiently small.

For the Hessian, \eqref{Hessianequal} implies that 
\begin{equation}
    H(\phi_{b,+})(x)\sim R^{\alpha-2}I,
\end{equation}
in the sense of quadratic forms,  since \(R\ll M\) and
\(\alpha-2<0\). 

The curvature formula \eqref{eq:curv}, together with the above estimates, gives.
\[
\kappa
\sim
\frac{
R^{\alpha-2}|\nabla\phi_{b,+}|^2
}
{|\nabla\phi_{b,+}|^3}
\sim
\frac{R^{\alpha-2}}{M^{\alpha-1}}
=
\frac{1}{R^{2-\alpha}M^{\alpha-1}}.
\]
Thus the curvature radius satisfies
\begin{equation}
    \rho\sim R^{2-\alpha}M^{\alpha-1},
\end{equation}
 while the length of each relevant level-curve component inside \(R\le |x|\le2R\) is
\(\mathcal{O}(R)\).

We now decompose \(J_0\) into subintervals \(J_\ell\) of length
\[
|J_\ell|
\sim
|\nabla\phi_{b,+}|\rho^{-1}
\sim
M^{\alpha-1}
\frac{1}{R^{2-\alpha}M^{\alpha-1}}
=
R^{\alpha-2}.
\]
Since \(|J_0|=\mathcal{O}(1)\), the number of such subintervals is \(\mathcal{O}(R^{2-\alpha})\).
For each \(J_\ell\), the corresponding region has normal thickness
\[
\mathcal{O}\left(
\frac{|J_\ell|}{|\nabla\phi_{b,+}|}
\right)
=
\mathcal{O}\left(
\frac{R^{\alpha-2}}{M^{\alpha-1}}
\right)
=
\mathcal{O}(\rho^{-1}).
\]
Thus each such piece is contained in an \(\mathcal{O}(\rho^{-1})\)-neighborhood of a
convex arc of length \(\mathcal{O}(R)\) and curvature radius \(\rho\). By
Lemma~\ref{LEM:strip}, each such strip contains at most
\[
\mathcal{O}\left(
1+\frac{R}{\rho^{\frac{1}{3}}}
\right)
\]
lattice points. Therefore the contribution of the dyadic region
\(R\le |x|\le2R\) is bounded by
\[
R^{2-\alpha}
\left(
1+
\frac{R}{(R^{2-\alpha}M^{\alpha-1})^{\frac{1}{3}}}
\right) \le R^{2-\alpha} \frac{R}{(R^{2-\alpha}R^{\alpha-1})^{\frac{1}{3}}} \lesssim 
R^{\frac83-\alpha},
\]
where the harmless \(R^{2-\alpha}\) term is also bounded by
\(R^{\frac83-\alpha}\) for \(R\ge1\). Summing over dyadic
\(1\lesssim R\le \eta M\), we get
\begin{equation}\label{0bound}
    \#(\tilde{S}_M^{(0)}\cap\mathbb Z^2)
\lesssim
M^{\frac83-\alpha}.
\end{equation}

\smallskip

\noindent
\underline{\it The singular region near \(b\).}
The region \(\tilde{S}_M^{(b)}\) is treated in exactly the same way by symmetry.
Indeed, after the change of variables \(x'=x-b\), one has
\[
|x'|\le \eta M,
\qquad
|x'+b|\sim M,
\]
and the phase
\[
|x'|^\alpha+|x'+b|^\alpha
\]
has the same structure as in the preceding argument. Hence
\begin{equation}\label{bbound}
    \#(\tilde{S}_M^{(b)}\cap\mathbb Z^2)\lesssim M^{\frac83-\alpha}.
\end{equation}

\smallskip

\noindent
\underline{\it The regular region.}
We now consider
\(\tilde{S}_M^{(\mathrm{reg})}\), where we are at least $\eta M\sim \eta|b|$ away from \(0\), \(b\), and \(b/2\)  .
Scale the problem to unit size by writing
\[
x=|b|y,
\qquad
\hat b=\frac b{|b|}.
\]
Then
\[
\phi_{b,+}(x)
=
|b|^\alpha g(y),
\qquad
g(y)=|y|^\alpha+|y-\hat b|^\alpha.
\]
Since \(M\sim |b|\), the dyadic region \(|x|\sim M\) corresponds to a fixed-size
region \(|y|\sim1\). On the regular region, we also have
\[
|y-\hat b|\gtrsim1,
\qquad
\left|y-\frac{\hat b}{2}\right|\gtrsim1.
\]
Thus
\[
|\nabla g(y)|\sim1,
\qquad
H(g)(y)\sim I
\]
in the sense of quadratic forms. Note that the hidden constant is independent of $\hat{b}$ since the graph of $g$ for different $\hat{b}$ is the same up to rotation. Returning to \(x\)-variables, this gives
\begin{equation}
|\nabla\phi_{b,+}(x)|=|b|^{\alpha-1}|\nabla g(y)|\sim M^{\alpha-1},
\end{equation}
and
\begin{equation}
 H(\phi_{b,+})(x) = |b|^{\alpha-2}H(g)(y)
\sim M^{\alpha-2}I.
\end{equation}

Therefore,
\begin{equation}
\kappa
\sim
\frac{M^{\alpha-2}M^{2\alpha-2}}{M^{3\alpha-3}}
=
M^{-1},
\qquad
\rho\sim M.   
\end{equation}

Each level-curve component in this region has length \(\mathcal{O}(M)\). Decomposing
\(J_0\) into \(\mathcal{O}(M^{2-\alpha})\) subintervals of length \(M^{\alpha-2}\), and dividing each level curve into $\mathcal{O}(1)$ subarcs if necessary, as in
Case 1, then applying Lemma~\ref{LEM:strip}, we obtain
\begin{equation}\label{regularbound}
\#(\tilde{S}_M^{(\mathrm{reg})}\cap\mathbb Z^2)
\lesssim
M^{2-\alpha}M^{\frac{2}{3}}
=
M^{\frac83-\alpha}.    
\end{equation}

\smallskip

\noindent
\underline{\it The minimum region.}
It remains to estimate \(\tilde{S}_M^{(\mathrm{min})}\), where
\[
\left|x-\frac b2\right|\le \eta M.
\]
Write $x=\frac b2+y $. We decompose this region dyadically according to
\[
R\le |y|\le2R,
\qquad
1\lesssim R\le \eta M.
\]
The very small region \(|y|\lesssim1\) contains \(\mathcal{O}(1)\) lattice points and is
harmless, so we can assume $R \gtrsim 1$ in the above.

On the dyadic region \(|y|\sim R\), with \(R\le \eta M \ll M\sim |b|\), both
\[
\left|\frac b2+y\right|\sim \frac{|b|}{2} \sim M,
\qquad
\left|-\frac b2+y\right|\sim \frac{|b|}{2}\sim M
\]
provided \(\eta>0\) is sufficiently small.

We now estimate the geometry on the dyadic region
\[
        R\le \left|x-\frac b2\right|\le 2R.
\]
Since \(M\sim |b|\) and \(R\le \eta M\), by choosing \(\eta>0\) sufficiently small we have $\left|\pm \frac{b}2+ty \right|\sim M$. Therefore, for every \(0\le t\le1\), $H(\phi_{b,+}) \Big( \frac{b}2+ty \Big)$ is symmetric and positive definite. Moreover, the two eigenvalues of $H(\phi_{b,+}) \Big( \frac{b}2+ty \Big)$ are around $ M^{\alpha-2}$ by \eqref{eq:H3} and Weyl's eigenvalue inequality. The fundamental theorem of calculus implies that
\[
        \nabla\phi_{b,+}\Big( \frac{b}2+y \Big)
        =
        \int_0^1 H(\phi_{b,+})\Big( \frac{b}2+ty \Big)y\,dt,
\]
and as in the estimates \eqref{symmetricmatrix} and \eqref{phi-lowerbound},
\[
    \bigg|\nabla\phi_{b,+}\Big( \frac{b}2+y \Big) \bigg|
\sim
M^{\alpha-2}|y|.
\]

Thus, on the dyadic region \(|y|\sim R\),
\begin{align}\label{eq:N34}
|\nabla\phi_{b,+}(x)|
\sim
M^{\alpha-2}R.
\end{align}
Also, since \(|x|\sim |x-b|\sim M\) in this region, 
\begin{align}\label{eq:H34}
H(\phi_{b,+})(x)\sim M^{\alpha-2}I
\end{align}
in the sense of quadratic forms, in view of \eqref{Hessianequal}.


Thus, by \eqref{eq:curv}, \eqref{eq:N34}, and \eqref{eq:H34}, we have
\[
\kappa
\sim
\frac{
M^{\alpha-2}|\nabla\phi_{b,+}|^2
}
{|\nabla\phi_{b,+}|^3}
=
\frac{M^{\alpha-2}}{|\nabla\phi_{b,+}|}
\sim
R^{-1}.
\]
Therefore the curvature radius is
\begin{equation}
    \rho\sim R,
\end{equation}
 while the length $L_{\mathcal{C}_\ell}$ of each relevant level curve  $\phi_{b,+}(x)=\ell$ in this dyadic region is
\(\mathcal{O}(R)\).

We decompose \(J_0\) into subintervals \(J_\ell\) of length
\[
|J_\ell|
\sim
|\nabla\phi_{b,+}|\rho^{-1}
\sim
(M^{\alpha-2}R)\cdot R^{-1}
=
M^{\alpha-2}.
\]
Thus the number of such subintervals is \(\mathcal{O}(M^{2-\alpha})\). For each
\(J_\ell\), the corresponding strip has normal thickness \(\mathcal{O}(\rho^{-1})\), and dividing the strip as before if necessary,
Lemma~\ref{LEM:strip} gives  at most
\[
\mathcal{O}\left(
1+\frac{R}{R^{\frac{1}{3}}}
\right)
=
\mathcal{O}(R^{\frac{2}{3}})
\]
lattice points. Hence the contribution from the dyadic region
\(|x-b/2|\sim R\) is bounded by
\[
\mathcal{O}(M^{2-\alpha}R^{\frac{2}{3}}).
\]
Summing over dyadic \(1\lesssim R\le \eta M\), we get
\begin{equation}\label{minbound}
\#(\tilde{S}_M^{(\mathrm{min})}\cap\mathbb Z^2)
\lesssim
\sum_{1\lesssim R\le \eta M}
M^{2-\alpha}R^{\frac{2}{3}}
\lesssim
M^{2-\alpha}M^{\frac{2}{3}}
=
M^{\frac83-\alpha}.
\end{equation}

Combining the four estimates \eqref{0bound},\eqref{bbound},\eqref{regularbound},\eqref{minbound}, we conclude that
\begin{equation}
    \#(\tilde{S}_M\cap\mathbb Z^2)
\lesssim
M^{\frac83-\alpha}
\end{equation}
in the transition case \(M\sim |b|\).

\medskip 

This completes the proof of the dyadic estimate \eqref{dyadicbound} and hence the lemma.
\end{proof}

\subsubsection{Three-vector and four-vector countings}

Based on the two-vector estimates above, we have the following three-vector counting:

\begin{lem}[Three-vector counting]
\label{3vec1} Let $1\leq N_j\leq N$ for $j\in\{1,2,3\}$ and $\alpha\in(1,2)$, we have the following counting estimates:

\begin{equation}
\begin{aligned}
&|S_k| \lesssim \min \left\{ (N_2 \wedge N_3)^{3-\alpha} N_1 \log N_1 + N_1^2 (N_2 \wedge N_3),\, N_2^2 (N_1 \wedge N_3)^{\frac{8}{3} - \alpha}, \right.\\[2pt]
&\left.\hspace{2cm} (N_1 \wedge N_2)^{3-\alpha} N_3 \log N_3 + N_3^2 (N_1 \wedge N_2) \right\}, \\[2pt]
&|S_{k_1}| \lesssim \min \left\{ N_3^{3-\alpha} N_2 \log N_2 + N_2^2 N_3,\, N_3^2 N_2^{\frac{8}{3} - \alpha} \right\}, \\[2pt]
&|S_{k_2}| \lesssim \min \left\{ N_3^{3-\alpha} N_1 \log N_1 + N_1^2 N_3,\, N_1^{3-\alpha} N_3 \log N_3 + N_3^2 N_1 \right\}, \\[2pt]
&|S_{k_3}| \lesssim \min \left\{ N_1^2 N_2^{\frac{8}{3} - \alpha},\, N_1^{3-\alpha} N_2 \log N_2 + N_2^2 N_1 \right\}. \\
\end{aligned}
\end{equation}
    
\end{lem}

\begin{proof}

Notice that we can first fix two vectors as follows:

\begin{equation}
\begin{cases}
    |S_k| \leq \min \left\{ \sum\limits_{k_1} |S_{k k_1}|, \sum\limits_{k_2} |S_{k k_2}|, \sum\limits_{k_3} |S_{k k_3}| \right\},\\[10pt]
    |S_{k_1}| \leq \min \left\{ \sum\limits_{k_2} |S_{k_1 k_2}|, \sum\limits_{k_3} |S_{k_1 k_3}| \right\}, \\[10pt]
    |S_{k_2}| \leq \min \left\{ \sum\limits_{k_1} |S_{k_2 k_1}|, \sum\limits_{k_3} |S_{k_2 k_3}| \right\}, \\[10pt]
    |S_{k_3}| \leq \min \left\{ \sum\limits_{k_1} |S_{k_3 k_1}|, \sum\limits_{k_2} |S_{k_3 k_2}| \right\}. 
\end{cases}
\end{equation}

The computations are similar so we only show one of them:

\begin{equation}
\begin{aligned}
    |S_k| \leq \sum\limits_{k_1} |S_{k k_1}|
    & \lesssim \sum\limits_{k_1} \left(\frac{(N_2 \wedge N_3)^{3-\alpha}}{\langle k - k_1 \rangle} + (N_2 \wedge N_3)\right)\\
    & \leq \sum\limits_{k_1} \frac{(N_2 \wedge N_3)^{3-\alpha}}{\langle k^{(1)} - k_1^{(1)} \rangle} + \sum\limits_{k_1}(N_2 \wedge N_3)\\
    & \leq (N_2 \wedge N_3)^{3-\alpha} N_1 \log N_1 + N_1^2 (N_2 \wedge N_3),
\end{aligned}
\end{equation}
\noi
where $k^{(1)}$ and $k_1^{(1)}$ represent the first coordinates of $k,k_1\in\mathbb{Z}^2$. 
    
\end{proof}

We also have the four-vector counting, i.e., the counting estimates for the whole set $S$ in \eqref{Eqn:S}.

\begin{lem}[Four-vector counting]
\label{4vec1} Let $1\leq N_j\leq N$ for $j\in\{1,2,3\}$ and $\alpha\in(1,2)$, we have the following counting estimate:

\begin{equation}
\begin{aligned}
    |S|\lesssim \min \Big\{ &N_3^{3-\alpha} (N_1 \wedge N_2)^2 (N_1 \vee N_2) \log(N_1 \vee N_2) + N_1^2 N_2^2 N_3,\,\, (N_1\wedge N_3)^{\frac{8}{3} - \alpha}N_2^2N^2, \\
 &N_1^{3-\alpha} (N_2 \wedge N_3)^2 (N_2 \vee N_3) \log(N_2 \vee N_3) + N_2^2 N_3^2 N_1,\, \, N_1^2 N_3^2 N_2^{\frac{8}{3} - \alpha} \Big\}.
\end{aligned}
\end{equation}

\end{lem}

\begin{proof}

Just notice that 

\begin{equation}
     |S| \leq \min \left\{ \sum\limits_{k_1, k_2} |S_{k_1 k_2}|, \sum\limits_{k, k_2} |S_{k k_2}|,\sum\limits_{k_1, k_3} |S_{k_1 k_3}|, \sum\limits_{k_2, k_3} |S_{k_2 k_3}| \right\}.
\end{equation}

The computation is the same as above.

\end{proof}

\subsubsection{$\Gamma-$condition counting}

Another structure which is important to our counting is the so-called $\Gamma$-condition. 
To be more precise,
if there exists a positive number $\Gamma$ such that 
\[ 
|k_{\max}| \le \Gamma < |k| \quad \text{or} \quad |k|\leq \Gamma < |k_{\max}|,
\]

\noindent 
where $k_{\max}$ is the frequency corresponding to $N_{\max}$,
then we say that $S$ given in \eqref{Eqn:S} satisfies the $\Gamma$-condition. To simplify the notation, let
\begin{align} 
\label{gammaB}
B_\Gamma = \{ (k,k_1,k_2,k_3) \in S: |k_{\max}| \le \Gamma < |k| \quad \text{or} \quad |k|\leq \Gamma < |k_{\max}|\}
\end{align}

\noindent 
and we denote $S^{\Gamma} = B_{\Gamma}\cap S$.

With the $\Gamma$-condition, we can improve the previous counting estimates. To simplify the discussion, we only consider the case that

\begin{equation}\label{Gamma assumption0}
    N_1\gg N_2 \vee N_3 ,\quad  \Gamma \sim N\sim |k| \sim |k_1|,
\end{equation}
\noi
which will be put in the remainder. See Section \ref{Gammacondi} for the details. Now we state two types of $\Gamma$-condition counting that we need.

\begin{lem}[$\Gamma$-condition counting I]
\label{Gammacounting}
 Suppose that $|m|\lesssim N^{\alpha}_{\med}$ and recall \eqref{parameters1},\eqref{Gamma assumption0}, then we have

\begin{itemize}
    \item 2-vector counting:

\begin{equation}
\begin{aligned}
        &|S^\Gamma_{k_{2}k_{3}}|\lesssim N^{\alpha}_{\med}N^{2-\alpha+\theta}, \\
        &|S^\Gamma_{k k_2}|\lesssim N_3^{\frac{8}{3}-\alpha} , \, |S^\Gamma_{k_1 k_2}|\lesssim N_3 ,\\
        & |S^\Gamma_{k_1k_3}|\lesssim N_2^{\frac{8}{3}-\alpha} ,  \,  |S^\Gamma_{k k_3|}\lesssim N_2, 
\end{aligned}
\end{equation}

    \item 3-vector counting:

\begin{equation}
\begin{aligned}
&|S^\Gamma_k|\lesssim (N_3^{\frac{8}{3}-\alpha} N^2_2) \wedge (N_2 N_3^2), \\
&|S^\Gamma_{k_1}|\lesssim (N_2^{\frac{8}{3}-\alpha} N^3_2) \wedge (N_3 N_2^2),\\
&|S^\Gamma_{k_2}| \lesssim N^{\frac{8}{3}-\alpha}_{3}N_{\med}^{\alpha} N^{2-\alpha+\theta},\\
&|S^{\Gamma} _{k_3}| \lesssim N_{2}^{\frac{8}{3}-\alpha}N_{\med}^{\alpha} N^{2-\alpha+\theta}.
\end{aligned}
\end{equation}
    \item 4-vector counting:

\begin{equation}
|S^{\Gamma}| \lesssim  N_{\min}^{2} N_{\med}^{\frac{8}{3}} N^{2-\alpha+\theta}.
\end{equation}
    
\end{itemize}

\end{lem}

\begin{proof}

The bounds of $|S^\Gamma_{k k_{2}}|$, $|S^\Gamma_{k_1 k_{2}}|$, $|S^\Gamma_{k k_{3}}|$, $|S^\Gamma_{k_1 k_{3}}|$ are simple by Lemma \ref{2vec1} and Lemma \ref{2vec2} since we have $|b|\gtrsim|k|\gg N_{\med}>N_{\min}$ now. So we only need to show the bound of $|S^\Gamma_{k_{2}k_{3}}|$, which comes from the counting of $k$ or $k_1$ under $\Gamma$-condition and the assumption \eqref{Gamma assumption0}. Note that

\[
\begin{cases}
    \Big| |k|^{\alpha}-|k_1|^{\alpha} \Big|\lesssim N^{\alpha}_{\med},\\
    |k|^{\alpha}>\Gamma^{\alpha}\geq |k_1|^{\alpha}
\end{cases}
\]
\noi
implies 

\[
    \Big| |k|^{\alpha}-\Gamma^{\alpha} \Big|\lesssim N^{\alpha}_{\med} \, , \quad \Big| |k_1|^{\alpha}-\Gamma^{\alpha} \Big|\lesssim N^{\alpha}_{\med}\, .
\]

Consider the function $f(x) = |x|^{\frac{2}{\alpha}}$ and apply mean value theorem to get

\[
    \Big| |k|^2-\Gamma^2 \Big| =\Big |f(|k|^\alpha)-f(\Gamma^\alpha)\Big|  \lesssim |f'(\xi_0)|\cdot \Big| |k|^{\alpha}-\Gamma^{\alpha} \Big| ,
\]
\noi
where $\xi_0$ is located in the interval with endpoints $|k|^\alpha$ and $\Gamma^\alpha$ . Hence $|\xi_0|\lesssim N^{\alpha}$ and 

\[
 \Big| |k|^2-\Gamma^2 \Big|\lesssim N^{\alpha}_{\med} N^{2-\alpha}.
\]

Similarly we have $ \Big| |k_1|^2-\Gamma^2 \Big|\lesssim N^{\alpha}_{\med} N^{2-\alpha}$. We call this `` condition $(\Gamma_1)$''. Note that the above estimates imply that condition $(\Gamma_1)$ holds directly under the restrictions in $S^{\Gamma}$ as well as the assumption \eqref{Gamma assumption0}. Then the choice of the integer $|k|^2$ or $|k_1|^2$ is at most $\mathcal{O}(N^{\alpha}_{\med} N^{2-\alpha})$. Therefore, the choice of $k$ or $k_1$ is at most $\mathcal{O}(N^{\alpha}_{\med} N^{2-\alpha+\theta})$ for any small number $\theta>0$ by divisor counting. (Notice that the vector $k$ can be viewed as a complex number and $|k|^2=k\overline{k}$.) A similar argument can be found in \cite[Proposition 4.5]{DNY2024}. This gives the stated bound for $|S^\Gamma_{k_{2}k_{3}}|$.

Regarding 3-vector and 4-vector counting, we use

\begin{equation}
\begin{cases}
    |S^\Gamma_k|\leq \min\left\{ \sum\limits_{k_2} |S^\Gamma_{k k_{2}}|,\,\sum\limits_{k_{3}} |S^\Gamma_{k k_{3}}|\right\}, \, \hspace{1cm}|S^\Gamma_{k_1}|\leq \min\left\{ \sum\limits_{k_{2}}|S^\Gamma_{k_1 k_{2}}|, \,\sum\limits_{k_{3}}|S^\Gamma_{k_1 k_{3}}|\right\},\\[10pt]
    |S^\Gamma_{k_{2}}|\leq \sum\limits_{\substack{k\\ \text{condition}\,(\Gamma_1)}} |S^\Gamma_{k k_{2}}|, \, \hspace{2.5cm}  |S^\Gamma_{k_{3}}|\leq \sum\limits_{\substack{k_1\\ \text{condition}\,(\Gamma_1)}}|S^\Gamma_{k_1 k_{3}}|, \\[10pt]
    |S^\Gamma|\leq \min\left\{ \sum\limits_{\substack{k,k_{2}\\ \text{condition}\,(\Gamma_1)}}|S^\Gamma_{k k_{2}}|, \, \sum\limits_{\substack{k_1,k_{3}\\ \text{condition}\,(\Gamma_1)}}|S^\Gamma_{k_1 k_{3}}|\right\}
\end{cases}
\end{equation}

\noindent to get desired estimates. Now we finish the proof.

\end{proof}

\begin{rem}\label{interpolationbounds}
    Since we have two bounds for $|S_k^{\Gamma}|$, interpolation tells that 
\begin{equation}
    |S_k^{\Gamma}|\lesssim (N_2N_3^2)^{\lambda_\alpha}\cdot(N_3^{\frac{8}{3}-\alpha}N_2^2)^{1-\lambda_\alpha} = N_2^{2-\lambda_\alpha}N_3^{\frac{8}{3}+\lambda_\alpha(\alpha-\frac{2}{3})}
\end{equation}

\noi
holds for every $\lambda_\alpha\in[0,1]$. Similarly, we can interpolate two bounds of $|S_{k_1}^{\Gamma}|$.
\end{rem}

To deal with the situation $|m|\gg N^{\alpha}_{\med}$, we need to count the variant of $S^{\Gamma}$ as follows:

\begin{equation}
    \label{Eqn:Sr} 
    S^r = \left\{(k, k_1, k_2, k_3) \in (\mathbb{Z}^2)^4 :\quad
    \begin{aligned} 
      & k_1 - k_2 + k_3-k = 0 , \quad k_2 \not\in \{ k_1, k_3 \}, \\
      & \Big| |{k_1}|^{\alpha} - |{k_2}|^{\alpha} + |{k_3}|^{\alpha} - |{k}|^{\alpha}
      \Big| \in (2^{r-1},2^r], \\
      & | k | \le  N, \,\, | k_j | \le  N_j\, \textup{ for } j \in \{ 1, 2, 3
      \} 
    \end{aligned}
    \right\}\cap B_{\Gamma}
  \end{equation}

\noindent where $r\in \Z_{\geq0}$ and $N^{\alpha}_{\med}\ll 2^r \lesssim N^{\alpha}$. Still we assume \eqref{Gamma assumption0}, then the following counting estimates hold:

\begin{lem}[$\Gamma$-condition counting II]
\label{Gammacounting2}
Suppose that $N_{\med}^{\alpha}\lesssim 2^r\lesssim N^\alpha$ and recall \eqref{parameters1}, \eqref{Gamma assumption0}, then we have

\begin{itemize}
    \item 2-vector counting:

\begin{equation}
\begin{aligned}
        &|S^r_{k_{2}k_{3}}|\lesssim 2^r\cdot N^{2-\alpha+\theta}, \\
        &|S^r_{k k_{2}}|\lesssim (2^r\cdot N_3^{\frac{8}{3}-\alpha})\wedge N_3^2 , \, |S^r_{k_1 k_{2}}|\lesssim (2^r\cdot N_{3})\wedge N^2_{3} ,\\
        &|S^r_{k_1k_{3}}|\lesssim (2^r\cdot N_{2}^{\frac{8}{3}-\alpha})\wedge N^2_{2},\,  |S^r_{k k_{3}}|\lesssim (2^r\cdot N_2)\wedge N_2^2   . 
\end{aligned}
\end{equation}

    \item 3-vector counting:

\begin{equation}
\begin{aligned}
  &|S^r_k|\lesssim (2^r\cdot N_2 N_3^2)\wedge(2^r \cdot N_3^{\frac{8}{3}-\alpha}N_2^2),\\
  &|S^r_{k_1}|\lesssim (2^r\cdot N_{3} N^2_{2})\wedge (2^r\cdot N_2^{\frac{8}{3}-\alpha}N_3^2),\\
  &|S^r_{k_{2}}| \lesssim 2^r\cdot N^{2}_{3} N^{2-\alpha+\theta},\\
  &|S^{r} _{k_{3}}| \lesssim 2^r\cdot N^{2}_{2}  N^{2-\alpha+\theta}.
\end{aligned}
\end{equation}

    \item 4-vector counting:

\begin{equation}
|S^{r}| \lesssim 2^r\cdot N_{2}^{2} N_{3}^2 N^{2-\alpha+\theta}.
\end{equation}

\end{itemize}

\end{lem}

\begin{proof}
    The counting for $k$ or $k_1$ is similar to that in Lemma\ \ref{Gammacounting} with $N^{\alpha}_{\med}$ replaced by $2^r$. Hence we have the bound for $|S^r_{k_{2}k_{3}}|$. 

    The second bounds of $|S^r_{k k_{2}}|$ etc. are trivial counting. While the first bound can be checked in the same way as Lemma \ref{2vec1} and Lemma \ref{2vec2}  by noticing that the length of the interval $J_0$ turns to $2^r$ in this setting. As for 3-vector and 4-vector counting, we use

\begin{equation}
\begin{cases}
    |S^r_k|\leq \min\left\{ \sum\limits_{k_2} |S^r_{k k_{2}}|,\,\sum\limits_{k_{3}} |S^r_{k k_{3}}|\right\}, \, \hspace{0.5cm}|S^r_{k_1}|\leq \min\left\{ \sum\limits_{k_{2}}|S^r_{k_1 k_{2}}|, \,\sum\limits_{k_{3}}|S^r_{k_1 k_{3}}|\right\},\\[10pt]
    |S^r_{k_{2}}|\leq \sum\limits_{k_3} |S^r_{k_2 k_{3}}|, \, \hspace{3.3cm}  |S^r_{k_{3}}|\leq \sum\limits_{k_2}|S^r_{k_2 k_{3}}|, \\[10pt]
    |S^r|\leq  \sum\limits_{k_2,k_{3}}|S^r_{k_2 k_{3}}|.
\end{cases}
\end{equation}

So we finish the proof.

\end{proof}

\begin{rem}
    When applying these estimates, we may encounter $|S_{kk_{2}}||S_{k_1k_{3}}|$ or $|S_{kk_1}||S_{k_{2}k_{3}}|$. To avoid multiple $2^r$, we should always use trivial counting for one of the two sets. To be precise, we will use $|S_{kk_{2}}||S_{k_1k_{3}}|\lesssim |S_{kk_{\min}}|\cdot N_{\min}^2$ and $|S_{kk_1}||S_{k_{2}k_{3}}|\lesssim N_{\min}^2\cdot|S_{k_{2}k_{3}}|$. 
\end{rem}

\subsection{Tensor estimates}

First, we introduce several base tensors. Given $m\in \Z$, $\alpha \in (1,2)$, and dyadic numbers $1 \le N_j \le N$, $j\in\{1,2,3\}$,
we define the base tensor ${\rm T}^{{\rm b},m}$ as

\begin{equation}
        \label{baseT}
    {\rm T}^{{\rm b},m} = {\rm T}^{{\rm b},m}_{kk_1k_2k_3} = \ind_{S} (k,k_1,k_2,k_3),
\end{equation}

\noindent 
where $\ind_S$ is the indicator function, and $S$ is the subset of $(\Z^2)^4$ given in \eqref{Eqn:S}.
Here,  we will not distinguish the tensor operator ${\rm T}^{{\rm b},m}$ and its kernel ${\rm T}^{{\rm b},m}_{kk_1k_2k_3}$.
Similarly, by \eqref{Eqn:Sr} we can define

\begin{equation}\label{baseTr}
    {\rm T}^{{\rm b},r} ={\rm T}^{{\rm b},r}_{kk_1k_2k_3} =\ind_{S^r} (k,k_1,k_2,k_3).
\end{equation}

\noindent 

To connect the tensor norms with our counting, we need the following lemma:

\begin{lem}\label{schurtest}
For tensors defined above, we have the following estimates:

\begin{equation}
\begin{aligned}
    &\| {\rm T}^{{\rm b},m}_{k k_1 k_2 k_3} \|^2_{kk_B
      \to  k_C} \lesssim|S_{k k_B}||S_{k_C}|,\\
      &\| {\rm T}^{{\rm b},m}_{k k_1 k_2 k_3}\ind_{B_{\Gamma}} \|^2_{k k_B
      \to  k_C}  \lesssim |S^\Gamma_{k k_B}||S^\Gamma_{k_C}|,\\
    &\| {\rm T}^{{\rm b},r}_{k k_1 k_2 k_3} \|^2_{k k_B
      \to  k_C}  \lesssim |S^r_{k k_B}||S^r_{k_C}|.
\end{aligned}
\end{equation}

\end{lem}

\begin{proof}

By Schur's test, we have 
\[
       \| {\rm T}^{{\rm b},m}_{k k_1 k_2 k_3} \|^2_{kk_B
      \to  k_C} \lesssim \bigg( \sup_{k, k_B \in (\mathbb{Z}^2)^{|B|+1}}
      \sum_{k_C \in (\mathbb{Z}^2)^{|C|}} | {\rm T}^{{\rm b},m}_{k k_1 k_2
      k_3} | \bigg) \times \bigg( \sup_{k_C \in (\mathbb{Z}^2)^{|C|}}
      \sum_{k, k_B \in (\mathbb{Z}^2)^{|B|+1}} | {\rm T}^{{\rm b},m}_{k k_1 k_2
      k_3} | \bigg) ,
\]
\noi
which gives the desired bound. Other tensor norms can be controlled in the same way.

\end{proof}

\medskip

\subsection{Random tensor estimates}

In this section, we state the ordinary random tensor estimate, whose proof can be found in \cite{DNY2022,Wang2024}.

\begin{prop}[{{\cite[Proposition 4.14]{DNY2022}}}]
  \label{Prop:tensor} 
  Let $A$ be a finite set and $h_{b c k_A} = h_{b c k_A}
  (\omega)$ be a tensor, where each $k_j$ and $(b, c) \in (\mathbb{Z}^2)^q$ \ for
  some integer $q \ge  2$. Given signs $\zeta_j \in \{ \pm \}$, we also
  assume that $\jb{b-b_0}, \jb{c-c_0} \lesssim M$ and $\jb{k_j } \lesssim M$ for some fixed $b_0, c_0 \in \Z^2$ and all $j \in A$, where $M$ is a dyadic number, and that in support of $h_{b c k_A}$
  there is no pairing in $k_A$. Define the tensor
  \begin{equation}
    \label{Eqn:Hbc} 
    H_{b c} = \sum_{k_A} h_{b c k_A} \prod_{j \in A}
    \eta^{\zeta_j}_{k_j},
  \end{equation}
  where we restrict $k_j \in E$ in \eqref{Eqn:Hbc}, $E\subseteq\Z^2$ being a finite set
  such that $\{ h_{b c k_A} \}$ is independent with $\{ \eta_k : k \in E \}$, and $\eta_{k_j} \in \{ g_{k_j} , |g_{k_j}|^2-1 \}$.  
  Then $\tau^{-1}M$-certainly, we have
  \[ 
    \| H_{b c} \|_{b \to  c} \lesssim \tau^{- \theta} M^{\theta} \cdot
    \max_{(B, C)} \| h \|_{b k_B \to  c k_C}, 
  \]
  for some $\tau\ll 1$,
  where $(B, C)$ runs over all partitions of $A$.
\end{prop}

\begin{rem} 
\label{RMK:absorb}
In what follows, we may omit the constant $\tau^{-\theta} M^\theta$, as these constants may be absorbed most of the time by choosing $\theta \ll 1$ in our analysis. See \cite{DNY2022,DNY2024}  for similar arguments. However, there do exist situations in which we cannot absorb $M^{\theta}$ even if $\theta\ll1$. For example, the operator norm with low frequency $L$ does not allow any high frequency increase.  This drives us to seek a variant of the above proposition under specific conditions, which is similar to \cite[Proposition 4.15]{DNY2022}.
\end{rem}

\begin{lem}[A special random tensor estimate]
\label{HLLtensor}
    Consider the same setting as in Proposition $\ref{Prop:tensor}$, with the following differences: (1) we only restrict $\langle k_j\rangle\lesssim L$ but do not impose any condition on $\langle b\rangle$ or $\langle c\rangle$; (2) we assume that $b,c\in\Z^2$, and that in the support of $h_{bck_A}$ we have $|b-\zeta c|\lesssim L$ where $\zeta\in\{\pm\}$; (3) the tensor $h_{bck_A}$ only depends on $b-\zeta c$, $|b|^\alpha-\zeta |c|^\alpha$ and $k_A$, and is supported in the set where $||b|^\alpha-\zeta|c|^\alpha|\leq L^{C_1}$;(4) a $\mathcal{O}(L^{-C_2})$ perturbation on $|b|^\alpha-\zeta |c|^\alpha$ does not affect the norm $\|h\|_{bk_B\to ck_C}$. Here, $C_1$, $C_2$ are constants that we can choose to be large enough. The other assumptions are the same. Then $\tau^{-1}L$-certainly we have
 
  \begin{equation}\label{HLLthisbound}
  \|H_{bc}\|_{b\to c}\lesssim \tau^{-\theta}L^\theta\cdot\sup_{(B,C)}\|h\|_{bk_B\to ck_C}.
  \end{equation}
\end{lem}

\begin{proof}

We may assume $\zeta=+$, since the other case is easier. Since $h_{bck_A}$ is supported in the set where $|b-c|\lesssim L$, by an orthogonality argument we may restrict $h$ to the set where $|b-f|\lesssim L$ and $|c-f|\lesssim L$, and to bound $\|H_{bc}\|_{b\to c}$ it suffices to bound these restricted operators uniformly in $f$.

For any $f\in\Z^2$, let $x=b-f$ and $y=c-f$, then $x$ and $y$ are both assumed to have size $\lesssim L$, and it suffices to estimate the norm

\[\|\widetilde{H}_{f;xy}\|_{x\to y},\quad\mathrm{where}\quad \widetilde{H}_{f;xy}=\sum_{k_A}\widetilde{h}_{f;xyk_A}\prod_{j\in A}\eta_{k_j}^{\zeta_j},\quad\mathrm{and}\quad\widetilde{h}_{f;xyk_A}:=h_{x+f,y+f,k_A}\cdot\mathbf{1}_{|x|,|y|\lesssim L}.\] 

For any fixed $f$, we may apply Proposition \ref{Prop:tensor} to get that $\tau^{-1}L$-certainly

\begin{equation}\label{thisbound}
\|\widetilde{H}_{f;xy}\|_{x\to y}\lesssim \tau^{-\theta}L^\theta \cdot \sup_{(B,C)}\|\widetilde{h}_{f;xyk_A}\|_{xk_B\to yk_C}\leq \tau^{-\theta}L^\theta\cdot \sup_{(B,C)}\|h\|_{bk_B\to ck_C},
\end{equation} 
so it suffices to establish (\ref{thisbound}) uniformly in $f$. By our assumption, $\widetilde{h}_{f;xyk_A}$ only depends on $(x,y,k_A)$ and $|x+f|^\alpha-|y+f|^{\alpha}$ and 
 $\mathcal{O}(L^{-C_2})$ perturbation on $|b|^\alpha-\zeta |c|^\alpha$ is acceptable. Hence we should analyze the number of choices of $f$ such that $|x+f|^\alpha-|y+f|^{\alpha}$ (under perturbation) and its domain have less than $\mathcal{O}(L^{C_3})$ types.

 First of all, we may assume $|f|\gtrsim L^{\frac{C_2}{2-\alpha}}$ since we may use Proposition \ref{Prop:tensor} $\mathcal{O}(L^{\frac{2C_2}{2-\alpha}})$ times and still get the bound $\tau^{-1}L$-certainly. Now denote 

 \begin{equation}\label{taylorremainder}
     R_{\alpha}(x_1,x_2) = |x_1+x_2|^\alpha-|x_2|^\alpha-\alpha|x_1|^{\alpha-2}x_1\cdot x_2 ,\quad x_1,x_2\in\Z^2,
 \end{equation}
\noi
 then we have $|x+f|^\alpha-|y+f|^{\alpha} = \alpha|y+f|^{\alpha-2}(y+f)\cdot (x-y) + R_\alpha(y+f,x-y)$, where $|y+f|\sim |f| \gg L \gtrsim|x-y|$.  To give an estimate of $R_{\alpha}(x_1,x_2)$ when $|x_1|\gg |x_2|$, we write

 \[z = \frac{2x_1 \cdot x_2 + \vert x_2 \vert^2}{\vert x_1 \vert^2} \in \mathbb{R}, \quad \vert z \vert \leqslant \frac{2\vert x_2 \vert}{\vert x_1 \vert} + \left( \frac{\vert x_2 \vert}{\vert x_1 \vert} \right)^2 \ll 1\]
 \noi
 and 

 \begin{equation}\label{taylorexpand}
 \begin{aligned}
      \vert x_1 + x_2 \vert^\alpha &= \left( \vert x_1 \vert^2 + 2x_1 \cdot x_2 + \vert x_2 \vert^2 \right)^{\frac{\alpha}{2}} = \vert x_1 \vert^\alpha \left( 1 + \frac{2x_1 \cdot x_2 + \vert x_2 \vert^2}{\vert x_1 \vert^2} \right)^{\frac{\alpha}{2}} = \vert x_1 \vert^\alpha(1 + z)^{\frac{\alpha}{2}}\\
      & = \vert x_1 \vert^\alpha + \alpha \vert x_1 \vert^{\alpha - 2} x_1 \cdot x_2 + \frac{\alpha}{2} \frac{\vert x_2 \vert^2}{\vert x_1 \vert^{2-\alpha} } + \frac{\alpha (\alpha -2)}{4} (1 - \theta_0) (1 + \theta_0 z)^{\frac{\alpha}{2} - 2} z^2|x_1|^\alpha     
 \end{aligned}
 \end{equation}
\noi
where we can use the Taylor expansion with Cauchy remainder. By \eqref{taylorremainder} and \eqref{taylorexpand}, we have

\[
    |R_{\alpha}(x_1,x_2)| \leq \frac{\alpha}{2} \cdot \frac{\vert x_2 \vert^2}{\vert x_1 \vert^{2 - \alpha}} + \frac{\alpha (2 - \alpha)}{4} \cdot \frac{1 - \theta_0}{1 + \theta_0 z} (1 + \theta_0 z)^{\frac{\alpha-2}{2}} z^2|x_1|^\alpha  \lesssim \,\, \frac{\vert x_2 \vert^2}{\vert x_1 \vert^{2 - \alpha}},
\]
\noi
and hence we can approximate

\begin{equation}\label{alexpand1}
 |x+f|^\alpha-|y+f|^{\alpha} = \alpha|y+f|^{\alpha-2}(y+f)\cdot (x-y) + \mathcal{O}(L^{2-C_2}).
\end{equation}
Then consider the function $G(x)=x|x|^{\alpha-2}$ as in the proof of Lemma \ref{2vec1} and the path $ \gamma_f(t) = f + ty, \, t\in[0,1].$ We know that $|\gamma_f(t)| \sim |f|$ and $\alpha DG(x)=H(|x|^\alpha)$ defined in \eqref{eq:H2}. Therefore

\[
    \Big|G(\gamma_f(1)) - G(\gamma_f(0))\Big| = \left|\int_0^1 (DG)(\gamma_f(t)) \, y \, dt\right| \lesssim |f|^{\alpha - 2} |y| \lesssim L^{1 - C_2},
\]
\noi
and we can write

\begin{equation}\label{alexpand2}
     |x + f|^\alpha - |y + f|^\alpha = \alpha |f|^{\alpha - 2} f \cdot (x - y) + \mathcal{O}(L^{2 - C_2}).
\end{equation}

Now we only need to consider the main term and its domain. Actually, we have the following claim:

\begin{claim}\label{maintype}
 Let $F_{f}(m) = \alpha |f|^{\alpha - 2} f \cdot m$, $Dom(F_f) =\{m\in\Z^2 : |m|\lesssim L, \, |F_{f}(m)|\lesssim L^{C_1}\}$. There exist $A_0\lesssim L^{\frac{c C^3_1}{\alpha-1}} $ and $\{(F_{f_j}, Dom(F_{f_j}))\}_{j=1}^{A_0}$ such that for any $f\in\Z^2$, there exists $j\in\{1,2,\dots,A_0\}$ such that
 \begin{center}
    $ Dom(F_f) = Dom(F_{f_j})$ , $F_f = F_{f_j} + \mathcal{O}(L^{-C^2_1})$ on $Dom(F_f)$.
    \end{center}
\end{claim}

\begin{proof}[Proof of Claim \ref{maintype}]

First suppose that \(Dom(F_f)\) contains three non-collinear points
\(q_0,q_1,q_2\). There are at most \(\mathcal{O}(L^6)\) choices, so we may fix
them. Since \(y_j:=q_j-q_0\in\mathbb Z^2\), \(j=1,2\), are linearly
independent, we can write

\[
     \alpha |f|^{\alpha - 2} f=: x_f = \lambda_1 y_1 + \lambda_2 y_2
\]
\noi
where
\[
    \lambda_1 = \frac{ (x_f \cdot y_1) |y_2|^2 - (x_f \cdot y_2) (y_1 \cdot y_2) }{ |y_1|^2 |y_2|^2 - |y_1 \cdot y_2|^2 },\quad \lambda_2 = \frac{ (x_f \cdot y_2) |y_1|^2 - (x_f \cdot y_1) (y_1 \cdot y_2) }{ |y_1|^2 |y_2|^2 - |y_1 \cdot y_2|^2 }.
\]
Note that $|y_1|^2 |y_2|^2 - |y_1 \cdot y_2|^2\geq 1$, and for $j=1,2$,
\[
    |x_f\cdot y_j|=\left| \alpha |f|^{\alpha - 2} f \cdot (q_j - q_0) \right| \leq |F_f(q_1)| + |F_f(q_0)| \lesssim L^{C_1},
\]
\noi
we have $ |x_f| \leq |\lambda_1|  |y_1| + |\lambda_2| |y_2| \lesssim L^{C_1 +3} $ and thus

\begin{equation}
    |f|\lesssim L^{\frac{C_1 +3}{\alpha-1}}.
\end{equation}

\noindent Since $f\in\Z^2$, there are at most $\mathcal{O}(L^{\frac{2C_1+6}{\alpha-1}})$ possibilities for this case.

\smallskip

Secondly, assume that \(Dom(F_f)\) is contained in a line. Empty domains
give one type. For a singleton domain, there are \(O(L^2)\) choices of its
point, and discretizing the single value \(F_f(q_0)\), whose magnitude is
at most \(L^{C_1}\), at mesh \(L^{-C_1^2}\) gives only polynomially many
types. We may therefore assume \(|Dom(F_f)|\geq2\). Fix two points
\(q_0,q_1\), put
\(q_1-q_0=(n_1,n_2)\), \(d=\gcd(n_1,n_2)\), and
\[
 k=\left(\frac{n_1}{d},\frac{n_2}{d}\right).
\]
Every lattice point on this line has the form \(m=q_0+\sigma k\) with
\(\sigma\in\mathbb Z\). Indeed, choose integers \(a_1,a_2\) with
\(a_1k_1+a_2k_2=1\); then
\(\sigma=(a_1,a_2)\cdot(m-q_0)\in\mathbb Z\). Therefore, we can write

\begin{equation}
\begin{cases}
        F_f(m) = a_f\sigma + b_f, \\
        Dom(F_f) = \{q_0+\sigma k : \sigma\in\Z,\, |q_0+\sigma k|\lesssim L, \, |a_f\sigma + b_f|\lesssim L^{C_1}\}.
\end{cases}
\end{equation}
\noi
where $|a_f| = \left|\alpha |f|^{\alpha - 2} f\cdot k\right|\lesssim L^{C_1}$ and $|b_f| = \left|\alpha |f|^{\alpha - 2} f\cdot q_0\right|\lesssim L^{C_1}$. Taking this form into account, we define the intermediate approximation functions as follows:

Take $\{a_j\}_{j=1}^{A_1}$, $\{b_l\}_{l=1}^{B_1} \subseteq \mathbb{R}$, such that for any $a,b\in\mathbb{R}$ with $|a|$, $|b|\lesssim L^{C_1}$ , there exist $ a_j$, $b_l$ with $|a-a_j|\leq L^{-C^2_1-1}$ and $|b-b_l|\leq L^{-C^2_1}$ . We may choose these nets with
\[
A_1\lesssim L^{C_1^2+C_1+1},
\qquad
B_1\lesssim L^{C_1^2+C_1}.
\] Then we define

\begin{equation}
\begin{cases}
    F^{q_0,k}_{j,l}(m) = a_j\sigma + b_l, \\
    Dom(F^{q_0,k}_{j,l}) = \{q_0+\sigma k : \sigma\in\Z,\, |q_0+\sigma k|\lesssim L, \, |a_j\sigma + b_l|\lesssim L^{C_1}\}.
\end{cases}
\end{equation}
Once $q_0,k$ are fixed, the geometric condition
$|q_0+\sigma k|\lesssim L$ restricts $\sigma$ to an interval, and the additional affine inequality cuts out a subinterval containing $0$. Then the domain is essentially a sequence of successive integers $\sigma$. Thus there are only $\mathcal{O}(L^2)$ possible domains for each
quadruple $(q_0,k,j,l)$. We include all of these domains as
variants of $F^{q_0,k}_{j,l}$. The resulting total number of approximation
functions with domains is still
$\mathcal{O}(L^{\frac{cC_1^3}{\alpha-1}})$ for some absolute constant $c$.

Now return to the proof. For any $f\in\Z^2$ for which $Dom(F_f)$ is
contained in a line, and with $q_0,q_1$ fixed, choose $j,l$ so that
$|a_f-a_j|\leq L^{-C_1^2-1}$ and
$|b_f-b_l|\leq L^{-C_1^2}$. On the whole geometric $\sigma$-interval,
\[
|(a_f-a_j)\sigma+(b_f-b_l)|\lesssim L^{-C_1^2}.
\]
We assign to $(j,l)$ the interval variant whose domain is exactly
$Dom(F_f)$. Therefore

\[
\{ f \in \mathbb{Z}^2 : Dom(F_f)\,\, \text{is contained in a line} \} = \bigcup_{\substack{1 \leq j \leq A_1 ,\, 1 \leq l \leq B_1,\\ |a_j|, |b_l| \leq L,\, |q_0|,|q_1|\lesssim L}} E_{j,l}^{q_0,k}
\]
where

\[
E_{j,l}^{q_0,k} = \{ f \in \mathbb{Z}^2 : F_f = F_{j,l}^{q_0,k} + \mathcal{O}(L^{-C_1^2}),\, Dom(F_f) = Dom(F_{j,l}^{q_0,k})\},
\]
with the interval variants included in the indices on the right-hand side. Now for any $f\in E_{j,l}^{q_0,k}$, the functions $F_f$ share the same domain and admit $O(L^{-C_1^2})$ errors, so we can take one $f_{j,l,q_0,k}$ as representative and the claim is proven.    
\end{proof}

By the claim, the pairs consisting of the admissible domain and the phase
function fall into at most \(A_0=L^{\mathcal{O}(\frac{C_1^3}{\alpha-1})}\) classes, up to
an \(\mathcal{O}(L^{-C_1^2}+L^{2-C_2})\) phase error. Choose
\(C_2>C_1^2+3\). Assumption (4) then makes the deterministic tensor norm
constant on each class. Applying \eqref{thisbound} to one representative
of every class and taking a union bound costs only a fixed polynomial in
\(L\), which is absorbed by the exponential decay of the probability of exceptional set. This proves \eqref{HLLthisbound} uniformly in \(f\).

\end{proof}

\begin{rem}
    The proposition is designed for the operator norm of high-low-low random averaging operator. Reconsider the beginning part of the proof, we can get the estimate \eqref{HLLthisbound} for $(b,c)\in (\Z^2)^q$ , $q\geq2$ as long as there is only one "high-frequency" term in both $b$ and $c$. In fact, we can do the shift $f$ for the two high terms in $\Z^2$ and leave those "low-frequency" terms unchanged. Then we get the desired estimates. We need the situation when we analyze the $(C,D)$ or $(D,C)$ case.
\end{rem}

\smallskip

\section{Bilinear and trilinear estimates}\label{esti} 

\subsection{Preparation}

In this subsection, we record some basic lemmas from \cite{DNY2022}. Define the original and truncated integral
operators
\begin{align}
    \label{Duhamel_I}
  \mathcal{J} v (t) = \int^t_0 v (t') d t', \quad
  \mathcal{J}_{\eta} v (t) = \eta (t) \int^t_0 \eta (t') v (t') d t', 
\end{align}  

\noi 
where $\eta(t)$ is a smooth cutoff function, such that $\eta(t) = 1$ for $|t|\le 1$; and $\eta(t) = 0$ for $|t|\ge 2$. 

\begin{rem}\label{abuseXb}
Notice that the Duhamel operator $\mathcal{I}$ from \eqref{Duhamel} roughly reduces to the integral operator $\mathcal J$ defined in \eqref{Duhamel_I} under the transformation $u (t)  \to
e^{i  t {\rm D}^\alpha} u (t)$. 
See \cite[Section 4.1]{DNY2022} for further details. In Fourier space, the operator $e^{it{\rm D}^\alpha}$ results in a $|k|^\alpha$ shift in time, which is considered in the definition of $X^b$ space in the form of $\widetilde{u}$ . 
\end{rem}

We have the following kernel estimates. See \cite[Lemma 4.1]{DNY2022} for a proof.

\begin{lem}
  \label{LEM:kerneles}
  We have the formula
  \begin{equation}
    \label{Eqn:Ik} 
    \widehat{\mathcal{J}_{\eta} v} (\lambda) =
    \int_{\mathbb{R}} \mathcal{K} (\lambda, \lambda') \hat{v} (\lambda')
    d \lambda',
  \end{equation}
  where the kernel $\mathcal{K}$ satisfies the following pointwise estimates
  \begin{equation}
    \label{Eqn:bk} 
    | \mathcal{K} (\lambda,\lambda') | + | \partial_{\lambda, \lambda'}
    \mathcal{K} (\lambda,\lambda')  | \lesssim \bigg( \frac{1}{\langle \lambda \rangle^3} +
    \frac{1}{\langle \lambda - \lambda' \rangle^3} \bigg) \frac{1}{\langle
    \lambda' \rangle} \lesssim \frac{1}{\langle \lambda \rangle \langle
    \lambda - \lambda' \rangle} .
  \end{equation}
\end{lem}

We also need the following short-time bound on Fourier restriction norm from \cite[Proposition 2.7]{DNY2024} and \cite[Lemma 4.2]{DNY2022}.

\begin{prop}
  \label{Prop:stb}{{\em (Short time bounds).\/}} 
  Let $\eta$ be any Schwartz
  function. Recall that $\eta_\tau (t) = \eta (\tau^{- 1} t)$ for $\tau \ll
  1$. Then for any $u \in X^{b_1}$, we have
  \[ 
    \| \eta_\tau \cdot u \|_{X^b} \lesssim \tau^{b_1 - b} \| u \|_{X^{b_1}}, 
  \]
  provided either $0 < b \le  b_1 < 1 / 2$, or $u_k (0) = 0$ and $1 / 2 <
  b \le  b_1 < 1$.
\end{prop}

Now we are ready to prove Proposition \ref{local2} by induction. So we assume $ \mathtt{Loc}(M) $ and there are basically two parts to deal with.

\subsection{Random averaging operator}

In this subsection, we prove \eqref{induct1} to \eqref{rao_HS}. In particular, we need to show \eqref{rao-OP} to \eqref{rao_HS} for $L_{\max}=M$ and \eqref{induct1} to \eqref{induct3} for $L=M$.

\subsubsection{Reduction to bilinear estimates}
First of all, the $\tau^{\theta}$ comes from Lemma \ref{Prop:stb} and a small change in $b$. Later we omit the factor $\tau^{\theta}L^{\theta}$ or $\tau^{\theta}N^{\theta}$ for simplicity. Then note that $\|\mathcal P^{\pm}\|_{Y^{1,0}} \lesssim L_{\max}^{C}$ and $\| \mathcal{P}^{\pm} \|_{Z^{1, b}} \lesssim N^4 L_{\max}^{C}$ for some $C\gg 1$ by trivial estimates. The key ingredient is the decomposition $y_L = \psi_{L,R_{L}}+z_L$ whose $X^b$ norm is known by the induction  hypothesis. A similar argument can be found in \cite[Corollary 4.1]{Wang2024}. By interpolation, it suffices to show that

\begin{equation}
\begin{aligned}
&\| \mathcal P^{\pm} \|_{Y^{1-b,b'}}\lesssim L^{-3\delta_0},\\
&\| \mathcal{P}^{+} \|_{Z^{1-b,b'}}  \lesssim N^{\frac{5-2\alpha}{2} + \frac{ 2\gamma_0}{3} } L^{ -\frac{\gamma_0}{2}}_{\max},
\end{aligned}
\end{equation}
\noi
where $b>b'>\frac{1}{2}$.

Recall that $\widehat{P^{\pm}_{k k'}} (\tau,
\tau')$ is the temporal Fourier transform of the kernel $P^{\pm}_{k k'} (t,
t')$ of $\mathcal{P}^{\pm}$ in \eqref{defop1} and \eqref{defop2}, i.e.
\[ 
  \widehat{(\mathcal{P}^{\pm} \psi)_k} (\lambda) = \sum_{k'}
  \int_{\mathbb{R}} \widetilde{P^{\pm}_{k k'}} (\lambda, \lambda')
  \widehat{\psi_{k'}} (\lambda') d \lambda' . 
\]

We start with the operator $\mathcal{P}^+$.
 Rewrite the temporal
Fourier transform of the kernel as
\[ 
  \begin{aligned}
     \widehat{P^+_{k k'}} (\lambda, \lambda')
      & = - i  \sum_{\substack{
      k_3 - k_2 = k - k'\\
      k_2 \not\in \{ k_3, k' \}
    }} \int_{\mathbb{R}^2} \mathcal{K} (\lambda, \Phi + \lambda'  - \lambda_2+
    \lambda_3)  
    \cdot \overline{(\widetilde{y_{L_2}})_{k_2}} (\lambda_2)
    \cdot ( \widetilde{y_{L_3}} )_{k_3} (\lambda_3)d \lambda_2 d \lambda_3\\
    & = - i \sum_{\substack{m\in\Z,\\|m|\lesssim N^{\alpha}}}\int_{\mathbb{R}^2} \sum_{k_2, k_3}
    \mathcal{K} (\lambda, m + \Phi - [\Phi] + \lambda'  - \lambda_2+ \lambda_3)\\
    &\hspace{3cm}\times\mathrm{T}_{k k' k_2 k_3}^{\text{b}, m} \cdot  \overline{(\widetilde{y_{L_2}})_{k_2}} (\lambda_2)\cdot ( \widetilde{y_{L_3}}
    )_{k_3} (\lambda_3) d \lambda_2 d \lambda_3,
  \end{aligned}
\]
where $\Phi = |k'|^{\alpha}-|k_2|^{\alpha}+|k_3|^{\alpha}-|k|^{\alpha}$ comes from the translation of each $\lambda_j$. (See Remark \ref{abuseXb}) $[\Phi]$ denotes the largest integer
that is not greater than $\Phi$, and the base tensor ${\rm T}^{{\rm b},m}$ is given by
\eqref{baseT}. Then by Minkowski inequality, Lemma \ref{LEM:kerneles} and H\"older inequality, we have

\begin{equation}
  \label{Eqn:re2} 
  \begin{aligned}
   \| \mathcal P^+    \|_{Y^{1-b,b'}}^2 &\lesssim  \int_{\mathbb{R}^2} \langle \lambda \rangle^{2 (1 - b)} \langle \lambda'
    \rangle^{- 2 b'} \| \widehat{P^+_{k k'}} (\lambda, \lambda') \|_{k \to 
    k'}^2 d \lambda d \lambda'\\
    & \lesssim \int_{\mathbb{R}^2} \langle \lambda \rangle^{- 2 b} \langle \lambda'
    \rangle^{- 2 b'} \bigg( \sum_{|m|\lesssim N^\alpha} \int_{\mathbb{R}^2} 
    \langle \lambda_2 \rangle^{- b} \langle \lambda_3 \rangle^{- b} \\
    &\hspace{1.5cm}\times \Big\|  \sum_{k_2, k_3} \langle \lambda - m + \Phi - [\Phi]
    - \lambda'  + \lambda_2 - \lambda_3 \rangle^{- 1} \\
    & \hspace{1.5cm} \times \mathrm{T}_{k k' k_2
    k_3}^{\text{b}, m} \cdot \overline{( \langle \lambda_2
    \rangle^b \widetilde{y_{L_2}} )_{k_2} } (\lambda_2) \cdot ( \langle \lambda_3 \rangle^b \widetilde{y_{L_3}}
    )_{k_3} (\lambda_3)  \Big\|_{k
    \to  k'}  d \lambda_2 d \lambda_3\bigg)^2 d \lambda
    d \lambda'\\
    & \lesssim \int_{\mathbb{R}^4} \langle \lambda \rangle^{- 2 b} \langle \lambda'
    \rangle^{- 2 b'}  \\
    &\hspace{1cm}\times\left(\sum_{|m|\lesssim N^\alpha}\langle \lambda - m - \lambda' + \lambda_2 - \lambda_3\rangle^{- 1} \langle \lambda_2
    \rangle^{- b} \langle \lambda_3 \rangle^{- b}\right)^2d \lambda_2 d \lambda_3 d
    \lambda d \lambda'\\
    & \hspace{1cm}  \times \sup_{m\in\Z}\Big\|\sum_{k_2, k_3} \mathrm{T}_{k k'
    k_2 k_3}^{\text{b}, m}  \cdot \overline{( \langle \lambda_2
    \rangle^b \widetilde{y_{L_2}} )_{k_2} }(\lambda_2) \cdot ( \langle \lambda_3 \rangle^b
    \widetilde{y_{L_3}} )_{k_3} (\lambda_3)\Big\|^2_{L^2_{\lambda_2\lambda_3}(k
    \to  k')} \\
  \end{aligned}
\end{equation}

\noi 
\noi
where $b > b' > \frac12$. Now we have to localize $m$ to avoid a $\log N$ loss. In fact, if $|\Phi|\sim |m|\gtrsim L^{C_1}_{\max}$, we will gain $L_{\max}^{-2bC_1}$ from the convolution in $\lambda$ and $\lambda_j$, then a trivial estimate is enough to achieve the desired bound by taking $C_1$ large. Hence, we assume that $|m|\lesssim L^{C_1}_{\max}$. Continuing the estimate \eqref{Eqn:re2}, we use the sum over $m$

\[
\sum_{|m|\lesssim L^{C_1}_{\max}}\langle m-\lambda + \lambda'-\lambda_2+\lambda_3\rangle^{-1}\lesssim C_1 \log L_{\max} \lesssim L_{\max}^{\theta}
\]
\noi
to obtain the estimate

\[
\begin{aligned}
    \| \mathcal P^+ \|_{Y^{1-b,b'}}&\lesssim  L^{\theta}_{\max}\cdot \sup_{m\in\Z}\Big\|\sum_{k_2, k_3} \mathrm{T}_{k k'
    k_2 k_3}^{\text{b}, m}  \cdot \overline{( \langle \lambda_2
    \rangle^b \widetilde{y_{L_2}} )_{k_2} }(\lambda_2) \cdot ( \langle \lambda_3 \rangle^b
    \widetilde{y_{L_3}} )_{k_3} (\lambda_3)\Big\|_{L^2_{\lambda_2\lambda_3}(k
    \to  k')}\\
    & \lesssim L^{\theta}_{\max}\left\| \sum_{k_2,k_3}\mathrm{T}_{k k'
    k_2 k_3}^{\text{b}, m}  \cdot \|( \langle \lambda_2
    \rangle^b \widetilde{y_{L_2}} )_{k_2} (\lambda_2)\|_{L^2_{\lambda_2}} \cdot \|( \langle \lambda_3 \rangle^b
    \widetilde{y_{L_3}} )_{k_3} (\lambda_3)\|_{L^2_{\lambda_3}} \right\|_{k\to k'}
\end{aligned}
\]
We can deal with $\| \mathcal P^-    \|_{Y^{1-b,b'}}$ in the same way. As for $\| \mathcal{P}^{\pm} \|_{Z^{1-b,b'}}$, we cannot localize $|m|\lesssim L_{\max}^{C_1}$, but the restriction $|m|\lesssim N^\alpha$ is acceptable. Multiplying $\log N$ will not affect the estimates for $Z^{1-b,b'}$ norm. Then for $(N,M)\in \mathcal{K}_0$ and  $L_j\leq M$, recall that

\[
 y_{L_j} = e^{-it{\rm D}^\alpha}(\Delta_{L_j} u^{\omega})+\sum_{(L_j,L)\in\mathcal{K}_0}\zeta_{L_j,L}+z_{L_j},
\]
\noi
and we know all bounds of the above terms by $ \mathtt{Loc}(M) $ since $L < L_j^{1-\delta}<M$. Moreover, the initial data component can be viewed as a special kind of random averaging operator with $h^{Lj,\frac{1}{2}}_{kk'} = \ind_{\frac{L_j}{2}<\jb{k}\leq L_j}\ind_{k=k'}$. Hence, it reduces to the following bilinear estimates with a slight abuse of notations:

\begin{equation}\label{needtoshow}
\|\mathscr{P}^{\pm}_{k k'}\|_{k\to k'} \lesssim L_{\max}^{-3\delta_0},\quad 
\|\mathscr{P}^{+}_{k k'}\|_{k k'}  \lesssim N^{\frac{5-2\alpha}{2} + \frac{ \gamma_0}2 } L^{ - \frac{\gamma_0}{2}}_{\max},
\end{equation}
\noi
where

\begin{equation}
\mathscr{P}^{+}_{k k_1} = \sum_{k_2, k_3} \mathrm{T}_{k k_1
    k_2 k_3}^{\text{b}, m}  \cdot \overline{ (w_{L_2})_{k_2} } \cdot (w_{L_3})_{k_3},\quad
    \mathscr{P}^{-}_{k k_2} = \sum_{k_1, k_3} \mathrm{T}_{k k_1
    k_2 k_3}^{\text{b}, m}  \cdot (w_{L_1})_{k_1} (w_{L_3})_{k_3}.
\end{equation}
Here $w_{L_j}$ are of the following two types:

\begin{itemize}
  \item Type (C), where
  \[ 
    (w_{L_j})_{k_j} := \sum_{ {k_j'} } h_{k_j
    k_j'}^{L_j, R_j} (\omega) \frac{g_{k_j'} (\omega)}{\jbb{k_j'}^{\frac{\alpha}{2}}},
  \]
  with $h_{k_j k_j'}^{L_j, R_j}$ supported in the set $\{  k_j': \frac{L_j}{2} < \jb{k_j'} \le  L_j \}$ and $\mathcal{B}_{\le 
  R_j}$-measurable for some $R_j \le  L_j^{1 - \delta}$, and satisfying
  the bounds
  \begin{equation}
    \label{Eqn:ee1} 
       \| h_{k_j k_j'}^{L_j, R_j} \|_{\ell_{k_j}^2 \to  \ell_{k_j'}^2}
      \lesssim R_j^{- \delta_0},\quad \| h_{k_j k_j'}^{L_j, R_j} \|_{{\ell_{k_j k_j'}^2}} \lesssim
      L_j^{\frac{5-2\alpha}{2} + \gamma_0} R_j^{ -\frac{\gamma_0}{4}}.
  \end{equation}

\noi
Moreover, from
  \eqref{induct3} we may assume that $h_{k_j k_j'}^{L_j, R_j}$ is supported in 
  \{$| k_j - k_j' | \lesssim L_j^{\kappa_0^{-\frac{1}{2}}} R_j$\}.

  \smallskip
  
  \item Type (D), where $(w_{L_j})_{k_j}$ is supported in $\{k_j \in \Z^2:  \jb{ k_j } \le 
  L_j \}$, and satisfies
  \begin{equation}
    \label{Eqn:ee3} 
    \| (w_{L_j})_{k_j} \|_{\ell_{k_j}^2} \lesssim L_j^{\frac{10}{3} -2\alpha+ \gamma} .
  \end{equation}
\end{itemize}

Notice that we omit the factor $\jb{\lambda_j}^b$ and integral in $\lambda_j$ to simplify the notation.

\begin{rem} (Meshing arguments) 
    To be precise, we need the meshing argument to "eliminate" the dependence on time variable of $h^{L,R}_{kk'}(\lambda)$ so that we can apply large deviation estimates. Since it has become a standard procedure nowadays, we only sketch its main ideas here. For detailed argument and further discussion, see \cite[Remark 4.3]{{DNY2024}} , \cite[Subsection 3.4]{DNY2021}, and \cite[Subsection 6.1]{DNY2022}. 

First, we may restrict to low frequency $|\lambda|\leq N^{\delta^{-7}}$ for some $\delta\ll1$, since we may gain considerable smoothing otherwise. Next, we divide the box $\{\lambda\in\mathbb{R}^2:|\lambda|\leq N^{\delta^{-7}}\}$ into small boxes of sidelength $N^{-\delta^{-1}}$. Finally, by taking averages
on these small boxes and using Poincar\'e inequality, 
we can approximate the tensor $\{\widetilde{h^{N,L}_{kk'}}(\lambda) \}_{k,k'\in \Z^2}$ by  $\{\widetilde{(h^{N,L}_{0})_{kk'}}(\lambda)\}_{k,k'\in \Z^2}$ in the sense that
\[
\|\widetilde{{h}}^{N,L} - \widetilde{{h}}^{N,L}_{0} \|_{Z^b(J)} \lesssim \, N^{-\delta^{-1}} \,\bigg( \int_{\mathbb{R}} \jb{\lambda}^{2b} \|  \widetilde{h^{N,L}_{kk'}} (\lambda)\|_{kk'}^2 d\lambda + \int_{\mathbb{R}}  \jb{\lambda}^{2b}\|  \partial_\lambda \widetilde{h^{N,L}_{kk'}} (\lambda)\|_{kk'}^2 d\lambda \bigg)^{\frac12}.
\]
Once the above is properly controlled, it suffices to consider at most $N^{2\delta^{-7}}$ different values of $\lambda$. This allows us to apply large deviation estimates by removing an exceptional set whose probability is still $O (e^{-(\tau^{-1}N)^{\theta}})$.
Roughly speaking, we may fix the $\lambda$ variable in the computation. Similar meshing arguments would be needed when we reduce the remainders to trilinear estimates.
\end{rem}

\mb
\subsubsection{Estimates for $\mathscr{P}^{+}$} We start from the high-low-low case. Based on the types of $w_{N_j}$, we shall discuss the following situations:

\gb
\paragraph{(1) Type (C, C)}  
\label{Subsubsub:cc+}

In this case, the bilinear form can be written as
\[ 
  \mathscr{P}_{k k_1}^{+} = \sum_{k_2, k_3} {\rm T}^{{\rm b},m}_{k k_1 k_2
  k_3} \cdot \sum_{k_2', k_3'}  \overline{h_{k_2 k_2'}^{L_2, R_2}}h_{k_3 k_3'}^{L_3, R_3}   \frac{ \overline{g_{k_2'}}g_{k_3'}}{\jbb{k_2'}^{\frac{\alpha}{2}} \jbb{k_3'}^{\frac{\alpha}{2}}} 
\]
where $ \{ h_{k_j k'_j}^{L_j, R_j} \}$ , $j \in \{ 2,
3 \}$  satisfies \eqref{Eqn:ee1}.

We first consider the no pairing case, i.e. $k'_2 \neq k'_3$, for which we apply Proposition \ref{HLLtensor} and Proposition \ref{PROP:puretensor} to get 
  \begin{align} 
  \label{P_CC}
  \begin{split}
    & \| \mathscr{P}_{kk_1}^{+} \|_{k \to  k_1}   \lesssim (L_2 L_3)^{-\frac{\alpha}{2}} \big ( \| \mathrm{T}_{k k_1
    k_2 k_3}^{\text{b}, m} \|_{k k_2 \to  k_1 k_3} + \|
    \mathrm{T}_{k k_1 k_2 k_3}^{\text{b}, m} \|_{k k_3 \to  k_1
    k_2} \\
    & \hspace{2cm} + \| \mathrm{T}_{k k_1 k_2
    k_3}^{\text{b}, m} \|_{k k_2 k_3 \to  k_1} + \|
    \mathrm{T}_{k k_1 k_2 k_3}^{\text{b}, m} \|_{k \to  k_1 k_2
    k_3} \big )  \prod_{j=2}^3 \| h_{k_j k_j'}^{L_j, R_j} \|_{k_j \to  k'_j} \\
    & \hspace{2cm}\lesssim (L_2 L_3)^{-\frac{\alpha}{2}} \big(|S_{kk_2}|^{\frac12}|S_{k_1 k_3}|^{\frac12}+|S_{kk_3}|^{\frac12}|S_{k_1 k_2}|^{\frac12}+|S_{k_1}|^{\frac12}+|S_{k }|^{\frac12}\big),
  \end{split} 
\end{align}

\noi 
which is enough for our purpose. When $L_2 = L_3$, \eqref{P_CC} follows from Proposition \ref{HLLtensor} and then Proposition \ref{PROP:puretensor} as $k_2' \neq k_3'$.  Suppose $L_2 < L_3$ (the other case is similar). Then the coefficients $h_{k_2 k_2'}^{L_2, R_2}$ may not be independent of $(g_{k_3'})_{\frac{L_3}{2}<\jb{k_3'} \le L_3}$. Therefore, we cannot use Proposition \ref{HLLtensor} to $k_2'$ and $k_3'$ directly. 
Instead, we first use Proposition \ref{HLLtensor} only with respect to $k_3'$, then Proposition \ref{PROP:puretensor} and finally apply Proposition \ref{HLLtensor} again with respect to $k_2'$. A similar argument can be found in \cite[Section 4.2]{Wang2024}.

Now we assume $L_2 \geq L_3$ and the other case is similar by using different bounds of $S_{k_1}$ and $S_k$. Then by Lemma \ref{2vec1}, Lemma \ref{2vec2} and Lemma \ref{3vec1}, we have

\[
\begin{aligned}
    \| \mathscr{P}_{k k_1}^{+} \|_{k \to  k_1}   &\lesssim (L_2 L_3)^{-\frac{\alpha}{2}} \left((L_2 L_3)^{\frac{4}{3}-\frac{\alpha}{2}}+(L_2 L_3)^{\frac{3-\alpha}{2}}  + L_2^{\frac{4}{3}-\frac{\alpha}{2}}L_3^1 + L_2^{\frac{3-\alpha}{2}}L_3^{\frac{1}{2}+\theta}+L_2^{\frac{1}{2}}L_3^1\right)\\
    & \lesssim L_2^{\frac{3-2\alpha}{2}}L_3^{\frac{3-2\alpha}{2}} +L_2^{\frac{1-\alpha}{2}}L_3^{1-\frac{\alpha}{2}} \lesssim L_{2}^{\frac{3-2\alpha}{2}}
\end{aligned}
\]

Similarly we can apply Proposition \ref{Prop:tensor} and Proposition \ref{PROP:puretensor} to get

\[
\begin{aligned}
    \| \mathscr{P}_{k k_1}^{+} \|_{k k_1}   &\lesssim (L_2 L_3)^{-\frac{\alpha}{2}} \|\mathrm{T}_{k k_1 k_2 k_3}^{\text{b}, m} \|_{k  k_1 k_2
    k_3}  \prod_{j=2}^3 \| h_{k_j k_j'}^{L_j, R_j} \|_{k_j \to  k'_j}\\
    & \lesssim  (L_2 L_3)^{-\frac{\alpha}{2}} |S|^{\frac{1}{2}}\,\,\, \lesssim (L_2 L_3)^{-\frac{\alpha}{2}} \cdot\left( N_1^{\frac{3-\alpha}{2}} L_2^{\frac{1}{2}+\theta} L_3+L_1^{\frac{1}{2}}L_2L_3\right)\\
    & \lesssim N_1^{\frac{3-\alpha}{2}} L_2^{\frac{1-\alpha}{2}+\theta} L_3^{1-\frac{\alpha}{2}}+N_1^{\frac{1}{2}}(L_2L_3)^{1-\frac{\alpha}{2}}\\
    & \lesssim N_1^{\frac{5-2\alpha}{2}} \lesssim N_1^{\frac{5-2\alpha}{2}+\frac{\gamma_0}{2}}L_2^{-\frac{\gamma_0}{2}}
\end{aligned}
\]
\noi
where $\theta$ comes from $\log L_2$ and this is enough for \eqref{needtoshow}.

Then we consider the pairing case, i.e. $k_2'=k_3'$ that implies $L_2= L_3 = L_{\max}$. We may assume that $R_2\leq R_3$ and the other case is similar. (If $R_2\geq R_3$, we should use the norm $\| h_{k_2 k_2'}^{L_2, R_2}\|_{k_2 k_2'}$ in the following estimates of $\mathscr{P}_{k k_1}^{+,2}$ .)

\[
\begin{aligned}
  \mathscr{P}_{k k_1}^{+} 
  & = \sum_{k_2, k_3} {\rm T}^{{\rm b},m}_{k k_1 k_2
  k_3} \cdot \sum_{k_2'}  \overline{h_{k_2 k_2'}^{L_2, R_2}}h_{k_3 k_2'}^{L_3, R_3}   \frac{ |g_{k_2'}|^2-1}{\jbb{k_2'}^{\alpha} } + \sum_{k_2, k_3} {\rm T}^{{\rm b},m}_{k k_1 k_2
  k_3} \cdot \sum_{k_2'}  \overline{h_{k_2 k_2'}^{L_2, R_2}}h_{k_3 k_2'}^{L_3, R_3}   \frac{ 1}{\jbb{k_2'}^{\alpha} }\\
  & = \mathscr{P}_{k k_1}^{+,1} + \mathscr{P}_{k k_1}^{+,2}
\end{aligned}
\]

The $\mathscr{P}_{k k_1}^{+,1}$ part is similar to, but easier than the non-pairing case. In fact, we can apply Proposition \ref{HLLtensor} (or Proposition \ref{Prop:tensor} for the $\ell^2_{kk_1}$ norm) , Proposition \ref{PROP:puretensor} and then Lemma \ref{LEM:mixedtensor} to get

\[
\begin{aligned}
    \|\mathscr{P}_{k k_1}^{+,1}\|_{k\to k_1} &\lesssim L_2^{-\alpha} \left\|\sum_{k_2, k_3} {\rm T}^{{\rm b},m}_{k k_1 k_2
  k_3} \overline{h_{k_2 k_2'}^{L_2, R_2}}h_{k_3 k_2'}^{L_3, R_3}  \right\|_{(kk_2'\to k_1)\cap(k\to k_1k_2')}  \\
  & \lesssim L_2^{-\alpha} \|{\rm T}^{{\rm b},m}_{k k_1 k_2
  k_3}\|_{(kk_2k_3\to k_1)\cap(k\to k_1 k_2 k_3)} \prod_{j=2}^3 \| h_{k_j k_j'}^{L_j, R_j} \|_{k_j \to  k'_j}\\
  & \lesssim L_2^{-\alpha} \left( |S_{k_1}|^{\frac{1}{2}}+|S_{k}|^{\frac{1}{2}} \right) 
\end{aligned}
\]
\noi
and

\[
    \| \mathscr{P}_{k k_1}^{+,1} \|_{k k_1}   \lesssim L_2 ^{-\alpha} \|\mathrm{T}_{k k_1 k_2 k_3}^{\text{b}, m} \|_{k  k_1 k_2
    k_3}  \prod_{j=2}^3 \| h_{k_j k_j'}^{L_j, R_j} \|_{k_j \to  k'_j} \lesssim  L_2^{-\alpha} |S|^{\frac{1}{2}}.
\]

The counting here is enough to get the desired bound. We next consider
$\mathscr{P}_{k k_1}^{+,2}$. After summing the dyadic increments in the two
contracted factors, Corollary \ref{cor:unitary-cancellation} and Remark
\ref{rem:unitary-redistribution} reduce its estimate to representative
dyadic terms of the form

\[
 \mathscr{P}_{k k_1}^{+,2}  = \sum_{k_2, k_3} {\rm T}^{{\rm b},m}_{k k_1 k_2
  k_3} \cdot \sum_{k_2'}  \overline{h_{k_2 k_2'}^{L_2, R_2}}h_{k_3 k_2'}^{L_3, R_3} \left(  \frac{1}{\jbb{k_2'}^{\alpha} }- \frac{1}{\jbb{k_2}^{\alpha} }\right).
\]

By \eqref{induct3}, we may assume that $|k_2-k_2'|\lesssim L_2^{\kappa_0^{-\frac{1}{2}}}R_2$ and hence

\[
\left|\frac{1}{\jbb{k_2'}^{\alpha} }- \frac{1}{\jbb{k_2}^{\alpha} }\right| \lesssim L_2^{-\alpha-1+\kappa_0^{-\frac{1}{2}}}R_2. 
\]

Now we may apply Proposition \ref{PROP:puretensor}  to get

\[
\begin{aligned}
    \|\mathscr{P}_{k k_1}^{+,2}\|_{k\to k_1} \lesssim&  \,L_2^{-\alpha-1+\kappa_0^{-\frac{1}{2}}}R_2 \cdot \min\left\{\|{\rm T}^{{\rm b},m}_{k k_1 k_2
  k_3}\|_{kk_2k_3\to k_1} ,  \|{\rm T}^{{\rm b},m}_{k k_1 k_2
  k_3}\|_{k\to k_1 k_2 k_3}\right\}\\
  & \, \times \left\|\sum_{k_2'} \overline{h_{k_2 k_2'}^{L_2, R_2}}h_{k_3 k_2'}^{L_3, R_3} \right\|_{k_2 k_3}\\
  \lesssim &\, L_2^{-\alpha-1+\kappa_0^{-\frac{1}{2}}}R_2 \cdot |S_{k_1}|^{\frac{1}{2}} \cdot\| h_{k_2 k_2'}^{L_2, R_2} \|_{k_2\to k_2'} \| h_{k_3 k_2'}^{L_3, R_3}\|_{k_3 k_2'}\\
  \lesssim &\, L_2^{\frac{3-4\alpha}{2}+\kappa_0^{-\frac{1}{2}} +\gamma_0} R_2^{1-\frac{\gamma_0}{4}} (L_2^{\frac{8}{3}-\alpha}L_2^2)^{\frac{1}{2}}\,
  \lesssim  \, L_2^{\frac{29}{6}-\frac{5\alpha}{2} + \kappa_0^{-\frac{1}{2}} +\frac{3\gamma_0}{4}}.
\end{aligned}
\]
\noi
and

\[
\begin{aligned}
    \|\mathscr{P}_{k k_1}^{+,2}\|_{k k_1} \lesssim&  \,L_2^{-\alpha-1+\kappa_0^{-\frac{1}{2}}}R_2 \cdot \|{\rm T}^{{\rm b},m}_{k k_1 k_2
  k_3}\|_{kk_1\to k_2k_3} \left\|\sum_{k_2'} \overline{h_{k_2 k_2'}^{L_2, R_2}}h_{k_3 k_2'}^{L_3, R_3} \right\|_{k_2 k_3}\\
  \lesssim &\, L_2^{-\alpha-1+\kappa_0^{-\frac{1}{2}}}R_2 \cdot |S_{k k_1}|^{\frac{1}{2}}|S_{k_2 k_3}|^{\frac{1}{2}} \cdot\| h_{k_2 k_2'}^{L_2, R_2} \|_{k_2\to k_2'} \| h_{k_3 k_2'}^{L_3, R_3}\|_{k_3 k_2'}\\
  \lesssim &\, L_2^{\frac{3-4\alpha}{2}+\kappa_0^{-\frac{1}{2}} +\gamma_0} R_2^{1-\frac{\gamma_0}{4}} (N_1^{3-\alpha}L_2^{3-\alpha})^{\frac{1}{2}}\,
  \lesssim  \, N_1^{\frac{3-\alpha}{2}} L_2^{\frac{8-5\alpha}{2} + \kappa_0^{-\frac{1}{2}} +\frac{3\gamma_0}{4}}.
\end{aligned}    
\]

These bounds are enough for \eqref{needtoshow} by our choice of parameters. Therefore, the proof of type (C,C) is finished.

\gb

\paragraph{(2) Type (C, D) and Type (D,C)}  
\label{Subsubsub:cd+}

The two situations are similar so we only consider

\[ 
  \mathscr{P}_{k k_1}^{+} = \sum_{k_2, k_3} {\rm T}^{{\rm b},m}_{k k_1 k_2
  k_3} \cdot \sum_{k_2'}  \overline{h_{k_2 k_2'}^{L_2, R_2}}   \frac{ \overline{g_{k_2'}}}{\jbb{k_2'}^{\frac{\alpha}{2}} } \cdot (z_{L_3})_{k_3}
\]
where $ \{ h_{k_2 k'_2}^{L_2, R_2} \}$ satisfies \eqref{Eqn:ee1} and $z_{L_3}$ satisfies the bound \eqref{Eqn:ee3}. Apply Proposition \ref{PROP:puretensor} followed by Proposition \ref{HLLtensor}  to get

\[
\begin{aligned}
    \|\mathscr{P}_{k k_1}^{+}\|_{k\to k_1} \lesssim&  \,L_2^{-\frac{\alpha}{2}} \cdot  \left\|\sum_{k_2} {\rm T}^{{\rm b},m}_{k k_1 k_2
  k_3}\sum_{k_2'}  \overline{h_{k_2 k_2'}^{L_2, R_2}}   \overline{g_{k_2'}}\right\|_{k\to k_1 k_3}\cdot \|z_{L_3}\|_{\ell_{k_3}^2}\\
  \lesssim &\, L_2^{-\frac{\alpha}{2}} L_3^{\frac{10}{3}-2\alpha+\gamma}\cdot\left\|\sum_{k_2} {\rm T}^{{\rm b},m}_{k k_1 k_2
  k_3}   \overline{h_{k_2 k_2'}^{L_2, R_2}}   \right\|_{(kk_2'\to k_1 k_3)\cap(k\to k_1 k_2' k_3)} \\
  \lesssim &\, L_2^{-\frac{\alpha}{2}} L_3^{\frac{10}{3}-2\alpha+\gamma} \|h_{k_2 k_2'}^{L_2, R_2}\|_{k_2'\to k_2}\cdot\|{\rm T}^{{\rm b},m}_{k k_1 k_2
  k_3}\|_{(kk_2\to k_1k_3)\cap(k\to k_1k_2k_3)}\\
  \lesssim & \, L_2^{-\frac{\alpha}{2}} L_3^{\frac{10}{3}-2\alpha+\gamma}\cdot \left(|S_{kk_2}|^{\frac{1}{2}}|S_{k_1 k_3}|^{\frac{1}{2}}+|S_k|^{\frac{1}{2}}\right).
\end{aligned}
\]

Now we use Lemma \ref{2vec1}, Lemma \ref{2vec2} and Lemma \ref{3vec1} to get

\[
\begin{aligned}
    \|\mathscr{P}_{k k_1}^{+}\|_{k\to k_1} &\lesssim L_2^{-\frac{\alpha}{2}} L_3^{\frac{10}{3}-2\alpha+\gamma}\Big( (L_2^{\frac{3-\alpha}{2}} L_3^{\frac{1}{2}+\theta}+ L_2^{
    \frac{1}{2}}L_3) \wedge (L_3^{\frac{4}{3}-\frac{\alpha}{2}}L_2)+
    (L_2L_3)^{\frac{4}{3}-\frac{\alpha}{2}}\Big)\\
    &\lesssim \, \begin{cases}
       L_2^{\frac{3-2\alpha}{2}}L_3^{\frac{29}{6}-\frac{5\alpha}{2}+\gamma}+L_2^{\frac{1-\alpha}{2}}L_3^{\frac{13}{3}-2\alpha+\gamma} , \, &\text{if} \, L_2\geq L_3,\\
       L_2^{\frac{3-2\alpha}{2}}L_3^{\frac{29}{6}-\frac{5\alpha}{2}+\gamma}+L_3^{\frac{14}{3}-\frac{5\alpha}{2}+\gamma}L_2^{1-\frac{\alpha}{2}} , \, &\text{if} \, L_2\leq L_3
   \end{cases}\\
   &\lesssim (L_2\vee L_3)^{\frac{29}{6}-\frac{5\alpha}{2}+\gamma}
\end{aligned}
\]

\noi
which are enough for the operator norm in \eqref{needtoshow} by our choice of parameters. Similarly, we can apply Proposition \ref{PROP:puretensor} followed by Proposition \ref{Prop:tensor}  to get

\[
\begin{aligned}
    \|\mathscr{P}_{k k_1}^{+}\|_{k k_1} \lesssim&  \,L_2^{-\frac{\alpha}{2}} \cdot \left\|\sum_{k_2} {\rm T}^{{\rm b},m}_{k k_1 k_2
  k_3}\sum_{k_2'}  \overline{h_{k_2 k_2'}^{L_2, R_2}}   \overline{g_{k_2'}}\right\|_{k k_1 \to k_3}\cdot \|z_{L_3}\|_{\ell_{k_3}^2}\\
  \lesssim &\, L_2^{-\frac{\alpha}{2}} L_3^{\frac{10}{3}-2\alpha+\gamma}\cdot\left\|\sum_{k_2} {\rm T}^{{\rm b},m}_{k k_1 k_2
  k_3}   \overline{h_{k_2 k_2'}^{L_2, R_2}}   \right\|_{(kk_1k_2'\to k_3)\cap(k k_1\to k_2' k_3)} \\
  \lesssim &\, L_2^{-\frac{\alpha}{2}} L_3^{\frac{10}{3}-2\alpha+\gamma} \|h_{k_2 k_2'}^{L_2, R_2}\|_{k_2'\to k_2}\cdot\|{\rm T}^{{\rm b},m}_{k k_1 k_2
  k_3}\|_{(kk_1k_2\to k_3)\cap(kk_1\to k_2k_3)}\\
  \lesssim & \, L_2^{-\frac{\alpha}{2}} L_3^{\frac{10}{3}-2\alpha+\gamma}\cdot \Big(|S_{k_3}|^{\frac{1}{2}} +  |S_{kk_1}|^{\frac{1}{2}}|S_{k_2 k_3}|^{\frac{1}{2}}\Big)\\
   \lesssim & \, L_2^{-\frac{\alpha}{2}} L_3^{\frac{10}{3}-2\alpha+\gamma}\cdot \Big( N_1^{\frac{3-\alpha}{2}}L_2^{\frac{1}{2}+\theta}+N_1^{
    \frac{1}{2}}L_2 +  (L_3\wedge L_2)^{\frac{3-\alpha}{2}} N_1^{\frac{3-\alpha}{2}}\Big) .
\end{aligned}
\]

Now if $L_2\leq L_3$ , then

\[
    \|\mathscr{P}_{k k_1}^{+}\|_{k k_1} \lesssim N_1^{\frac{3-\alpha}{2}} L_3^{\frac{10}{3}-2\alpha+\gamma},
\]
\noi
and if $L_2\geq L_3$, then

\[
\begin{aligned}
    \|\mathscr{P}_{k k_1}^{+}\|_{k k_1} &\lesssim (N_1^{\frac{3-\alpha}{2}} L_2^{\frac{1-\alpha}{2}+\theta}+ N_1^{\frac{1}{2}}L_2^{1-\frac{\alpha}{2}})\cdot L_3^{\frac{10}{3}-2\alpha+\gamma}+ N_1^{\frac{3-\alpha}{2}}L_2^{-\frac{\alpha}{2}} L_3^{\frac{29}{6}-\frac{5\alpha}{2}+\gamma}\\
    &\lesssim N_1^{\frac{3-\alpha}{2}} \lesssim N_1^{\frac{5-2\alpha}{2}}L_2^{\frac{\alpha-2}{2}},
\end{aligned}
\]
\noi
which is also enough for \eqref{needtoshow} by our choice of parameters.

\gb

\paragraph{(3) Type (D, D)}  
\label{Subsubsub:dd+}

Finally, we consider

\[ 
  \mathscr{P}_{k k_1}^{+} = \sum_{k_2, k_3} {\rm T}^{{\rm b},m}_{k k_1 k_2
  k_3} \cdot  \overline{(z_{L_2})_{k_2}} \cdot (z_{L_3})_{k_3}
\]
where $z_{L_2}$, $z_{L_3}$ satisfy the bound \eqref{Eqn:ee3}. Apply Proposition \ref{PROP:puretensor} to get

\[
\begin{aligned}
    \|\mathscr{P}_{k k_1}^{+}\|_{k\to k_1} & \lesssim \|{\rm T}^{{\rm b},m}_{k k_1 k_2
  k_3}\|_{kk_3\to k_1 k_3} \|z_{L_2}\|_{\ell_{k_2}^2}\|z_{L_3}\|_{\ell_{k_3}^2}\\
  & \lesssim |S_{kk_3}|^{\frac{1}{2}}|S_{k_1 k_2}|^{\frac{1}{2}} (L_2L_3)^{\frac{10}{3}-2\alpha+\gamma}\, \lesssim (L_2 L_3)^{\frac{29}{6}-\frac{5\alpha}{2}+\gamma}
\end{aligned}
\]
\noi
and

\[
\begin{aligned}
    \|\mathscr{P}_{k k_1}^{+}\|_{k k_1} & \lesssim \|{\rm T}^{{\rm b},m}_{k k_1 k_2
  k_3}\|_{kk_1\to k_2 k_3} \|z_{L_2}\|_{\ell_{k_2}^2}\|z_{L_3}\|_{\ell_{k_3}^2}\\
  & \lesssim |S_{kk_1}|^{\frac{1}{2}}|S_{k_2 k_3}|^{\frac{1}{2}} (L_2L_3)^{\frac{10}{3}-2\alpha+\gamma}\\
  & \lesssim  N_1^{\frac{3-\alpha}{2}} (L_2\vee L_3)^{\frac{10}{3}-2\alpha+\gamma} (L_2\wedge L_3)^{\frac{29}{6}-\frac{5\alpha}{2} + \gamma}\\
  &\lesssim N_1^{\frac{3-\alpha}{2}} (L_2\vee L_3)^{\frac{10}{3}-2\alpha+\gamma},
\end{aligned}
\]
\noi
which are enough for \eqref{needtoshow} by our choice of parameters. Hence, we complete the proof of \eqref{rao-OP} and \eqref{rao_HS} for $\mathcal{P}^{+}$.

\subsubsection{Estimates for $\mathscr{P}^{-}$} Now we consider the low-high-low case. Still we have the following possibilities:

\gb

\paragraph{(1) Type (C, C)}  
\label{Subsubsub:cc-} 

The bilinear form we consider now can be written as

\[ 
  \mathscr{P}_{k k_2}^{-} = \sum_{k_1, k_3} {\rm T}^{{\rm b},m}_{k k_1 k_2
  k_3} \cdot \sum_{k_1', k_3'}  h_{k_1 k_1'}^{L_1, R_1}h_{k_3 k_3'}^{L_3, R_3}   \frac{ g_{k_1'}g_{k_3'}}{\jbb{k_1'}^{\frac{\alpha}{2}} \jbb{k_3'}^{\frac{\alpha}{2}}} 
\]
where $ \{ h_{k_j k'_j}^{L_j, R_j} \}$ , $j \in \{ 1,
3 \}$  satisfies \eqref{Eqn:ee1}.

This case is easier than that of high-low-low since there is no possibility of pairing. However, we still need to discuss when the independence between coefficients and  Gaussian random variables holds. We omit the details. Apply Proposition \ref{HLLtensor} and Proposition \ref{PROP:puretensor} to get

\[
  \begin{aligned} 
    & \| \mathscr{P}_{kk_2}^{-} \|_{k \to  k_2}   \lesssim (L_1 L_3)^{-\frac{\alpha}{2}} \big ( \| \mathrm{T}_{k k_1
    k_2 k_3}^{\text{b}, m} \|_{k k_1 \to  k_2 k_3} + \|
    \mathrm{T}_{k k_1 k_2 k_3}^{\text{b}, m} \|_{k k_3 \to  k_1
    k_2} \\
    & \hspace{2cm} + \| \mathrm{T}_{k k_1 k_2
    k_3}^{\text{b}, m} \|_{k k_1 k_3 \to  k_2} + \|
    \mathrm{T}_{k k_1 k_2 k_3}^{\text{b}, m} \|_{k \to  k_1 k_2
    k_3} \big )  \prod_{j=1,3} \| h_{k_j k_j'}^{L_j, R_j} \|_{k_j \to  k'_j} \\
    & \hspace{2cm}\lesssim (L_1 L_3)^{-\frac{\alpha}{2}} \big(|S_{kk_1}|^{\frac12}|S_{k_2 k_3}|^{\frac12}+|S_{kk_3}|^{\frac12}|S_{k_1 k_2}|^{\frac12}+|S_{k_2}|^{\frac12}+|S_{k }|^{\frac12}\big).
\end{aligned}
\]

Furthermore, we assume that $L_1\geq L_3$ and the other case is almost the same. Then

\[
    \| \mathscr{P}_{kk_2}^{-} \|_{k \to  k_2}   \lesssim (L_1 L_3)^{-\frac{\alpha}{2}}\cdot\Big((L_1 L_3)^{\frac{3-\alpha}{2}} + L_1^{\frac{3-\alpha}{2}} L_3^{\frac{1}{2}+\theta}+L_1^{\frac{1}{2}}L_3\Big)\lesssim L_1^{\frac{3-2\alpha}{2}}.
\]

This is the desired bound in \eqref{needtoshow}.

\gb
\paragraph{(2) Type (C, D) and Type (D,C)}  
\label{Subsubsub:cd-}

The two situations are similar so we still only consider

\[ 
  \mathscr{P}_{k k_2}^{-} = \sum_{k_1, k_3} {\rm T}^{{\rm b},m}_{k k_1 k_2
  k_3} \cdot \sum_{k_1'}  h_{k_1 k_1'}^{L_1, R_1}   \frac{ g_{k_1'}}{\jbb{k_1'}^{\frac{\alpha}{2}} } \cdot (z_{L_3})_{k_3}
\]
where $ \{ h_{k_1 k'_1}^{L_1, R_1} \}$ satisfies \eqref{Eqn:ee1} and $z_{L_3}$ satisfies the bound \eqref{Eqn:ee3}. Apply Proposition \ref{PROP:puretensor} followed by Proposition \ref{HLLtensor}  to get

\[
\begin{aligned}
    \|\mathscr{P}_{k k_2}^{-}\|_{k\to k_2} \lesssim&  \,L_1^{-\frac{\alpha}{2}} \cdot \left\|\sum_{k_1} {\rm T}^{{\rm b},m}_{k k_1 k_2
  k_3} \sum_{k_1'}  h_{k_1 k_1'}^{L_1, R_1}   g_{k_1'} \right\|_{kk_3\to k_2} \cdot \|z_{L_3}\|_{\ell_{k_3}^2}\\
  \lesssim &\, L_1^{-
    \frac{\alpha}{2}} L_3^{\frac{10}{3}-2\alpha+\gamma}\cdot\left\|\sum_{k_1} {\rm T}^{{\rm b},m}_{k k_1 k_2
  k_3}  h_{k_1 k_1'}^{L_1, R_1}    \right\|_{(kk_1'k_3\to k_2)\cap(kk_3\to k_1'k_2)}\\
  \lesssim &\, L_1^{-
    \frac{\alpha}{2}} L_3^{\frac{10}{3}-2\alpha+\gamma} \|h_{k_1 k_1'}^{L_1, R_1}\|_{k_1'\to k_1}\cdot\|{\rm T}^{{\rm b},m}_{k k_1 k_2
  k_3}\|_{(kk_1k_3\to k_2)\cap(kk_3\to k_1k_2)}\\
  \lesssim & \, L_1^{-
    \frac{\alpha}{2}} L_3^{\frac{10}{3}-2\alpha+\gamma}\cdot \Big(|S_{k_2}|^{\frac{1}{2}}+|S_{kk_3}|^{\frac{1}{2}}|S_{k_1 k_2}|^{\frac{1}{2}} \Big)\\
  \lesssim & \, L_1^{-
    \frac{\alpha}{2}} L_3^{\frac{10}{3}-2\alpha+\gamma}\cdot \Big(  (L_3^{\frac{3-\alpha}{2}} L_1^{\frac{1}{2}+\theta}  + L_1 L_3^{\frac{1}{2}})\wedge( L_1^{\frac{3-\alpha}{2}} L_3^{\frac{1}{2}+\theta} + L_3 L_1^{\frac{1}{2}})+ (L_1 L_3)^{\frac{3-\alpha}{2}} \Big).
\end{aligned}
\]

Therefore, we use the first bound of $|S_{k_2}|$ if $L_1\leq L_3$ while the second bound if $L_1\geq L_3$ to get

\[
\|\mathscr{P}_{k k_2}^{-}\|_{k\to k_2} \lesssim
\begin{cases}
L_1^{\frac{3-2\alpha}{2}}L_3^{\frac{29}{6}-\frac{5\alpha}{2}+\gamma}+L_1^{\frac{29}{6}-\frac{5\alpha}{2}+\gamma}, \, &\text{if}\quad L_1\geq L_3,\\
L_1^{\frac{3-2\alpha}{2}}L_3^{\frac{29}{6}-\frac{5\alpha}{2}+\gamma}+L_3^{\frac{29}{6}-\frac{5\alpha}{2}+\gamma}, \, &\text{if}\quad L_1\leq L_3.
\end{cases}
\]

These estimates are enough to derive the bound \eqref{needtoshow}.

\gb

\paragraph{(3) Type (D, D)}  
\label{Subsubsub:dd-}

Finally, we consider

\[ 
  \mathscr{P}_{k k_2}^{-} = \sum_{k_1, k_3} {\rm T}^{{\rm b},m}_{k k_1 k_2
  k_3} \cdot  (z_{L_1})_{k_1} \cdot (z_{L_3})_{k_3}
\]
where $z_{L_1}$, $z_{L_3}$ satisfy the bound \eqref{Eqn:ee3}. Apply Proposition \ref{PROP:puretensor} to get

\[
\begin{aligned}
    \|\mathscr{P}_{k k_2}^{-}\|_{k\to k_2} & \lesssim \|{\rm T}^{{\rm b},m}_{k k_1 k_2
  k_3}\|_{kk_1\to k_2 k_3} \|z_{L_1}\|_{\ell_{k_1}^2}\|z_{L_3}\|_{\ell_{k_3}^2}\\
  & \lesssim |S_{kk_1}|^{\frac{1}{2}}|S_{k_2 k_3}|^{\frac{1}{2}} (L_1L_3)^{\frac{10}{3}-2\alpha+\gamma}\, \lesssim (L_1 L_3)^{\frac{29}{6}-\frac{5\alpha}{2}+\gamma},
\end{aligned}
\]
\noi
which are enough for \eqref{needtoshow} by our choice of parameters. So we finish the proof of \eqref{rao-OP} and \eqref{rao_HS}.

\noi
\subsubsection{The kernel of random averaging operator}
 Here we prove \eqref{induct1} to \eqref{induct3} for $L=M$.

First of all , we denote

\[
\begin{aligned}
    \mathcal{P}^{N,L}(z) &= \eta_{\tau}\mathcal{I}\Pi_{N}\mathcal{N}_3(z,v_L,v_L)-\eta_{\tau}\mathcal{I}\Pi_{N}\mathcal{N}_3\left(z,v_{\frac{L}{2}},v_{\frac{L}{2}}\right)\\
    & = \sum_{L_{\max}=L} \eta_{\tau}\mathcal{I}\Pi_{N}\mathcal{N}_3(z,y_{L_2},y_{L_3})\, = \sum_{L_{\max}=L}\mathcal{P}^{+}_{N,L_2,L_3}(z)
\end{aligned}
\]
\noi
and let $\mathfrak{p}^{N,L}_{kk'}$ be its kernel. By \eqref{rao-OP}  and \eqref{rao_HS} which we have shown above, the trivial estimates imply the following

\[
\begin{aligned}
    &\|\mathcal{P}^{N,L}\|_{Y^{b,b}}\lesssim \tau^{\theta} L^{-2\delta_0}(\log L)^2 \lesssim \tau^{\theta} L^{-\frac{3\delta_0}{2}},\\    &\|\mathcal{P}^{N,L}\|_{Z^{b,b}}\lesssim\tau^{\theta} N^{ \frac{5-2\alpha}{2}+ \frac{3\gamma_0}{4} } L^{ - \frac{\gamma_0}{3}}(\log L)^2 \lesssim \tau^{\theta} N^{\frac{5-2\alpha}{2}+ \frac{3\gamma_0}{4} } L^{ -\frac{7\gamma_0}{24}}
\end{aligned}
\]

For any $L\leq M$. Recall that 

\[
    \varphi^{N,L}(t) = \eta(t)e^{-it{\rm D}^\alpha}(e^{ik'x}) + 2\eta_{\tau}(t)\mathcal{I}\Pi_N\mathcal{N}_3 (\varphi^{N,L},v_L,v_L),
\]
\noi
and $H^{N,L}_{kk'} = \varphi^{N,L}_k$. Now we write

\[
\begin{aligned}
    \varphi^{N,M} - \varphi^{N,\frac{M}{2}} = & 2\eta_{\tau}\mathcal{I}\Pi_N\mathcal{N}_3(\varphi^{N,M},v_M,v_M) - 2\eta_{\tau}\mathcal{I}\Pi_N\mathcal{N}_3\left(\varphi^{N,\frac{M}{2}},v_{\frac{M}{2}},v_{\frac{M}{2}}\right)\\
     = & 2\eta_{\tau}\mathcal{I}\Pi_N \mathcal{N}_3(\varphi^{N,M} -\varphi^{N,\frac{M}{2}}, v_M, v_M) + 2\mathcal{P}^{N,M}(\varphi^{N,\frac{M}{2}})\\
     = & 2\sum_{L\leq M}\mathcal{P}^{N,L}(\varphi^{N,M} -\varphi^{N,\frac{M}{2}})  + 2\sum_{L<M}\mathcal{P}^{N,M}(\varphi^{N,L}-\varphi^{N,\frac{L}{2}}) + 2\mathcal{P}^{N,M}(\varphi^{N,\frac{1}{2}}). 
\end{aligned}
\]

Note that $h^{N,M}_{kk'} = (\varphi^{N,M} - \varphi^{N,\frac{M}{2}})_k$, so we have

\[
\begin{aligned}
    h^{N,M}_{kk'}(t) =& 2\sum_{L\leq M} \sum_{k^{\ast}} \int_{\mathbb{R}} \mathfrak{p}^{N,L}_{kk^{\ast}}(t,t')h^{N,M}_{k^{\ast}k'}(t')dt' + 2\sum_{L<M}\sum_{k^{\ast}}\int_{\mathbb{R}}  \mathfrak{p}^{N,M}_{kk^{\ast}}(t,t')h^{N,L}_{k^{\ast}k'}(t')dt'\\
    &+ 2\sum_{k^{\ast}}\int_{\mathbb{R}}  \mathfrak{p}^{N,M}_{kk^{\ast}}(t,t')H^{N,\frac{1}{2}}_{k^{\ast}k'}(t')dt'.
\end{aligned}
\]

Now we are ready to prove \eqref{induct1} to \eqref{induct3} for $L=M$ assuming $ \mathtt{Loc}(M) $. For example, 

\[
\begin{aligned}
    \|h^{N,M}\|_{Y^b} &\lesssim \sum_{L\leq M} \|\mathcal{P}^{N,L}\|_{Y^{b,b}} \cdot \|h^{N,M}\|_{Y^b} + \|\mathcal{P}^{N,M}\|_{Y^{b,b}} \cdot \sum_{L<M}  (\|h^{N,L}\|_{Y^b}+1)\\
    &\lesssim \tau^{\theta}\|h^{N,M}\|_{Y^b} \sum_{L\leq M} L^{-\frac{3\delta_0}{2}} +  \tau^{\theta}M^{-\frac{3\delta_0}{2}} (\sum_{L<M} L^{-\delta_0} + 1).
\end{aligned}
\]

This implies \eqref{induct1} for $L=M$ by choosing $\tau\ll 1$. Similarly, we get\eqref{induct2} and \eqref{induct3} for $L=M$ by Proposition \ref{Prop:weightedZb} and the fact that $\mathfrak{p}^{N,L}_{kk^{\ast}}$ is supported in $|k-k^{\ast}|\lesssim L$. Also see the proof in \cite[Section 3]{DNY2024}.

\gb

\subsection{The remainder}

In this subsection, we focus on the remainder. To be precise, we shall prove \eqref{induct6} for $N=2M$ based on $ \mathtt{Loc}(M) $ and \eqref{induct1} to \eqref{rao_HS} for $L=M$. Recall that

\begin{equation}
    z_{2M} = \eta_{\tau}(t)\mathcal{I}\Pi_{2M}\mathcal{R}_{2M}
\end{equation}
\noi
where $\Rc_{2M}$ is given by \eqref{remainder}. We need to show that the right hand side is a contraction $\Phi$, mapping $\mathcal{Z} = \{z: \|z\|_{X^b}\lesssim (2M)^{\frac{10}{3}-2\alpha+\gamma}\}$ to itself. Here we only prove that it sends $\mathcal{Z}$ to $\mathcal{Z}$. For the difference estimate, expand every cubic difference by multilinearity. Each resulting term contains exactly one factor $z-\widetilde z$, while all other remainder factors are bounded by the radius of $\mathcal Z$. The estimates below apply without change to such a term, with the bound for that distinguished factor left unsummed. Consequently,
\[
\|\Phi(z)-\Phi(\widetilde z)\|_{X^b}
\leq C\tau^\theta\|z-\widetilde z\|_{X^b}.
\]
After decreasing $\tau$ so that $C\tau^\theta\leq\frac12$, the map is a contraction.

\noi

To start with, Lemma \ref{Prop:stb} and a small change for $b$ will produce the factor $\tau^{b_1-b}$. Then we separate the other terms of $\eta_{\tau}(t)\mathcal{I}\Pi_{2M}\mathcal{R}_{2M}$ into the following groups:

\begin{itemize}
    \item [(a)] $\eta_{\tau}(t)\mathcal{I}\Pi_{2M}\mathcal{N}_3(\psi_{2M,L_{2M}},y_{N_2},y_{N_3})$, where $N_2\gtrsim (2M)^{1-\delta}$ or $N_3\gtrsim (2M)^{1-\delta}$;
\noi   
    \item [(b)] $\eta_{\tau}(t)\mathcal{I}\Pi_{2M}\mathcal{N}_3(y_{N_1},\psi_{2M,L_{2M}},y_{N_3})$, where $N_1, N_3\leq M$;
\noi    
    \item [(c)] $\eta_{\tau}(t)\mathcal{I}\Pi_{2M}\mathcal{N}_3(z_{2M},y_{N_2},y_{N_3})$, where $N_2, N_3\leq M$;
\noi
    \item [(d)] $\eta_{\tau}(t)\mathcal{I}\Pi_{2M}\mathcal{N}_3(y_{N_1},z_{2M},y_{N_3})$, where $N_1, N_3\leq M$;
\noi
    \item [(e)] $\eta_{\tau}(t)\mathcal{I}\Delta_{2M}\mathcal{N}_3(y_{N_1},y_{N_2},y_{N_3})$, where $N_1,N_2, N_3\leq M$;
\noi
    \item [(f)] $\eta_{\tau}(t)\mathcal{I}\Pi_{M'}\mathcal{Q}_3(y_{M'},y_{M'},y_{M'})$, where $M'\in \{M,2M\}$.
    
\end{itemize}

\noi

Notice that $L_{2M}\leq (2M)^{1-\delta}< M$ and  $z_{2M}\in \mathcal{Z}$, so we know the bounds of all the terms appearing in the trilinear forms above. To simplify the symbols, we still write $N=2M$ in the following estimates. By choosing $\tau\ll 1$, we only need to show that

\begin{equation}\label{needtoprove}
    \|\mathscr{F}\|_{X^{b_{1}}}\lesssim N^{\frac{10}{3}-2\alpha+\gamma}
\end{equation}
\noi
where $\mathscr{F}$ ranges from (a) to (f). However, we have to deal with these situations in different ways.

\gb

\subsubsection{At least two high frequency terms}\label{2Hterms}

First of all, we consider (a) to (d) when $N_{\max}, N_{\med}\gtrsim N^{1-\delta}$.  Like the estimates for random averaging operator, a trivial counting shows that $\|\mathscr{F}\|_{X^1}\lesssim N^C$. So it suffices to prove

\[
    \|\mathscr{F}\|_{X^{1-b}}\lesssim N^{\frac{10}{3}-2\alpha+\frac{3\gamma}{4}}
\]
\noi
and we get the desired bound by interpolation. (Recall that $b,b_1$ are slightly above $\frac{1}{2}$.) In fact,

\begin{equation}
  \label{Eqn:ftremain1} 
  \begin{aligned}
     \widetilde{\mathscr{F}} (\lambda, k)
      & = - i  \sum_{\substack{
      k = k_1 - k_2 + k_3  \\
      k_2 \not\in \{ k_1, k_3 \}
    }} \int_{\mathbb{R}^3} \mathcal{K} (\lambda, \Phi + \lambda_1  - \lambda_2+
    \lambda_3) \cdot ( \widetilde{y_{N_1}} )_{k_1} (\lambda_1)\\
   & \hspace{3cm} \times \overline{(\widetilde{y_{N_2}})_{k_2}} (\lambda_2)
    \cdot ( \widetilde{y_{N_3}} )_{k_3} (\lambda_3)d \lambda_1 d \lambda_2 d \lambda_3\\
    & = - i \sum_{m\in\Z,|m|\lesssim N^{\alpha}}\int_{\mathbb{R}^3} \sum_{k_1,k_2, k_3}
    \mathcal{K} (\lambda, m + \Phi - [\Phi] + \lambda_1  - \lambda_2+ \lambda_3)\cdot \mathrm{T}_{k k_1 k_2 k_3}^{\text{b}, m}\\
    & \hspace{3cm} \times   ( \widetilde{y_{N_1}} )_{k_1} (\lambda_1)  \cdot  \overline{(\widetilde{y_{N_2}})_{k_2}} (\lambda_2)\cdot ( \widetilde{y_{N_3}}
    )_{k_3} (\lambda_3) d\lambda_1 d \lambda_2 d \lambda_3,
  \end{aligned}
\end{equation}
\noi
where $\Phi = |k'|^{\alpha}-|k_2|^{\alpha}+|k_3|^{\alpha}-|k|^{\alpha}$ , $[\Phi]$ is the integer
part of $\Phi$, and the base tensor ${\rm T}^{{\rm b},m}$ is given by
\eqref{baseT}. Similar to \eqref{Eqn:re2}, we have the following estimates by Minkowski inequality, Lemma \ref{LEM:kerneles} and H\"older inequality

\begin{equation}
  \label{remainder2H} 
  \begin{aligned}
   \| \mathscr{F}    \|_{X^{1-b}}^2 &\lesssim  \int_{\mathbb{R}} \langle \lambda \rangle^{2 (1 - b)}  \| \widetilde{\mathscr{F}} (\lambda, k) \|_{k}^2 d \lambda\\
    & \lesssim \int_{\mathbb{R}^4} \langle \lambda \rangle^{- 2 b} \left(\sum_{|m|\lesssim N^\alpha}\langle \lambda - m - \lambda_1 + \lambda_2 - \lambda_3\rangle^{- 1}\prod_{j=1}^3\langle \lambda_j \rangle^{- b} \right)^2 d\lambda_1 d \lambda_2 d \lambda_3 d
    \lambda\\
    & \hspace{1cm}  \times \sup_{m\in\Z}\Big\|\sum_{k_1,k_2, k_3} \mathrm{T}_{k k'
    k_2 k_3}^{\text{b}, m} \cdot \prod_{j=1}^3 ( \langle \lambda_j \rangle^b \widetilde{y_{N_j}}
    )_{k_j}^{\zeta_j} (\lambda_j)\Big\|^2_{L^2_{\lambda_1\lambda_2\lambda_3}\ell_{k}^2} \\
    & \lesssim \log N \cdot \sup_{m\in\Z}\Big\|\sum_{k_1,k_2, k_3} \mathrm{T}_{k k'
    k_2 k_3}^{\text{b}, m} \cdot \prod_{j=1}^3 ( \langle \lambda_j \rangle^b \widetilde{y_{N_j}}
    )_{k_j}^{\zeta_j} (\lambda_j)\Big\|^2_{L^2_{\lambda_1\lambda_2\lambda_3}\ell_{k}^2}
  \end{aligned}
\end{equation}
\noi
where $\zeta_1,\zeta_3 = +$ and $\zeta_2 = -$. Similar to the random averaging operator, it reduces to the following trilinear estimates:

\begin{equation}
  \label{Eqn:remainderaim1} 
  \| \mathscr{X}_k \|_{k} \lesssim N^{\frac{10}{3}-2\alpha+ \frac{2\gamma}{3} },
\end{equation}
where $\mathscr{X}_k $ is given by
\begin{equation}
  \label{Eqn:Xk1} 
  \mathscr{X}_k = \sum_{k_1, k_2, k_3} \mathrm{T}_{k k_1 k_2
  k_3}^{\text{b}, m} \cdot (w_{N_1})_{k_1} \cdot   \overline{(w_{N_2}) _{k_2}} \cdot (w_{N_3})_{k_3}
\end{equation}
and $w_{N_j}$ are either of the following two types (with a slight abuse of notation, we still call them type (C) and (D)). 
\begin{itemize}
  \item Type (C), where
  \[ 
    (w_{N_j})_{k_j} = \sum_{ k_j' } h_{k_j
    k_j'}^{N_j, L_j} \frac{g_{k_j'} (\omega)}{\jbb{k_j'}^{\frac{\alpha}{2}}}, 
  \]
  with $h_{k_j k_j'}^{N_j, L_j} (\omega)$ supported in the set $ \{
  \frac{N_j}{2} < \jb{k_j'} \le  N_j\}$,
  $\mathcal{B}_{\le  L_j}$-measurable for some $L_j \le  N_j^{1 -
  \delta}$, and satisfying the bounds
  \begin{equation}
    \label{Eqn:eee1} 
      \| h^{N_j, L_j}_{k_j k_j'} \|_{\ell_{k_j}^2 \to  \ell^2_{k_j'}}
      \lesssim L_j^{- \delta_0},\quad \| h^{N_j, L_j}_{k_j k_j'} \|_{\ell_{k_j k_j'}^2} \lesssim
      N_j^{\frac{5-2\alpha}{2} + \gamma_0} L_j^{ -\frac{\gamma_0}{4}}.
  \end{equation}
 Moreover, by
  \eqref{induct3} we may assume that $h^{N_j, L_j}_{k_j k_j'}$ is supported in $\{| k_j
  - k_j' | \lesssim N_j^{\kappa_0^{-\frac{1}{2}}} L_j\}$.
  
  \item Type (D), where $(w_{N_j})_{k_j}$ is supported in $\{ \jb{k_j} \lesssim
  N_j \}$, and satisfies
  \begin{equation}
    \label{Eqn:eee2} 
    \| (w_{N_j})_{k_j} \|_{\ell_{k_j}^2} \lesssim N_j^{\frac{10}{3}-2\alpha +\gamma} .
  \end{equation}
\end{itemize}
Note that we omit the factor $\jb{\lambda_j}^b$ and integral for $\lambda_j$ to simplify the notation. Considering the types of $w_{N_j}$, we should discuss the following situations:

\gb
\paragraph{(1) Type (C,C,C)}

In this situation, the trilinear form can be written as
\[
  \mathscr{X}_k = \sum_{k_1, k_2, k_3} \mathrm{T}^{\text{b},
  m}_{k k_1 k_2 k_3} \cdot \sum_{\substack{
    k_1', k'_2, k_3'}} h_{k_1 k'_1}^{N_1, L_1} \overline{h_{k_2 k'_2}^{N_2, L_2}}
  h_{k_3 k'_3}^{N_3, L_3} \frac{g_{k'_1} \overline{g_{k'_2}}g_{k'_3}}{\jbb{k'_1}^{\frac{\alpha}{2}} \jbb{k'_2}^{\frac{\alpha}{2}}
  \jbb{k'_3}^{\frac{\alpha}{2}}}
\]
\noi
where $h^{N_j, L_j}_{k_j k_j'}$ satisfies
\eqref{Eqn:eee1} and ${\rm T}^{{\rm b},m}_{k k_1 k_2 k_3}$ is defined in \eqref{baseT}.  We begin with the no pairing case, i.e. $k_2'\neq k_1', k_3'$. Apply Proposition \ref{Prop:tensor} and Proposition \ref{PROP:puretensor} repeatedly to get

\[
\begin{aligned}
    \|\mathscr{X}_k\|_{k}&\lesssim (N_1 N_2 N_3)^{-\frac{\alpha}{2}}\left\|\sum_{k_1, k_2, k_3} \mathrm{T}^{\text{b},
  m}_{k k_1 k_2 k_3} \cdot h_{k_1 k'_1}^{N_1, L_1} \overline{h_{k_2 k'_2}^{N_2, L_2}}
  h_{k_3 k'_3}^{N_3, L_3}\right\|_{kk_1'k_2'k_3'}\\
  & \lesssim (N_1 N_2 N_3)^{-\frac{\alpha}{2}}\|\mathrm{T}^{\text{b},
  m}_{k k_1 k_2 k_3}\|_{kk_1k_2k_3}\cdot\prod_{j=1}^3\|h_{k_j k'_j}^{N_j, L_j}\|_{k_j\to k_j'}\, \lesssim (N_1 N_2 N_3)^{-\frac{\alpha}{2}} |S|^{\frac{1}{2}}.
\end{aligned}
\]

Here we want to apply Lemma \ref{4vec1}. By the first two bounds, we have

\begin{equation}\label{LHLCCC1}
|S|\lesssim
\begin{cases}
 N_{2}^{\frac{8}{3}-\alpha}N_{1}^{2}N_{3}^2, \quad &\text{if}\,  N_{2} \geq N_{\med},\\
 (N_1\wedge N_3)^{\frac{8}{3}-\alpha}N_{2}^2 N^2, \quad &\text{if}\,  N_{2} = N_{\min},
\end{cases}
\end{equation}
\noi
which gives the bound

\begin{equation}\label{LHLCCC2}
\|\mathscr{X}_k\|_{k}\lesssim
\begin{cases}
 N_2^{\frac{4}{3}-\alpha}(N_1N_3)^{1-\frac{\alpha}{2}}, \quad &\text{if}\,  N_{2} \geq N_{\med},\\
 (N_1\wedge N_3)^{\frac{4}{3}-\alpha} N_{2}^{1-\frac{\alpha}{2}} N(N_1\vee N_{3})^{-\frac{\alpha}{2}}, \quad &\text{if}\,  N_{2} = N_{\min}.
\end{cases}
\end{equation}

Since we assume that $N_{\max}\geq N_{\med}\gtrsim N^{1-\delta}$, combining the two cases gives

\[
    \|\mathscr{X}_k\|_{k}\lesssim N^{\frac{10}{3}-2\alpha+\delta}.
\]
This is enough to get \eqref{Eqn:remainderaim1} by our choice of parameters.

Secondly, we consider the single pairing case, i.e. $k_1'=k_2'\neq k_3'$ or $k_1'\neq k_2'=k_3'$. The two cases are similar so we only show the case $k_1'=k_2'$, which implies $N_1=N_2$. Furthermore, we may assume that $L_1\geq L_2$. (Otherwise we use the norm $\|h^{N_2,L_2}_{k_2k_2'}\|_{k_2k_2'}$ in the following estimates. ) Then we can write

\[
\begin{aligned}
    \mathscr{X}_k 
    & = \sum_{k_1,k_2,k_3} \mathrm{T}^{\text{b},
  m}_{k k_1 k_2 k_3} \sum_{k_1' \neq k_3'} h_{k_1k_1'}^{N_1,L_1} \overline{h_{k_2k_1'}^{N_2,L_2}} h_{k_3k_3'}^{N_3L_3} \frac{(|g_{k_1'}|^2 - 1) g_{k_3'}}{\jbb{ k_1'}^{\alpha} \jbb{ k_3' }^{\frac{\alpha}{2}}} \\
  & \, + \sum_{k_1,k_2,k_3} \mathrm{T}^{\text{b},
  m}_{k k_1 k_2 k_3} \sum_{k_1' \neq k_3'} h_{k_1k_1'}^{N_1,L_1} \overline{h_{k_2k_1'}^{N_2,L_2}}h_{k_3k_3'}^{N_3,L_3} \frac{1}{\jbb{ k_1'}^{\alpha}} \cdot \frac{g_{k_3'}}{\jbb{ k_3' }^{\frac{\alpha}{2}}}\\
  &  = \mathscr{X}_k^{1} + \mathscr{X}_k^{2}.
\end{aligned}
\]

Notice that $N_1=N_2$ admits at least one $N_{\med}$ among the three numbers, so $N_1=N_2\gtrsim N^{1-\delta}$. Then $\mathscr{X}_k^{1}$ can be treated similarly as the no-pairing case to get \eqref{Eqn:remainderaim1}. In fact, 

\[
\begin{aligned}
    \|\mathscr{X}_k^{1}\|_{k}&\lesssim N_1^{-\alpha} N_3^{-\frac{\alpha}{2}}\left\|\sum_{k_1, k_2, k_3} \mathrm{T}^{\text{b},
  m}_{k k_1 k_2 k_3} \cdot h_{k_1 k'_1}^{N_1, L_1} \overline{h_{k_2 k'_1}^{N_2, L_2}}
  h_{k_3 k'_3}^{N_3, L_3}\right\|_{kk_1'k_3'}\\
  & \lesssim N_1^{-\alpha} N_3^{-\frac{\alpha}{2}}\|\mathrm{T}^{\text{b},
  m}_{k k_1 k_2 k_3}\|_{kk_1k_2k_3}\cdot\prod_{j=1}^3\|h_{k_j k'_j}^{N_j, L_j}\|_{k_j\to k_j'}.
\end{aligned}
\]

The following estimates are almost the same.

\noi

For $\mathscr{X}_k^{2}$, first sum the dyadic increments in the two
contracted factors. Corollary \ref{cor:unitary-cancellation} and Remark
\ref{rem:unitary-redistribution} then reduce the estimate to representative
dyadic terms of the form

\[
    \mathscr{X}_k^{2} = \sum_{k_1,k_2,k_3} \mathrm{T}^{\text{b},
  m}_{k k_1 k_2 k_3} \sum_{k_1',k_3'}  h_{k_1k_1'}^{N_1,L_1} \overline{h_{k_2k_1'}^{N_2,L_2}} h_{k_3k_3'}^{N_3,L_3} \left( \frac{1}{\jbb{k_1'}^{\alpha}} - \frac{1}{\jbb{k_1}^{\alpha}} \right) \frac{g_{k_3'}}{\jbb{k_3'}^{\frac{\alpha}{2}}}.
\]

By \eqref{induct3}, we may assume that $|k_1-k_1'|\lesssim N_1^{\kappa_0^{-\frac{1}{2}}}L_1\ll N_1$ where $|k_1'|\sim N_1$. Hence, we have $|k_1|\sim N_1$ and

\[
\left|\frac{1}{\jbb{k_1'}^{\alpha} }- \frac{1}{\jbb{k_1}^{\alpha} }\right| \lesssim N_1^{-\alpha-1+\kappa_0^{-\frac{1}{2}}}L_1. 
\]

Now we may apply Proposition \ref{Prop:tensor} and Proposition \ref{PROP:puretensor}  to get

\[
\begin{aligned}
    \|\mathscr{X}_k^{2}\|_{k} \lesssim &  \,N_1^{-\alpha-1+\kappa_0^{-\frac{1}{2}}}L_1 N_3^{-\frac{\alpha}{2}} \cdot \left\|\sum_{k_1,k_2,k_3}{\rm T}^{{\rm b},m}_{k k_1 k_2
  k_3}\sum_{k_1'} h_{k_1 k_1'}^{N_1, L_1} \overline{h_{k_2 k_1'}^{N_2, L_2}}h_{k_3 k_3'}^{N_3, L_3} \right\|_{k k_3'\cap(k\to k_3')}\\
  \lesssim & \,N_1^{-\alpha-1+\kappa_0^{-\frac{1}{2}}}L_1 N_3^{-\frac{\alpha}{2}} \cdot \Big(\|{\rm T}^{{\rm b},m}_{k k_1 k_2
  k_3}\|_{kk_3\to k_1 k_2} + \|{\rm T}^{{\rm b},m}_{k k_1 k_2
  k_3}\|_{k\to k_1k_2k_3}\Big)\\
  & \times \left\|\sum_{k_1'} h_{k_1 k_1'}^{N_1, L_1} \overline{h_{k_2 k_1'}^{N_2, L_2}}\right\|_{k_1k_2}\cdot\|h_{k_3 k_3'}^{N_3, L_3}\|_{k_3\to k_3'}\\
  \lesssim & \,N_1^{-\alpha-1+\kappa_0^{-\frac{1}{2}}}L_1 N_3^{-\frac{\alpha}{2}} \cdot \Big(|S_{kk_3}|^\frac{1}{2}|S_{k_1k_2}|^\frac{1}{2} + |S_{k}|^\frac{1}{2}\Big)\cdot \|h_{k_1 k_1'}^{N_1, L_1}\|_{k_1k_1'}\prod_{j=2}^3\|h_{k_j k_j'}^{N_j, L_j}\|_{k_j\to k_j'}\\
  \lesssim & \, N_1^{\frac{5-4\alpha}{2}+\frac{3\gamma_0}{4}+\kappa_0^{-\frac{1}{2}}} N_3^{-\frac{\alpha}{2}}\cdot \Big(|S_{kk_3}|^\frac{1}{2}|S_{k_1k_2}|^\frac{1}{2} + |S_{k}|^\frac{1}{2}\Big).
\end{aligned}    
\]

By Lemma \ref{2vec1} and Lemma \ref{3vec1}, we have

\[
    |S_{kk_3}|^\frac{1}{2}|S_{k_1k_2}|^\frac{1}{2} + |S_{k}|^\frac{1}{2}\lesssim
    N_1^{\frac{3-\alpha}{2}}N_3^{\frac{3-\alpha}{2}} + N_1^{\frac{3-\alpha}{2}}N_3^{\frac{1}{2}+\theta}+N_1^{\frac{1}{2}}N_3.
\]
Therefore

\[
    \|\mathscr{X}_k^{2}\|_{k} \lesssim N_1^{\frac{8-5\alpha}{2}+\frac{3\gamma_0}{4}+\kappa_0^{-\frac{1}{2}}} N_3^{\frac{3-2\alpha}{2}+\frac{3\gamma_0}{4}+\kappa_0^{-\frac{1}{2}}}+ N_1^{3-2\alpha+\frac{3\gamma_0}{4}+\kappa_0^{-\frac{1}{2}}}N_3^{1-\frac{\alpha}{2}}
\]
\noi
which gives \eqref{Eqn:remainderaim1} by $N_1\gtrsim N^{1-\delta}$ and our choice of parameters.

Finally we consider the over pairing case, where we can write

\[
    \mathscr{X}_k = \sum_{k_1,k_2,k_3} \mathrm{T}^{\text{b},
  m}_{k k_1 k_2 k_3} \sum_{k_1'} h_{k_1k_1'}^{N_1,L_1} \overline{h_{k_2k_1'}^{N_2,L_2}} h_{k_3k_1'}^{N_3,L_3} \frac{|g_{k_1'}|^2 g_{k_1'}}{\jbb{ k_1'}^{\frac{3\alpha}{2}} }.
\]
Since \(N_1=N_2=N_3=:N_*\) and there are at least two high-frequency
terms, this case closes directly. By \eqref{preofgauss}, repeated
Cauchy-Schwarz and Proposition \ref{PROP:puretensor} give
\[
 \|\mathscr X_k\|_k
 \lesssim
 N_*^{-\frac{3\alpha}{2}+3\theta}
 \|{\rm T}^{{\rm b},m}\|_{kk_1k_2k_3}
 \prod_{j=1}^3\|h^{N_*,L_j}\|_{k_j\to k_j'}
 \lesssim
 N_*^{-\frac{3\alpha}{2}+3\theta}|S|^{1/2}
 \lesssim N_*^{\frac{10}{3}-2\alpha+3\theta}.
\]
Here the last step uses the fourth bound in Lemma \ref{4vec1},
\(|S|\lesssim N_*^{\frac{20}{3}-\alpha}\). This gives \eqref{Eqn:remainderaim1} by our choice of parameters.

\gb
\paragraph{(2) Type (C,C,D) or (D,C,C)}

These two situations are similar so we only consider

\[
  \mathscr{X}_k = \sum_{k_1, k_2, k_3} \mathrm{T}^{\text{b},
  m}_{k k_1 k_2 k_3} \cdot \sum_{\substack{
    k_1', k'_2}} h_{k_1 k'_1}^{N_1, L_1} \overline{h_{k_2 k'_2}^{N_2, L_2}}
   \frac{g_{k'_1} \overline{g_{k'_2}}}{\jbb{k'_1}^{\frac{\alpha}{2}} \jbb{k'_2}^{\frac{\alpha}{2}}
  } \cdot (z_{N_3})_{k_3}
\]
\noi
where $h^{N_j, L_j}_{k_j k_j'}$ satisfies
\eqref{Eqn:eee1} and $(z_{N_3})_{k_3}$ satisfies \eqref{Eqn:eee2}. Still we begin with the no pairing case, i.e. $k_1'\neq k_2'$. Apply Proposition \ref{PROP:puretensor} and Proposition \ref{Prop:tensor}  to get

\begin{equation}\label{LHLCCD1}
\begin{aligned}
    \|\mathscr{X}_k\|_{k}&\lesssim (N_1 N_2)^{-\frac{\alpha}{2}}\left\|\sum_{k_1, k_2} \mathrm{T}^{\text{b},
  m}_{k k_1 k_2 k_3} \sum_{k_1'\neq k_2'} h_{k_1 k'_1}^{N_1, L_1} \overline{h_{k_2 k'_2}^{N_2, L_2}}g_{k_1'}\overline{g_{k_2'}}
  \right\|_{k\to k_3}\cdot \|(z_{N_3})_{k_3}\|_{k_3}\\
  & \lesssim (N_1 N_2)^{-\frac{\alpha}{2}} N_3^{\frac{10}{3}-2\alpha+\gamma}\Big(\|\mathrm{T}^{\text{b},
  m}_{k k_1 k_2 k_3}\|_{kk_1k_2\to k_3} + \|\mathrm{T}^{\text{b},
  m}_{k k_1 k_2 k_3}\|_{kk_1\to k_2k_3}\\
  &\hspace{1cm} + \|\mathrm{T}^{\text{b},
  m}_{k k_1 k_2 k_3}\|_{kk_2\to k_1k_3} + \|\mathrm{T}^{\text{b},
  m}_{k k_1 k_2 k_3}\|_{k\to k_1k_2 k_3}\Big)\cdot\prod_{j=1}^2\|h_{k_j k'_j}^{N_j, L_j}\|_{k_j\to k_j'}\\
  &\lesssim (N_1 N_2)^{-\frac{\alpha}{2}} N_3^{\frac{10}{3}-2\alpha+\gamma}\Big(|S_{k_3}|^{\frac{1}{2}}+ |S_{kk_1}|^{\frac{1}{2}}|S_{k_2k_3}|^{\frac{1}{2}}+|S_{kk_2}|^{\frac{1}{2}}|S_{k_1k_3}|^{\frac{1}{2}}+|S_{k}|^{\frac{1}{2}}\Big).
\end{aligned}
\end{equation}

If $ N_3=N_{\min} $, then we choose the bounds

\begin{equation}\label{LHLCCD2}
\begin{aligned}
    \|\mathscr{X}_k\|_{k}&\lesssim (N_1 N_2)^{-\frac{\alpha}{2}} N_3^{\frac{10}{3}-2\alpha+\gamma}\Big( N_1N_2^{\frac{4}{3}-\frac{\alpha}{2}}+  (N_1\wedge N_3)^{\frac{3-\alpha}{2}} N_1^{\frac{3-\alpha}{2}} + (N_2\wedge N_3)^{\frac{4}{3}-\frac{\alpha}{2}}N_2^{\frac{4}{3}-\frac{\alpha}{2}}\\
    & \hspace{1cm}+ (N_1\wedge N_2)^{\frac{3-\alpha}{2}}N_3^{\frac{1}{2}+\theta}+(N_1\wedge N_2)^{\frac{1}{2}}N_3 \Big)\\
    &\lesssim N_2^{\frac{4}{3}-\alpha}N_1^{1-\frac{\alpha}{2}} + N_1^{\frac{3-2\alpha}{2}}N_2^{-\frac{\alpha}{2}}+N_1^{-\frac{\alpha}{2}}N_2^{\frac{4}{3}-\alpha}+N_1^{\frac{3-2\alpha}{2}}N_2^{-\frac{\alpha}{2}}+N_1^{\frac{1-\alpha}{2}}N_3^{\frac{13}{3}-2\alpha+\gamma}N_2^{-\frac{\alpha}{2}}.
\end{aligned}
\end{equation}
Since $N_1\geq N_3$ and $N_2\gtrsim N^{1-\delta}$, this proves \eqref{Eqn:remainderaim1}. Otherwise $N_{3} \geq N_{\med}\gtrsim N^{1-\delta}$, we choose the bounds

\[
\begin{aligned}
    \|\mathscr{X}_k\|_{k}&\lesssim (N_1 N_2)^{-\frac{\alpha}{2}} N_3^{\frac{10}{3}-2\alpha+\gamma}\left( ( N_2^{\frac{4}{3}-\frac{\alpha}{2}}N_1)\wedge(N_1^{\frac{3-\alpha}{2}}N_2^{\frac{1}{2}+\theta}+ N_1^{\frac{1}{2}}N_2)+ N_2^{\frac{3-\alpha}{2}} N_1^{\frac{3-\alpha}{2}}\right.\\ 
    &\hspace{1cm}\left.  + N_1^{\frac{4}{3}-\frac{\alpha}{2}}N_2^{\frac{4}{3}-\frac{\alpha}{2}} +(N_1^{\frac{4}{3}-\frac{\alpha}{2}}N_2)\wedge(N_2^{\frac{3-\alpha}{2}}N_1^{\frac{1}{2}+\theta}+N_2^{\frac{1}{2}}N_1) \right)\\
    &\lesssim N_3^{\frac{10}{3}-2\alpha+\gamma}\cdot \left((N_1N_2)^{\frac{3-2\alpha}{2}}+ (N_1\vee N_2)^{\frac{4}{3}-\alpha}(N_1\wedge N_2)^{1-\frac{\alpha}{2}} + (N_1\vee N_2)^{\frac{3-2\alpha}{2}} \right).
\end{aligned}
\]
This is enough for \eqref{Eqn:remainderaim1} since the negative power of $(N_1\vee N_2)\gtrsim N^{1-\delta}$ can slightly decrease the total power by our choice of parameters. 

\gb

Secondly, we consider the pairing case, i.e $k_1'=k_2'$ that implies $N_1=N_2$. In addition, we assume $L_1\geq L_2$ as before. Then we can write

\[
\begin{aligned}
    \mathscr{X}_k 
    & = \sum_{k_1,k_2,k_3} \mathrm{T}^{\text{b},
  m}_{k k_1 k_2 k_3} \sum_{k_1'} h_{k_1k_1'}^{N_1,L_1} \overline{h_{k_2k_1'}^{N_2,L_2}} \frac{(|g_{k_1'}|^2 - 1)} {\jbb{ k_1'}^{\alpha}}\cdot (z_{N_3})_{k_3} \\
  & \, + \sum_{k_1,k_2,k_3} \mathrm{T}^{\text{b},
  m}_{k k_1 k_2 k_3} \sum_{k_1'} h_{k_1k_1'}^{N_1,L_1} \overline{h_{k_2k_1'}^{N_2,L_2}}\frac{1}{\jbb{ k_1'}^{\alpha}}\cdot (z_{N_3})_{k_3}\\
  &  = \mathscr{X}_k^{3} + \mathscr{X}_k^{4}.
\end{aligned}
\]

Still $\mathscr{X}_k^{3}$ can be treated similarly to the no pairing case. In fact,

\[
\begin{aligned}
    \|\mathscr{X}_k^{3}\|_{k}&\lesssim N_1^{-\alpha} \left\|\sum_{k_1, k_2, k_3} \mathrm{T}^{\text{b},
  m}_{k k_1 k_2 k_3} \cdot h_{k_1 k'_1}^{N_1, L_1} \overline{h_{k_2 k'_1}^{N_2, L_2}}(|g_{k_1'}|^2 - 1)
  \right\|_{k\to k_3}\cdot\|(z_{N_3})_{k_3}\|_{k_3}\\
  & \lesssim N_1^{-\alpha} N_3^{\frac{10}{3}-2\alpha+\gamma}\|\mathrm{T}^{\text{b},
  m}_{k k_1 k_2 k_3}\|_{(kk_1k_2\to k_3)\cap(k\to k_1 k_2k_3)}\cdot\prod_{j=1}^2\|h_{k_j k'_j}^{N_j, L_j}\|_{k_j\to k_j'}.
\end{aligned}
\]

The following estimates are almost the same. After summing the dyadic
increments, Corollary \ref{cor:unitary-cancellation} and Remark
\ref{rem:unitary-redistribution} reduce $\mathscr{X}_k^{4}$ to the form

\[
    \mathscr{X}_k^{4} = \sum_{k_1,k_2,k_3} \mathrm{T}^{\text{b},
  m}_{k k_1 k_2 k_3} \sum_{k_1',k_3'}  h_{k_1k_1'}^{N_1,L_1} \overline{h_{k_2k_1'}^{N_2,L_2}} \left( \frac{1}{\jbb{k_1'}^{\alpha}} - \frac{1}{\jbb{k_1}^{\alpha}} \right)\cdot (z_{N_3})_{k_3}.
\]

By \eqref{induct3}, we may assume that $|k_1-k_1'|\lesssim N_1^{\kappa_0^{-\frac{1}{2}}}L_1$. Still $N_1=N_2$ implies that $N^{1-\delta}\lesssim N_1\leq N$. So we have $|k_1|\sim N_1$ and

\[
\left|\frac{1}{\jbb{k_1'}^{\alpha} }- \frac{1}{\jbb{k_1}^{\alpha} }\right| \lesssim N_1^{-\alpha-1+\kappa_0^{-\frac{1}{2}}}L_1. 
\]

Now we may apply Proposition \ref{Prop:tensor} and Proposition \ref{PROP:puretensor}  to get

\[
\begin{aligned}
    \|\mathscr{X}_k^{4}\|_{k} \lesssim &  \,N_1^{-\alpha-1+\kappa_0^{-\frac{1}{2}}}L_1  \cdot\min\left\{ \left\|\sum_{k_1,k_2,k_3}{\rm T}^{{\rm b},m}_{k k_1 k_2
  k_3}\sum_{k_1'} h_{k_1 k_1'}^{N_1, L_1} \overline{h_{k_2 k_1'}^{N_2, L_2}} \right\|_{k\to k_3},\right.\\
  &\hspace{0.5cm}\left.\left\|\sum_{k_1,k_2,k_3}{\rm T}^{{\rm b},m}_{k k_1 k_2
  k_3}\sum_{k_1'} h_{k_1 k_1'}^{N_1, L_1} \overline{h_{k_2 k_1'}^{N_2, L_2}} \right\|_{k k_3}\right\} \cdot\|(z_{N_3})_{k_3}\|_{k_3}\\
  \lesssim & \,N_1^{-\alpha-1+\kappa_0^{-\frac{1}{2}}}L_1 N_3^{\frac{10}{3}-2\alpha+\gamma} \cdot \min\{\|{\rm T}^{{\rm b},m}_{k k_1 k_2
  k_3}\|_{kk_1k_2\to k_3} , \\
  & \hspace{0.5cm}\|{\rm T}^{{\rm b},m}_{k k_1 k_2
  k_3}\|_{kk_3\to k_1k_2}\} \cdot \left\|\sum_{k_1'} h_{k_1 k_1'}^{N_1, L_1} \overline{h_{k_2 k_1'}^{N_2, L_2}}\right\|_{k_1k_2}\\
  \lesssim & \, N_1^{\frac{5-4\alpha}{2}+\frac{3\gamma_0}{4}+\kappa_0^{-\frac{1}{2}}}N_3^{\frac{10}{3}-2\alpha+\gamma} \cdot \min\{|S_{k_3}|^{\frac{1}{2}}, |S_{kk_3}|^{\frac{1}{2}}|S_{k_1k_2}|^{\frac{1}{2}}\}.
\end{aligned}    
\]

If $N_3\leq N_1$, we count $S_{kk_3}$ and $S_{k_1k_2}$ to get

\[
    \|\mathscr{X}_k^{4}\|_{k} \lesssim N_1^{\frac{5-4\alpha}{2}+\frac{3\gamma_0}{4}+\kappa_0^{-\frac{1}{2}}}N_3^{\frac{10}{3}-2\alpha+\gamma} \cdot N_1^{\frac{3-\alpha}{2}}N_3^{\frac{3-\alpha}{2}}\leq N_1^{\frac{8-5\alpha}{2}+\frac{3\gamma_0}{4}+\kappa_0^{-\frac{1}{2}}} N_3^{\frac{29}{6}-\frac{5\alpha}{2}+\gamma},
\]
\noi
which is enough for \eqref{Eqn:remainderaim1} since $N_1=N_2\gtrsim N^{1-\delta}$. While $N_1 \leq N_3$ is easier because we have $N_3\gtrsim N^{1-\delta}$ now. Then we count $S_{k_3}$ to get

\[
    \|\mathscr{X}_k^{4}\|_{k} \lesssim N_1^{\frac{5-4\alpha}{2}+\frac{3\gamma_0}{4}+\kappa_0^{-\frac{1}{2}}}N_3^{\frac{10}{3}-2\alpha+\gamma} \cdot N_1^{\frac{7}{3}-\frac{\alpha}{2}}\lesssim N_1^{\frac{29}{6}-\frac{5\alpha}{2}+\frac{3\gamma_0}{4}+\kappa_0^{-\frac{1}{2}}}N_3^{\frac{10}{3}-2\alpha+\gamma},
\]
\noi
which also proves \eqref{Eqn:remainderaim1} by decreasing the total power with $N_1$.

\noi
\paragraph{(3) Type (C,D,C)}

This situation admits no pairing because of the signs. We can write

\[
  \mathscr{X}_k = \sum_{k_1, k_2, k_3} \mathrm{T}^{\text{b},
  m}_{k k_1 k_2 k_3} \cdot \sum_{\substack{
    k_1', k'_3}} h_{k_1 k'_1}^{N_1, L_1} h_{k_3 k'_3}^{N_3, L_3}
   \frac{g_{k'_1} g_{k'_3}}{\jbb{k'_1}^{\frac{\alpha}{2}} \jbb{k'_3}^{\frac{\alpha}{2}}
  } \cdot \overline{(z_{N_2})_{k_2}}
\]
\noi
where $h^{N_j, L_j}_{k_j k_j'}$ satisfies
\eqref{Eqn:eee1} and $(z_{N_2})_{k_2}$ satisfies \eqref{Eqn:eee2}.Apply Proposition \ref{PROP:puretensor} and Proposition \ref{Prop:tensor}  to get

\[
\begin{aligned}
    \|\mathscr{X}_k\|_{k}&\lesssim (N_1 N_3)^{-\frac{\alpha}{2}}\left\|\sum_{k_1, k_2} \mathrm{T}^{\text{b},
  m}_{k k_1 k_2 k_3} \sum_{k_1', k_3'} h_{k_1 k'_1}^{N_1, L_1} h_{k_3 k'_3}^{N_3, L_3}g_{k_1'}g_{k_3'}
  \right\|_{k\to k_2}\cdot \|(z_{N_2})_{k_2})\|_{k_2}\\
  & \lesssim (N_1 N_3)^{-\frac{\alpha}{2}} N_2^{\frac{10}{3}-2\alpha+\gamma}\Big(\|\mathrm{T}^{\text{b},
  m}_{k k_1 k_2 k_3}\|_{kk_1k_3\to k_2} + \|\mathrm{T}^{\text{b},
  m}_{k k_1 k_2 k_3}\|_{kk_1\to k_2k_3}\\
  &\hspace{1cm} + \|\mathrm{T}^{\text{b},
  m}_{k k_1 k_2 k_3}\|_{kk_3\to k_1k_2} + \|\mathrm{T}^{\text{b},
  m}_{k k_1 k_2 k_3}\|_{k\to k_1k_2 k_3}\cdot\prod_{j=1}^2\|h_{k_j k'_j}^{N_j, L_j}\|_{k_j\to k_j'}\\
  &\lesssim (N_1 N_3)^{-\frac{\alpha}{2}} N_2^{\frac{10}{3}-2\alpha+\gamma}\Big(|S_{k_2}|^{\frac{1}{2}}+ |S_{kk_1}|^{\frac{1}{2}}|S_{k_2k_3}|^{\frac{1}{2}}+|S_{kk_3}|^{\frac{1}{2}}|S_{k_1k_2}|^{\frac{1}{2}}+|S_{k}|^{\frac{1}{2}}\Big).
\end{aligned}
\]

If $N_2=N_{\min}$, then we have 

\[
\begin{aligned}
    \|\mathscr{X}_k\|_{k}&\lesssim (N_1 N_3)^{-\frac{\alpha}{2}} N_2^{\frac{10}{3}-2\alpha+\gamma}\Big( (N_1\wedge N_3)^{\frac{4}{3}-\frac{\alpha}{2}}N+ N_2^{\frac{3-\alpha}{2}} N_1^{\frac{3-\alpha}{2}} + N_2^{\frac{3-\alpha}{2}}N_3^{\frac{3-\alpha}{2}}+ (N_1\wedge N_3)^{\frac{4}{3}-\frac{\alpha}{2}}N_2 \Big)\\
    &\lesssim (N_1\vee N_3)^{-\frac{\alpha}{2}}(N_1\wedge N_3)^{\frac{4}{3}-\alpha} N + (N_1\vee N_3)^{\frac{3-2\alpha}{2}}(N_1\wedge N_3)^{-\frac{\alpha}{2}}
\end{aligned}
\]
\noi
which proves \eqref{Eqn:remainderaim1} since $(N_1\vee N_3)\geq (N_1\wedge N_3)\gtrsim N^{1-\delta}$. Otherwise $N_2=N_{\med}$ or $N_{\max}$, so that $N_{\min} = N_1\wedge N_3$. We should use the counting estimates:

\[
\begin{aligned}
    \|\mathscr{X}_k\|_{k}&\lesssim (N_1 N_3)^{-\frac{\alpha}{2}} N_2^{\frac{10}{3}-2\alpha+\gamma}\Big( (N_1\vee N_3)^{\frac{3-\alpha}{2}}(N_1\wedge N_3)^{\frac{1}{2}+\theta}+(N_1\vee N_3)^{\frac{1}{2}}(N_1\wedge N_3) \\
    & \hspace{1cm}+ (N_2\wedge N_3)^{\frac{3-\alpha}{2}} N_1^{\frac{3-\alpha}{2}}+ (N_1\wedge N_2)^{\frac{3-\alpha}{2}}N_3^{\frac{3-\alpha}{2}} \Big)\\
    &\lesssim  N_2^{\frac{10}{3}-2\alpha+\gamma}\Big( (N_1 N_3)^{\frac{3-2\alpha}{2}} + (N_1\vee N_3)^{\frac{1-\alpha}{2}}(N_1\wedge N_3)^{1-\frac{\alpha}{2}}\Big)
\end{aligned}
\]
\noi
which proves \eqref{Eqn:remainderaim1} since $(N_1\vee N_3), \, N_2\geq N_{\med}\gtrsim N^{1-\delta}$.

\gb
\paragraph{(4) Type (C,D,D) or (D,D,C)}

These two situations are similar as well, so we only consider

\[ 
  \mathscr{X}_k = \sum_{k_1, k_2, k_3} \mathrm{T}^{\text{b},
  m}_{k k_1 k_2 k_3} \cdot \sum_{
    k_1'} h_{k_1 k'_1}^{N_1, L_1} 
   \frac{g_{k'_1}}{\jbb{k'_1}^{\frac{\alpha}{2}}} \cdot \overline{(z_{N_2})_{k_2}}\cdot(z_{N_3})_{k_3}
\]
\noi
where $h^{N_1, L_1}_{k_1 k_1'}$ satisfies
\eqref{Eqn:eee1} and $(z_{N_2})_{k_2}$, $(z_{N_3})_{k_3}$ satisfy \eqref{Eqn:eee2}.Apply Proposition \ref{PROP:puretensor} and Proposition \ref{Prop:tensor}  to get

\[
\begin{aligned}
    \|\mathscr{X}_k\|_{k}&\lesssim N_1^{-\frac{\alpha}{2}}\min\left\{\left\|\sum_{k_1} \mathrm{T}^{\text{b},
  m}_{k k_1 k_2 k_3} \sum_{k_1'} h_{k_1 k'_1}^{N_1, L_1} g_{k_1'}
  \right\|_{kk_2\to k_3},\right.\\
  &\hspace{1cm}\left.\left\|\sum_{k_1} \mathrm{T}^{\text{b},
  m}_{k k_1 k_2 k_3} \sum_{k_1'} h_{k_1 k'_1}^{N_1, L_1} g_{k_1'}
  \right\|_{k\to k_2 k_3}\right\}\cdot \prod_{j=2}^3\|(z_{N_j})_{k_j})\|_{k_j}\\
  & \lesssim N_1^{-\frac{\alpha}{2}} (N_2N_3)^{\frac{10}{3}-2\alpha+\gamma}\min\Big\{ \|\mathrm{T}^{\text{b},
  m}_{k k_1 k_2 k_3}\|_{kk_1k_2\to k_3} + \|\mathrm{T}^{\text{b},
  m}_{k k_1 k_2 k_3}\|_{kk_2\to k_1k_3},\\
  &\hspace{1cm} \|\mathrm{T}^{\text{b},
  m}_{k k_1 k_2 k_3}\|_{k\to k_1k_2 k_3} + \|\mathrm{T}^{\text{b},
  m}_{k k_1 k_2 k_3}\|_{kk_1\to k_2k_3}\Big\}\cdot\|h_{k_1 k'_1}^{N_1, L_1}\|_{k_1\to k_1'}\\
  &\lesssim N_1^{-\frac{\alpha}{2}} (N_2N_3)^{\frac{10}{3}-2\alpha+\gamma}\min\Big\{|S_{k_3}|^{\frac{1}{2}}+ |S_{kk_2}|^{\frac{1}{2}}|S_{k_1k_3}|^{\frac{1}{2}},|S_{k}|^{\frac{1}{2}}+ |S_{kk_1}|^{\frac{1}{2}}|S_{k_2k_3}|^{\frac{1}{2}}\Big\}.
\end{aligned}
\]

If $N_1=N_{\min}$, we use the first bound to get

\[
\begin{aligned}
    \|\mathscr{X}_k\|_{k}&\lesssim N_1^{-\frac{\alpha}{2}} (N_2N_3)^{\frac{10}{3}-2\alpha+\gamma}\Big(N_2^{\frac{4}{3}-\frac{\alpha}{2}}N_1 +(N_1\wedge N_3)^{\frac{4}{3}-\frac{\alpha}{2}}N_2^{\frac{4}{3}-\frac{\alpha}{2}} \Big)\\
    & \lesssim N_1^{1-\frac{\alpha}{2}}N_2^{\frac{14}{3}-\frac{5\alpha}{2}+\gamma}\cdot N_3^{\frac{10}{3}-2\alpha+\gamma} \lesssim N_2^{\frac{17}{3}-3\alpha+\gamma} N_3^{\frac{10}{3}-2\alpha+\gamma}.
\end{aligned}   
\]
This proves \eqref{Eqn:remainderaim1} since $N_2, N_3\gtrsim N^{1-\delta}$. While $N_1=N_{\med}$ or $N_{\max}$, we use the second bound to get

\[
\begin{aligned}
    \|\mathscr{X}_k\|_{k}\lesssim & N_1^{-\frac{\alpha}{2}} (N_2N_3)^{\frac{10}{3}-2\alpha+\gamma}\Big(((N_1\wedge N_3)^{\frac{4}{3}-\frac{\alpha}{2}}N_2)\wedge(N_1^{\frac{3-\alpha}{2}}N_3^{\frac{1}{2}+\theta} +N_1^{\frac{1}{2}}N_3)+ (N_2\wedge N_3)^\frac{3-\alpha}{2}N_1^{\frac{3-\alpha}{2}} \Big)\\
    \lesssim & \begin{cases}
        N_1^{-\frac{\alpha}{2}}N_2^{\frac{13}{3}-2\alpha+\gamma}N_3^{\frac{14}{3}-\frac{5\alpha}{2}+\gamma} + N_1^{\frac{3-2\alpha}{2}}N_2^{\frac{29}{6}-\frac{5\alpha}{2}+\gamma} N_3^{\frac{10}{3}-2\alpha+\gamma}, \, &\text{if} \, N_3\geq N_2,\\
        (N_1^{\frac{3-2\alpha}{2}}N_3^{\frac{29}{6}-\frac{5\alpha}{2}+\gamma}  + N_1^{\frac{1-\alpha}{2}}N_3^{\frac{13}{3}-2\alpha+\gamma})\cdot N_2^{\frac{10}{3}-2\alpha+\gamma} , \, &\text{if} \, N_3\leq N_2
    \end{cases}\\
    \lesssim & (N_2\vee N_3)^{\frac{10}{3}-2\alpha+\gamma} N_{1}^{\frac{29}{6}-\frac{5\alpha}{2}+\gamma} + N^{9-5\alpha+2\gamma+\frac{4\delta}{3}}.
\end{aligned}  
\]
This also proves \eqref{Eqn:remainderaim1} by our choice of parameters.

\gb
\paragraph{(5) Type (D,C,D)}

Here we can write

\[
  \mathscr{X}_k = \sum_{k_1, k_2, k_3} \mathrm{T}^{\text{b},
  m}_{k k_1 k_2 k_3} \cdot (z_{N_1})_{k_1}\cdot \sum_{
    k_2'} \overline{h_{k_2 k'_2}^{N_2, L_2}} 
   \frac{\overline{g_{k'_2}}}{\jbb{k'_2}^{\frac{\alpha}{2}}} \cdot(z_{N_3})_{k_3}
\]
\noi
where $h^{N_2, L_2}_{k_2 k_2'}$ satisfies
\eqref{Eqn:eee1} and $(z_{N_1})_{k_1}$, $(z_{N_3})_{k_3}$ satisfy \eqref{Eqn:eee2}. Apply Proposition \ref{PROP:puretensor} and Proposition \ref{Prop:tensor}  to get

\begin{equation}\label{LHLDCD0}
\begin{aligned}
    \|\mathscr{X}_k\|_{k}&\lesssim N_2^{-\frac{\alpha}{2}}\min\Big\{\left\|\sum_{k_2} \mathrm{T}^{\text{b},
  m}_{k k_1 k_2 k_3} \sum_{k_2'} \overline{h_{k_2 k'_2}^{N_2, L_2}}\overline{ g_{k_2'}}
  \right\|_{kk_1\to k_3},\\
  &\hspace{2cm}\left\|\sum_{k_2} \mathrm{T}^{\text{b},
  m}_{k k_1 k_2 k_3} \sum_{k_2'} \overline{h_{k_2 k'_2}^{N_2, L_2}} \overline{g_{k_2'}}
  \right\|_{kk_3\to k_1}\Big\}\cdot \prod_{j=1,3}\|(z_{N_j})_{k_j})\|_{k_j}\\
  & \lesssim N_2^{-\frac{\alpha}{2}} (N_1N_3)^{\frac{10}{3}-2\alpha+\gamma}\min\Big\{ \|\mathrm{T}^{\text{b},
  m}_{k k_1 k_2 k_3}\|_{kk_1k_2\to k_3} + \|\mathrm{T}^{\text{b},
  m}_{k k_1 k_2 k_3}\|_{kk_1\to k_2k_3},\\
  &\hspace{1cm}  \|\mathrm{T}^{\text{b},
  m}_{k k_1 k_2 k_3}\|_{kk_2k_3\to k_1} + \|\mathrm{T}^{\text{b},
  m}_{k k_1 k_2 k_3}\|_{kk_3\to k_1k_2}\Big\}\cdot\|h_{k_2 k'_2}^{N_2, L_2}\|_{k_2\to k_2'}\\
  &\lesssim N_2^{-\frac{\alpha}{2}} (N_1N_3)^{\frac{10}{3}-2\alpha+\gamma}\min\Big\{|S_{k_3}|^{\frac{1}{2}}+ |S_{kk_1}|^{\frac{1}{2}}|S_{k_2k_3}|^{\frac{1}{2}}, |S_{k_1}|^{\frac{1}{2}}+ |S_{kk_3}|^{\frac{1}{2}}|S_{k_1k_2}|^{\frac{1}{2}}\Big\}.
\end{aligned}
\end{equation}

If $N_{\min}=N_2$, we have

\[
\begin{aligned}
    \|\mathscr{X}_k\|_{k}&\lesssim N_2^{-\frac{\alpha}{2}} (N_1N_3)^{\frac{10}{3}-2\alpha+\gamma}\min\Big\{N_1^{\frac{3-\alpha}{2}}N_2^{\frac{1}{2}+\theta}+N_1^{\frac{1}{2}}N_2+N_2^{\frac{3-\alpha}{2}}N_1^{\frac{3-\alpha}{2}},\\
    &\hspace{1cm} N_3^{\frac{3-\alpha}{2}}N_2^{\frac{1}{2}+\theta}+N_3^{\frac{1}{2}}N_2+N_2^{\frac{3-\alpha}{2}}N_3^{\frac{3-\alpha}{2}} \Big\}\\
    &\lesssim (N_1\vee N_3)^{\frac{10}{3}-2\alpha+\gamma}(N_1\wedge N_3)^{\frac{29}{6}-\frac{5\alpha}{2}+\gamma}
\end{aligned}
\]
\noi
which proves \eqref{Eqn:remainderaim1}. If $N_{\min} = N_1$, we use another bound of $|S_{k_3}|^{\frac{1}{2}}+ |S_{kk_1}|^{\frac{1}{2}}|S_{k_2k_3}|^{\frac{1}{2}}$ to get

\begin{equation}\label{LHLDCD1}
\begin{aligned}
    \|\mathscr{X}_k\|_{k}&\lesssim N_2^{-\frac{\alpha}{2}} (N_1N_3)^{\frac{10}{3}-2\alpha+\gamma}\Big(N_2^{\frac{4}{3}-\frac{\alpha}{2}}N_1+(N_2\wedge N_3)^{\frac{3-\alpha}{2}}N_1^{\frac{3-\alpha}{2}}\Big)\\
    &\lesssim N_2^{\frac{4}{3}-\alpha}N_3^{\frac{23}{3}-4\alpha+2\gamma}+  N_2^{-\frac{\alpha}{2}}(N_1N_3)^{\frac{29}{6}-\frac{5\alpha}{2}+\gamma},
\end{aligned}
\end{equation}
\noi
which proves \eqref{Eqn:remainderaim1} because of $N_2,N_3\gtrsim N^{1-\delta}$ and our choice of parameters. If $N_{\min}= N_3$, we choose another bound of $|S_{k_1}|^{\frac{1}{2}}+ |S_{kk_3}|^{\frac{1}{2}}|S_{k_1k_2}|^{\frac{1}{2}}$ to get

\begin{equation}\label{LHLDCD2}
\begin{aligned}
    \|\mathscr{X}_k\|_{k}&\lesssim N_2^{-\frac{\alpha}{2}} (N_1N_3)^{\frac{10}{3}-2\alpha+\gamma}\Big( N_2^{\frac{4}{3}-\frac{\alpha}{2}}N_3 + (N_1\wedge N_2)^{\frac{3-\alpha}{2}} N_3^{\frac{3-\alpha}{2}}\Big)\\
    &\lesssim N_2^{\frac{4}{3}-\alpha}N_1^{\frac{23}{3}-4\alpha+2\gamma}+  N_2^{-\frac{\alpha}{2}}(N_1N_3)^{\frac{29}{6}-\frac{5\alpha}{2}+\gamma}.
\end{aligned}
\end{equation}
As in the case $N_{\min}=N_1$, we prove \eqref{Eqn:remainderaim1}.

\gb
\paragraph{(6) Type (D,D,D)}

Finally, we consider

\[ 
  \mathscr{X}_k = \sum_{k_1, k_2, k_3} \mathrm{T}^{\text{b},
  m}_{k k_1 k_2 k_3} \cdot (z_{N_1})_{k_1} \cdot \overline{(z_{N_2})_{k_2}}\cdot(z_{N_3})_{k_3}
\]
\noi
where  $(z_{N_j})_{k_j}$ satisfy \eqref{Eqn:eee2}. We apply Proposition \ref{PROP:puretensor} to get

\[
\begin{aligned}
    \|\mathscr{X}_k\|_{k}&\lesssim \|\mathrm{T}^{\text{b},
  m}_{k k_1 k_2 k_3}\|_{kk_{\min}\to k_{\max}k_{\med}}\prod_{j=1}^3\|(z_{N_j})_{k_j}\|_{k_j}\\
  &\lesssim (N_{\med}N_{\min})^{\frac{3-\alpha}{2}}(N_1N_2N_3)^{\frac{10}{3}-2\alpha+\gamma}\, \lesssim N_{\max}^{\frac{10}{3}-2\alpha+\gamma}N_{\med}^{\frac{29}{6}-\frac{5\alpha}{2}+\gamma}N_{\min}^{\frac{29}{6}-\frac{5\alpha}{2}+\gamma}
\end{aligned}
\]
\noi
which proves \eqref{Eqn:remainderaim1}.

\gb
\subsubsection{Low-high-low with high frequency random term}\label{LHLremainder}

Now we consider (b) when the low frequency is much less than $N^{1-\delta}$. That is to say $N_1, N_3\ll N^{1-\delta}$. The reduction step is the same as above but now we have $N_2\gg N_1\vee N_3$. Then \eqref{induct3} allows us to assume that $|k_2|\sim N_2$. To be precise, it reduces to 

\noi

\begin{equation}
  \label{Eqn:remainderaim2} 
  \| \mathscr{X}^{-}_k \|_{k} \lesssim N^{\frac{10}{3}-2\alpha+ \frac{2\gamma}{3} },
\end{equation}
\noi
where $\mathscr{X}^{-}_k $ is given by

\[
  \mathscr{X}^{-}_k = \sum_{k_1, k_2, k_3} \mathrm{T}_{k k_1 k_2
  k_3}^{\text{b}, m}\cdot (w_{N_1})_{k_1} \cdot   \sum_{ k_2' } \overline{h_{k_2
    k_2'}^{N_2, L_2}} \frac{\overline{g_{k_2'}}}{\jbb{k_2'}^{\frac{\alpha}{2}}} \cdot (w_{N_3})_{k_3}
\]
\noi
and $w_{N_j}$ are either of  type (C) or (D) as above.  Then we have the following situations, many of which can be treated as those in Section \ref{2Hterms} :

\gb
\paragraph{(1) Type (C,C,C)}

Note that there is no need to consider pairing because $N_2\gg N_1, N_3$ and the sign. We write

\[ 
  \mathscr{X}_k^{-} = \sum_{k_1, k_2, k_3} \mathrm{T}^{\text{b},
  m}_{k k_1 k_2 k_3} \cdot \sum_{\substack{
    k_1', k'_2, k_3'}} h_{k_1 k'_1}^{N_1, L_1} \overline{h_{k_2 k'_2}^{N_2, L_2}}
  h_{k_3 k'_3}^{N_3, L_3} \frac{g_{k'_1} \overline{g_{k'_2}}g_{k'_3}}{\jbb{k'_1}^{\frac{\alpha}{2}} \jbb{k'_2}^{\frac{\alpha}{2}}
  \jbb{k'_3}^{\frac{\alpha}{2}}}
\]

\noi
where $h^{N_j, L_j}_{k_j k_j'}$ satisfies
\eqref{Eqn:eee1}. We apply Proposition \ref{Prop:tensor} and Proposition \ref{PROP:puretensor} repeatedly to get

\[
\|\mathscr{X}^{-}_k\|_{k}
   \lesssim (N_1 N_2 N_3)^{-\frac{\alpha}{2}}\|\mathrm{T}^{\text{b},
  m}_{k k_1 k_2 k_3}\|_{kk_1k_2k_3}\cdot\prod_{j=1}^3\|h_{k_j k'_j}^{N_j, L_j}\|_{k_j\to k_j'}\,\lesssim (N_1 N_2 N_3)^{-\frac{\alpha}{2}} |S|^{\frac{1}{2}}.
\]
Then we can prove \eqref{Eqn:remainderaim2} as \eqref{LHLCCC1}  and \eqref{LHLCCC2}.

\gb
\paragraph{(2) Type (C,C,D) or (D,C,C)}

These two situations are similar so we only consider

\[
  \mathscr{X}^{-}_k = \sum_{k_1, k_2, k_3} \mathrm{T}^{\text{b},
  m}_{k k_1 k_2 k_3} \cdot \sum_{
    k_1', k'_2} h_{k_1 k'_1}^{N_1, L_1} \overline{h_{k_2 k'_2}^{N_2, L_2}}
   \frac{g_{k'_1} \overline{g_{k'_2}}}{\jbb{k'_1}^{\frac{\alpha}{2}} \jbb{k'_2}^{\frac{\alpha}{2}}
  } \cdot (z_{N_3})_{k_3}
\]
\noi
where $h^{N_j, L_j}_{k_j k_j'}$ satisfies
\eqref{Eqn:eee1} and $(z_{N_3})_{k_3}$ satisfies \eqref{Eqn:eee2}. Note that there is still no possibility of pairing due to $N_2\gg N_1$. As \eqref{LHLCCD1}, we apply Proposition \ref{PROP:puretensor} and Proposition \ref{Prop:tensor}  to get

\[
    \|\mathscr{X}^{-}_k\|_{k}\lesssim (N_1 N_2)^{-\frac{\alpha}{2}} N_3^{\frac{10}{3}-2\alpha+\gamma}\Big(|S_{k_3}|^{\frac{1}{2}}+ |S_{kk_1}|^{\frac{1}{2}}|S_{k_2k_3}|^{\frac{1}{2}}+|S_{kk_2}|^{\frac{1}{2}}|S_{k_1k_3}|^{\frac{1}{2}}+|S_{k}|^{\frac{1}{2}}\Big).
\]

If $N_1\geq N_3$, we get \eqref{Eqn:remainderaim2} as the estimates in \eqref{LHLCCD2}. Now we assume $N_1\leq N_3$, then

\[
\begin{aligned}
    \|\mathscr{X}^{-}_k\|_{k}&\lesssim (N_1 N_2)^{-\frac{\alpha}{2}} N_3^{\frac{10}{3}-2\alpha+\gamma} \Big( N_1 N_2^{\frac{4}{3}-\frac{\alpha}{2}}+ (N_1N_3)^{\frac{3-\alpha}{2}}+ (N_1\wedge N_3)^{\frac{4}{3}-\frac{\alpha}{2}}N_2^{\frac{4}{3}-\frac{\alpha}{2}}+ N_3^{\frac{3-\alpha}{2}}N_1^{\frac{1}{2}+\theta}+N_3^{\frac{1}{2}}N_1\Big)\\
    &\lesssim N_2^{\frac{4}{3}-\alpha}N_1^{1-\frac{\alpha}{2}}N_3^{\frac{10}{3}-2\alpha+\gamma}+ N_2^{-\frac{\alpha}{2}}N_1^{\frac{3-2\alpha}{2}}N_3^{\frac{29}{6}-\frac{5\alpha}{2}+\gamma} + N_2^{-\frac{\alpha}{2}}N_1^{1-\frac{\alpha}{2}}N_3^{\frac{23}{6}-2\alpha+\gamma}\, \lesssim N_2^{\frac{7}{3}-\frac{3\alpha}{2}}
\end{aligned}
\]
\noi
which is enough for \eqref{Eqn:remainderaim2} by our choice of parameters.

\gb
\paragraph{(3) Type (D,C,D)}

In this situation, we write

\[
  \mathscr{X}^{-}_k = \sum_{k_1, k_2, k_3} \mathrm{T}^{\text{b},
  m}_{k k_1 k_2 k_3} \cdot (z_{N_1})_{k_1}\cdot \sum_{
    k_2'} \overline{h_{k_2 k'_2}^{N_2, L_2}} 
   \frac{\overline{g_{k'_2}}}{\jbb{k'_2}^{\frac{\alpha}{2}}} \cdot(z_{N_3})_{k_3}
\]
\noi
where $h^{N_2, L_2}_{k_2 k_2'}$ satisfies
\eqref{Eqn:eee1} and $(z_{N_1})_{k_1}$, $(z_{N_3})_{k_3}$ satisfy \eqref{Eqn:eee2}. As \eqref{LHLDCD0}, we apply Proposition \ref{PROP:puretensor} and Proposition \ref{Prop:tensor}  to get

\[
    \|\mathscr{X}^{-}_k\|_{k}\lesssim N_2^{-\frac{\alpha}{2}} (N_1N_3)^{3-2\alpha+\gamma}\min\Big\{|S_{k_3}|^{\frac{1}{2}}+ |S_{kk_1}|^{\frac{1}{2}}|S_{k_2k_3}|^{\frac{1}{2}}, |S_{k_1}|^{\frac{1}{2}}+ |S_{kk_3}|^{\frac{1}{2}}|S_{k_1k_2}|^{\frac{1}{2}}\Big\}.
\]

Then we can prove \eqref{Eqn:remainderaim2} as \eqref{LHLDCD1} or \eqref{LHLDCD2}.

\mb
\subsubsection{High frequency remainder and two low frequency terms}

Here we consider (c) and (d) when the low frequency is much less than $N^{1-\delta}$. Recall the definition of $\mathcal{P}^{\pm}$; this is equivalent to

\[
\begin{aligned}
    &\|\mathcal{P}^{+}_{N,N_2,N_3}(z_N)\|_{X^b} \lesssim \tau^{\theta} (N_2\vee N_3)^{-2\delta_0} \|z_N\|_{X^b} \leq N^{\frac{10}{3}-2\alpha+\gamma},\\
    &\|\mathcal{P}^{-}_{N,N_1,N_3}(z_N)\|_{X^b} \lesssim \tau^{\theta} (N_1\vee N_3)^{-2\delta_0} \|z_N\|_{X^b} \leq N^{\frac{10}{3}-2\alpha+\gamma}.
\end{aligned}
\]

Hence, there is no need to use Lemma \ref{Prop:stb} in advance and reduce to trilinear estimates in this case.

\gb
\subsubsection{$\Gamma$-condition}\label{Gammacondi}

In this part, we consider (e). Note that the truncation $\Delta_N$ provides a lower bound for $k$. So we have the $\Gamma$-condition $|k|>\left(\frac{N^2}{4}-1\right)^{\frac{1}{2}}\geq |k_{\max}|$ where $\Gamma \sim N$.

\gb

\paragraph{5.3.4.1} Two different reductions

\noi

Actually, we only need to deal with the type (C,$\  \cdot  \ $,$\  \cdot \ $) with high-low-low frequency since other situations can be put into the remainder as above without the $\Gamma$-condition. (Note that the $\Gamma$-condition implies no low-low-low case.) To be precise, we may assume that $N_1\gg N_2 \vee N_3 $, $ \Gamma \sim N\sim |k| \sim |k_1|$ and $y_{N_1} = \zeta_{N_1,L_1}$ defined by \eqref{matrices}. Now we write

\[
  \begin{aligned}
     \widetilde{\mathscr{F}}(\lambda, k)
    = & - i \sum_{|m|\lesssim N_{\med}^{\alpha}}\int_{\mathbb{R}^3} \sum_{k_1,k_2, k_3}
    \mathcal{K} (\lambda, m + \Phi - [\Phi] + \lambda_1  - \lambda_2+ \lambda_3)\cdot \mathrm{T}_{k k_1 k_2 k_3}^{\text{b}, m}\ind_{B_\Gamma}\\
    & \hspace{2cm} \times   ( \widetilde{\zeta_{N_1,L_1}} )_{k_1} (\lambda_1)  \cdot  \overline{(\widetilde{y_{N_2}})_{k_2}} (\lambda_2)\cdot ( \widetilde{y_{N_3}}
    )_{k_3} (\lambda_3) d\lambda_1 d \lambda_2 d \lambda_3\\
    & -i \sum_{\alpha\log N_{\med}\ll r \lesssim \alpha\log N}\int_{\mathbb{R}^3} \sum_{k_1,k_2, k_3}\sum_{2^{r-1}<|m|\leq 2^r}
    \mathcal{K} (\lambda, m + \Phi - [\Phi] + \lambda_1  - \lambda_2+ \lambda_3) \\
    & \hspace{2cm} \times \mathrm{T}_{k k_1 k_2 k_3}^{\text{b}, m}\ind_{B_{\Gamma}}  ( \widetilde{\zeta_{N_1,L_1}} )_{k_1} (\lambda_1)  \cdot  \overline{(\widetilde{y_{N_2}})_{k_2}} (\lambda_2)\cdot ( \widetilde{y_{N_3}}
    )_{k_3} (\lambda_3) d\lambda_1 d \lambda_2 d \lambda_3\\
    = & (I) + (II).
  \end{aligned}
\]

\noi

We can treat $(I)$ as \eqref{remainder2H} and reduce it to the trilinear estimates:

\begin{equation}
  \label{Eqn:remainderaim3} 
  \| \mathscr{X}_k^{\Gamma} \|_{k} \lesssim N^{\frac{10}{3}-2\alpha+ \frac{2\gamma}{3} },
\end{equation}
where $\mathscr{X}^{\Gamma}_k $ is given by
\[
  \mathscr{X}_k^{\Gamma} = \sum_{k_1, k_2, k_3} \mathrm{T}_{k k_1 k_2
  k_3}^{\text{b}, m}\ind_{B_{\Gamma}} \cdot \sum_{ k_1' } h_{k_1
    k_1'}^{N_1, L_1} \frac{g_{k_1'} (\omega)}{\jbb{k_1'}^{\frac{\alpha}{2}}} \cdot  \overline{(w_{N_2}) _{k_2}} \cdot (w_{N_3})_{k_3}
\]
\noi
and $w_{N_j}$ are either of type (C) or (D) as before. Moreover, we have $|m|\lesssim N^{\alpha}_{\med}$.

However, the reduction for $(II)$ is different. In fact,

\[
   \| \mathscr{F} \|_{X^{1-b}}^2 \lesssim  \int_{\mathbb{R}} \langle \lambda \rangle^{2 (1 - b)}  \| (I) \|_{k}^2 d \lambda + \int_{\mathbb{R}} \langle \lambda \rangle^{2 (1 - b)}  \| (II) \|_{k}^2 d \lambda
\]
\noi
and again by Minkowski inequality, Lemma \ref{LEM:kerneles} and H\"older inequality, we have

\medskip

\[
\begin{aligned}
 \int_{\mathbb{R}} \langle \lambda \rangle^{2 (1 - b)}  \| (II) \|_{k}^2 d \lambda
  \lesssim & \int_{\mathbb{R}} \langle \lambda \rangle^{- 2 b} \bigg( \sum_{r} \int_{\mathbb{R}^3} \prod_{j=1}^3 \langle \lambda_j \rangle^{- b}
    \cdot \Big\|  \sum_{k_1,k_2, k_3} \langle \lambda - 2^r
    - \lambda_1  + \lambda_2 - \lambda_3 \rangle^{- 1} \\
    &  \times \sum_{2^{r-1}<|m|\leq 2^r}\mathrm{T}_{k k_1 k_2
    k_3}^{\text{b}, m}\ind_{B_{\Gamma}}\cdot \prod_{j=1}^3 ( \langle \lambda_j \rangle^b \widetilde{y_{N_j}}
    )_{k_j}^{\zeta_j} (\lambda_j) d\lambda_j \Big\|_{k}\bigg)^2 d \lambda\\  
    \lesssim & \int_{\mathbb{R}} \langle \lambda \rangle^{- 2 b}\bigg( \sum_{r} \int_{\mathbb{R}^3} \prod_{j=1}^3 \langle \lambda_j \rangle^{- b}\langle \lambda - 2^r
    - \lambda_1  + \lambda_2 - \lambda_3 \rangle^{- 1}\\
    &  \times \Big\|\sum_{k_1,k_2, k_3}\mathrm{T}_{k k_1 k_2
    k_3}^{\text{b}, r}\cdot \prod_{j=1}^3 ( \langle \lambda_j \rangle^b \widetilde{y_{N_j}}
    )_{k_j}^{\zeta_j} (\lambda_j) d\lambda_j \Big\|_{k}\bigg)^2 d \lambda \\
    \lesssim & \int_{\mathbb{R}} \langle \lambda \rangle^{- 2 b}\left(\sum_{r}\frac{2^{r(1+\varepsilon_0)}}{\langle \lambda - 2^r - \lambda_1 + \lambda_2 - \lambda_3\rangle^{ 2}\prod\limits_{j=1}^3\langle \lambda_j \rangle^{2b}} \right) \prod_{j=1}^3 d\lambda_j d\lambda\\
    &  \times \sum_{r} 2^{-r\varepsilon_0} \left\|\sum_{k_1,k_2, k_3}\frac{1}{2^{\frac{r}{2}}}\mathrm{T}_{k k_1 k_2
    k_3}^{\text{b}, r}\cdot \prod_{j=1}^3 ( \langle \lambda_j \rangle^b \widetilde{y_{N_j}}
    )_{k_j}^{\zeta_j} (\lambda_j) d\lambda_j \right\|_{L^2_{\lambda_1\lambda_2\lambda_3}\ell^2_k}^2\\
    \lesssim & \sum_{r} 2^{r(1+\varepsilon_0-2b)} \cdot \sum_{r} 2^{-r\varepsilon_0}\\
    & \times \sup_{r} \left\|\sum_{k_1,k_2, k_3}\frac{1}{2^{\frac{r}{2}}}\mathrm{T}_{k k_1 k_2
    k_3}^{\text{b}, r}\cdot \prod_{j=1}^3 ( \langle \lambda_j \rangle^b \widetilde{y_{N_j}}
    )_{k_j}^{\zeta_j} (\lambda_j) d\lambda_j \right\|_{L^2_{\lambda_1\lambda_2\lambda_3}\ell^2_k}^2
\end{aligned}
\]

\noi
where the base tensor ${\rm T}^{{\rm b},r}$ is given by \eqref{baseTr} and $y_{N_1} = \zeta_{N_1,L_1}$. By choosing $0<\varepsilon_0<2b-1$, it reduces to the following trilinear estimates:

\begin{equation}
  \label{Eqn:remainderaim4} 
  \| \mathscr{X}_k^{r} \|_{k} \lesssim N^{\frac{10}{3}-2\alpha+ \frac{2\gamma}{3} },
\end{equation}
where $\mathscr{X}^{r}_k $ is given by

\[
  \mathscr{X}_k^{r} = \frac{1}{2^{\frac{r}{2}}}\sum_{k_1, k_2, k_3} \mathrm{T}_{k k_1 k_2
  k_3}^{\text{b}, r} \cdot \sum_{ k_1' } h_{k_1
    k_1'}^{N_1, L_1} \frac{g_{k_1'} (\omega)}{\jbb{k_1'}^{\frac{\alpha}{2}}} \cdot  \overline{(w_{N_2}) _{k_2}} \cdot (w_{N_3})_{k_3}
\]
\noi
and $w_{N_j}$ are either of type (C) or (D) as before.

\gb

\paragraph{5.3.4.2}Estimates for $\mathscr{X}^{\Gamma}$.
We will show \eqref{Eqn:remainderaim3} here. The situations are as follows:

\gb
\subparagraph{(1) Type (C,C,C)}

Now we write

\[
  \mathscr{X}_k^{\Gamma} = \sum_{k_1, k_2, k_3} \mathrm{T}^{\text{b},
  m}_{k k_1 k_2 k_3}\ind_{B_{\Gamma}} \cdot \sum_{
    k_1', k'_2, k_3'} h_{k_1 k'_1}^{N_1, L_1} \overline{h_{k_2 k'_2}^{N_2, L_2}}
  h_{k_3 k'_3}^{N_3, L_3} \frac{g_{k'_1} \overline{g_{k'_2}}g_{k'_3}}{\jbb{k'_1}^{\frac{\alpha}{2}} \jbb{k'_2}^{\frac{\alpha}{2}}
  \jbb{k'_3}^{\frac{\alpha}{2}}}
\]

\noi
where $h^{N_j, L_j}_{k_j k_j'}$ satisfies
\eqref{Eqn:eee1}. This situation admits one pairing case $k_2'=k_3'$, but we still begin with the case without pairing. Apply Proposition \ref{Prop:tensor} and Proposition \ref{PROP:puretensor} repeatedly to get

\[
\|\mathscr{X}^{\Gamma}_k\|_{k}
   \lesssim (N_1 N_2 N_3)^{-\frac{\alpha}{2}}\|\mathrm{T}^{\text{b},
  m}_{k k_1 k_2 k_3}\ind_{B_{\Gamma}}\|_{kk_1k_2k_3}\cdot\prod_{j=1}^3\|h_{k_j k'_j}^{N_j, L_j}\|_{k_j\to k_j'}\, \lesssim (N_1 N_2 N_3)^{-\frac{\alpha}{2}} |S^{\Gamma}|^{\frac{1}{2}}.
\]

Note that we have $N_1\sim N$ now. Then apply Lemma \ref{Gammacounting} to derive the bound:

\[
    \|\mathscr{X}^{\Gamma}_k\|_{k}\lesssim (N_1 N_2 N_3)^{-\frac{\alpha}{2}} N^{1-\frac{\alpha}{2}+\theta}N_{\med}^{\frac{4}{3}}N_{\min}\leq N^{\frac{10}{3}-2\alpha+\theta}
\]
\noi
which proves \eqref{Eqn:remainderaim3}. 

\noi

Secondly, we consider the pairing case, i.e. $ k_2'=k_3'$ that implies $N_2=N_3$. Furthermore, we may assume that $L_2\geq L_3$. (Otherwise we use the norm $\|h^{N_3,L_3}_{k_3k_3'}\|_{k_3k_3'}$ in the following estimates. ) Then we can write

\[
\begin{aligned}
    \mathscr{X}^{\Gamma}_k 
    & = \sum_{k_1,k_2,k_3} \mathrm{T}^{\text{b},
  m}_{k k_1 k_2 k_3}\ind_{B_{\Gamma}} \sum_{k_1' ,k_2'} h_{k_1k_1'}^{N_1,L_1} \overline{h_{k_2k_2'}^{N_2,L_2}} h_{k_3k_2'}^{N_3L_3} \frac{g_{k_1'}(|g_{k_2'}|^2 - 1) }{\jbb{ k_1'}^{\frac{\alpha}{2}} \jbb{ k_2' }^{\alpha}} \\
  & \, + \sum_{k_1,k_2,k_3} \mathrm{T}^{\text{b},
  m}_{k k_1 k_2 k_3}\ind_{B_{\Gamma}} \sum_{k_1', k_2'} h_{k_1k_1'}^{N_1,L_1} \overline{h_{k_2k_2'}^{N_2,L_2}}h_{k_3k_2'}^{N_3,L_3}  \cdot \frac{g_{k_1'}}{\jbb{ k_1' }^{\frac{\alpha}{2}}}\cdot \frac{1}{\jbb{ k_2'}^{\alpha}}\\
  &  = \mathscr{X}^{\Gamma,1}_k + \mathscr{X}^{\Gamma,2}_k.
\end{aligned}
\]

Notice that  $\mathscr{X}_k^{\Gamma,1}$ can be treated like the no-pairing case to get \eqref{Eqn:remainderaim3}. So we only consider $\mathscr{X}_k^{\Gamma,2}$. After summing the dyadic increments in the contracted factors, Corollary \ref{cor:unitary-cancellation} and Remark \ref{rem:unitary-redistribution} reduce its estimate to representative terms of the form

\[
    \mathscr{X}_k^{\Gamma,2} = \sum_{k_1,k_2,k_3} \mathrm{T}^{\text{b},
  m}_{k k_1 k_2 k_3}\ind_{B_{\Gamma}} \sum_{k_1',k_2'}  h_{k_1k_1'}^{N_1,L_1} \overline{h_{k_2k_2'}^{N_2,L_2}} h_{k_3k_2'}^{N_3,L_3} \frac{g_{k_1'}}{\jbb{k_1'}^{\frac{\alpha}{2}}} \left( \frac{1}{\jbb{k_2'}^{\alpha}} - \frac{1}{\jbb{k_2}^{\alpha}} \right) .
\]

By \eqref{induct3}, we may assume that $|k_2-k_2'|\lesssim N_2^{\kappa_0^{-\frac{1}{2}}}L_2$. Hence, we have 

\[
\left|\frac{1}{\jbb{k_2'}^{\alpha} }- \frac{1}{\jbb{k_2}^{\alpha} }\right| \lesssim N_2^{-\alpha-1+\kappa_0^{-\frac{1}{2}}}L_2. 
\]

And Proposition \ref{Prop:tensor} together with Proposition \ref{PROP:puretensor}  implies that

\begin{equation}\label{gammaCCCcal}
\begin{aligned}
    \|\mathscr{X}_k^{\Gamma,2}\|_{k} \lesssim &  \,N_1^{-\frac{\alpha}{2}}N_2^{-\alpha-1+\kappa_0^{-\frac{1}{2}}}L_2  \cdot \left\|\sum_{k_1,k_2,k_3}{\rm T}^{{\rm b},m}_{k k_1 k_2
  k_3}\ind_{B_{\Gamma}}\sum_{k_2'} h_{k_1 k_1'}^{N_1, L_1} \overline{h_{k_2 k_2'}^{N_2, L_2}}h_{k_3 k_2'}^{N_3, L_3} \right\|_{k k_1'\cap(k\to k_1')}\\
  \lesssim & \,N_1^{-\frac{\alpha}{2}}N_2^{-\alpha-1+\kappa_0^{-\frac{1}{2}}}L_2  \cdot \Big(\|{\rm T}^{{\rm b},m}_{k k_1 k_2
  k_3}\ind_{B_{\Gamma}}\|_{kk_1\to k_2 k_3} + \|{\rm T}^{{\rm b},m}_{k k_1 k_2
  k_3}\ind_{B_{\Gamma}}\|_{k\to k_1k_2k_3}\Big)\\
  & \times \left\|\sum_{k_2'}  \overline{h_{k_2 k_1'}^{N_2, L_2}}h_{k_3 k_2'}^{N_3, L_3}\right\|_{k_2k_3}\cdot\|h_{k_1 k_1'}^{N_1, L_1}\|_{k_1\to k_1'}\\
  \lesssim & \, N_1^{-\frac{\alpha}{2}}N_2^{\frac{5-4\alpha}{2}+\frac{3\gamma_0}{4}+\kappa_0^{-\frac{1}{2}}} \cdot \Big(|S^{\Gamma}_{kk_1}|^\frac{1}{2}|S^{\Gamma}_{k_2k_3}|^\frac{1}{2} + |S^{\Gamma}_{k}|^\frac{1}{2}\Big).
\end{aligned}    
\end{equation}

By Lemma \ref{Gammacounting}, we know that

\[
\begin{aligned}
    \|\mathscr{X}_k^{\Gamma,2}\|_{k} &\lesssim N_1^{-\frac{\alpha}{2}}N_2^{\frac{5-4\alpha}{2}+\frac{3\gamma_0}{4}+\kappa_0^{-\frac{1}{2}}} \cdot \Big( N_2^{\frac{3-\alpha}{2}}N^{1-\frac{\alpha}{2}+\theta}+ N_2^{\frac{7}{3}-\frac{\alpha}{2}}\Big)\\
    & \lesssim N^{1-\alpha+\theta} N_2^{\frac{8-5\alpha}{2}+\frac{3\gamma_0}{4}+\kappa_0^{-\frac{1}{2}}} + N^{-\frac{\alpha}{2}}N_2^{\frac{29}{6}-\frac{5\alpha}{2}+\frac{3\gamma_0}{4}+\kappa_0^{-\frac{1}{2}}}\, \lesssim N^{1-\alpha+\theta}
\end{aligned}
\]
\noi
which proves \eqref{Eqn:remainderaim3} as well.

\gb
\subparagraph{(2) Type (C,D,D)}

We consider

\[
  \mathscr{X}^{\Gamma}_k = \sum_{k_1, k_2, k_3} \mathrm{T}^{\text{b},
  m}_{k k_1 k_2 k_3}\ind_{B_{\Gamma}} \cdot \sum_{
    k_1'} h_{k_1 k'_1}^{N_1, L_1} 
   \frac{g_{k'_1}}{\jbb{k'_1}^{\frac{\alpha}{2}}} \cdot \overline{(z_{N_2})_{k_2}}\cdot(z_{N_3})_{k_3}
\]
\noi
where $h^{N_1, L_1}_{k_1 k_1'}$ satisfies
\eqref{Eqn:eee1} and $(z_{N_2})_{k_2}$, $(z_{N_3})_{k_3}$ satisfy \eqref{Eqn:eee2}. Apply Proposition \ref{PROP:puretensor} , Proposition \ref{Prop:tensor}  and Lemma \ref{Gammacounting} to get

\[
\begin{aligned}
    \|\mathscr{X}^{\Gamma}_k\|_{k}&\lesssim N_1^{-\frac{\alpha}{2}}\left\|\sum_{k_1} \mathrm{T}^{\text{b},
  m}_{k k_1 k_2 k_3}\ind_{B_{\Gamma}} \sum_{k_1'} h_{k_1 k'_1}^{N_1, L_1} g_{k_1'}
  \right\|_{k\to k_2k_3}\cdot \prod_{j=2}^3\|(z_{N_j})_{k_j})\|_{k_j}\\
  &\lesssim N_1^{-\frac{\alpha}{2}} (N_2N_3)^{\frac{10}{3}-2\alpha+\gamma}\Big(|S^{\Gamma}_{k}|^{\frac{1}{2}}+ |S^{\Gamma}_{kk_1}|^{\frac{1}{2}}|S^{\Gamma}_{k_2k_3}|^{\frac{1}{2}}\Big)\\
  &\lesssim N_1^{-\frac{\alpha}{2}} (N_2N_3)^{\frac{10}{3}-2\alpha+\gamma}\Big( |S^{\Gamma}_{k}|^{\frac{1}{2}} + N_{\min}N^{1-\frac{\alpha}{2}+\theta}\Big)\\
  &\lesssim N_1^{-\frac{\alpha}{2}} (N_2N_3)^{\frac{10}{3}-2\alpha+\gamma}\cdot |S^{\Gamma}_{k}|^{\frac{1}{2}} + N^{1-\alpha+\delta}N_{\med}^{\frac{23}{3}-4\alpha+2\gamma}.
\end{aligned}
\]

The second term is good enough for \eqref{Eqn:remainderaim3}, so we only discuss the first term later. In fact, if $N_{3}\geq N_2$, we use the estimate of $|S^{\Gamma}_{k}|^{\frac{1}{2}}\lesssim N_3^{\frac{4}{3}-\frac{\alpha}{2}}N_2$ to get

\begin{equation}\label{leq1/2}
    N_1^{-\frac{\alpha}{2}}(N_2N_3)^{\frac{10}{3}-2\alpha+\gamma}|S^{\Gamma}_{k}|^{\frac{1}{2}}\lesssim N_1^{-\frac{\alpha}{2}}N_{3}^{\frac{14}{3}-\frac{5\alpha}{2}+\gamma}N_{2}^{\frac{13}{3}-2\alpha+\gamma}
     \lesssim N^{9-5\alpha+2\gamma}.
\end{equation}

This proves \eqref{Eqn:remainderaim3} by our choice of parameters. If $N_{3}\leq N_2$, we use Remark \ref{interpolationbounds} to get

\begin{equation}\label{geq1/2}
\begin{aligned}
    N_1^{-\frac{\alpha}{2}}(N_2N_3)^{\frac{10}{3}-2\alpha+\gamma}|S^{\Gamma}_{k}|^{\frac{1}{2}}&\lesssim N_1^{-\frac{\alpha}{2}}N_2^{\frac{13}{3}-2\alpha+\gamma-\frac{\lambda_\alpha}{2}}N_3^{\frac{14}{3}-\frac{5\alpha}{2}+\gamma+\frac{\lambda_\alpha}{2}(\alpha-\frac{2}{3})}\\
    &\lesssim N_1^{-\frac{\alpha}{2}}N_2^{\frac{10}{3}-\frac{3\alpha}{2}}N_3^{\frac{14}{3}-\frac{5\alpha}{2}+\gamma+(1-\frac{\alpha}{2}+\gamma)(\alpha-\frac{2}{3})}\\
    &\lesssim N_1^{\frac{10}{3}-2\alpha}N_3^{\frac{14}{3}-\frac{5\alpha}{2}+\frac{7\gamma}{3}+(1-\frac{\alpha}{2})(\alpha-\frac{2}{3})}.    
\end{aligned}
\end{equation}

Here we take $\lambda_\alpha=2-\alpha+2\gamma$. To eliminate powers of $N_3$, we require $\frac{14}{3}-\frac{5\alpha}{2}+(1-\frac{\alpha}{2})(\alpha-\frac{2}{3})<0$, which implies $\alpha>\frac{\sqrt{337}-7}{6}$. This holds because we have already chosen $\alpha>\frac{29}{15}$, so we get the desired bound \eqref{Eqn:remainderaim3}.

\gb
\subparagraph{(3) Type (C,C,D)}

In this situation, we write

\[ 
  \mathscr{X}^{\Gamma}_k = \sum_{k_1, k_2, k_3} \mathrm{T}^{\text{b},
  m}_{k k_1 k_2 k_3}\ind_{B_{\Gamma}} \cdot \sum_{
    k_1', k'_2} h_{k_1 k'_1}^{N_1, L_1} \overline{h_{k_2 k'_2}^{N_2, L_2}}
   \frac{g_{k'_1} \overline{g_{k'_2}}}{\jbb{k'_1}^{\frac{\alpha}{2}} \jbb{k'_2}^{\frac{\alpha}{2}}
  } \cdot (z_{N_3})_{k_3}
\]
\noi
where $h^{N_j, L_j}_{k_j k_j'}$ satisfies
\eqref{Eqn:eee1} and $(z_{N_3})_{k_3}$ satisfies \eqref{Eqn:eee2}. Note that there is no possibility of pairing due to $N_2\ll N_1$. So we apply Proposition \ref{PROP:puretensor} and Proposition \ref{Prop:tensor}  to get

\[
\begin{aligned}
    \|\mathscr{X}^{\Gamma}_k\|_{k}&\lesssim (N_1 N_2)^{-\frac{\alpha}{2}}\left\|\sum_{k_1, k_2} \mathrm{T}^{\text{b},
  m}_{k k_1 k_2 k_3}\ind_{B_{\Gamma}} \ \sum_{k_1', k_2'} h_{k_1 k'_1}^{N_1, L_1} \overline{h_{k_2 k'_2}^{N_2, L_2}}g_{k_1'}\overline{g_{k_2'}}
  \right\|_{k\to k_3}\cdot \|(z_{N_3})_{k_3}\|_{k_3}\\
  &\lesssim (N_1 N_2)^{-\frac{\alpha}{2}} N_3^{\frac{10}{3}-2\alpha+\gamma}\Big(|S^{\Gamma}_{k_3}|^{\frac{1}{2}}+ |S^{\Gamma}_{kk_1}|^{\frac{1}{2}}|S^{\Gamma}_{k_2k_3}|^{\frac{1}{2}}+|S^{\Gamma}_{kk_2}|^{\frac{1}{2}}|S^{\Gamma}_{k_1k_3}|^{\frac{1}{2}}+|S^{\Gamma}_{k}|^{\frac{1}{2}}\Big).
\end{aligned}
\]

Now if $N_2\geq N_3$, then

\[
\begin{aligned}
    \|\mathscr{X}^{\Gamma}_k\|_{k}&\lesssim (N_1 N_2)^{-\frac{\alpha}{2}} N_{3}^{\frac{10}{3}-2\alpha+\gamma}\Big( N_2^{\frac{4}{3}}N^{1-\frac{\alpha}{2}+\theta}+ N_3^{\frac{3-\alpha}{2}}N^{1-\frac{\alpha}{2}+\theta} +(N_2N_3)^{\frac{4}{3}-\frac{\alpha}{2}}+N_2^{\frac{1}{2}}N_3\Big)\\
    &\lesssim N^{1-\alpha+\theta}N_{2}^{\frac{4}{3}-\frac{\alpha}{2}}N_{3}^{\frac{10}{3}-2\alpha+\gamma} + N^{1-\alpha+\theta}N_{3}^{\frac{29}{6}-\frac{5\alpha}{2}+\gamma}\\
    &\hspace{0.5cm}+ N^{-\frac{\alpha}{2}}N_{2}^{\frac{4}{3}-\alpha}N_{3}^{\frac{14}{3}-\frac{5\alpha}{2}+\gamma} + N^{-\frac{\alpha}{2}}N_{2}^{\frac{1-\alpha}{2}}N_{3}^{\frac{13}{3}-2\alpha+\gamma}\\
    &\lesssim N^{\frac{7}{3}-\frac{3\alpha}{2}+\theta}.
\end{aligned}
\]

This proves \eqref{Eqn:remainderaim3}. Otherwise $N_2\leq N_3$, then

\[
\begin{aligned}
    \|\mathscr{X}^{\Gamma}_k\|_{k}&\lesssim (N_1 N_2)^{-\frac{\alpha}{2}} N_{3}^{\frac{10}{3}-2\alpha+\gamma}\Big( |S^{\Gamma}_{k_{3}}|^{\frac{1}{2}}+ N_2^{\frac{3-\alpha}{2}}N^{1-\frac{\alpha}{2}+\theta} +(N_2N_3)^{\frac{4}{3}-\frac{\alpha}{2}}+N_3^{\frac{4}{3}-\frac{\alpha}{2}}N_2 \Big)\\
    &\lesssim (N_1 N_2)^{-\frac{\alpha}{2}} N_{3}^{\frac{10}{3}-2\alpha+\gamma} |S^{\Gamma}_{k_{3}}|^{\frac{1}{2}}+ N^{1-\alpha+\theta}+N_1^{-\frac{\alpha}{2}}N_3^{\frac{14}{3}-\frac{5\alpha}{2}+\gamma}N_2^{1-\frac{\alpha}{2}}.
\end{aligned}
\]

By our choice of parameters, we only need to consider the first term here. Since $|S^{\Gamma}_{k_3}|\leq|S_{k_3}|$ , we have

\[
    |S^{\Gamma}_{k_3}|\lesssim (N^{1-\frac{\alpha}{2}+\theta}N_3^{\frac{\alpha}{2}}N_2^{\frac{4}{3}-\frac{\alpha}{2}})\wedge (N_1^{\frac{3-\alpha}{2}}N_2^{\frac{1}{2}}+N_1^{\frac{1}{2}}N_2).
\]
If $N_3\leq N_1^{\frac{1}{2}}$, we use the first bound to get

\begin{equation}\label{leq1/2new}
    (N_1 N_2)^{-\frac{\alpha}{2}} N_{3}^{\frac{10}{3}-2\alpha+\gamma} |S^{\Gamma}_{k_{3}}|^{\frac{1}{2}}\lesssim N^{1-\alpha+\theta}N_3^{\frac{10}{3}-\frac{3\alpha}{2}+\gamma}N_2^{\frac{4}{3}-\alpha}\lesssim N^{\frac{8}{3}-\frac{7\alpha}{4}+\gamma},
\end{equation}
While $N_3\geq N_1^{\frac{1}{2}}$, we use the second bound to get

\begin{equation}\label{geq1/2new}
     (N_1 N_2)^{-\frac{\alpha}{2}} N_{3}^{\frac{10}{3}-2\alpha+\gamma} |S^{\Gamma}_{k_{3}}|^{\frac{1}{2}}\lesssim (N_1^{\frac{3-2\alpha}{2}}+ N_1^{\frac{1-\alpha}{2}}N_2^{1-\frac{\alpha}{2}})\cdot N_3^{\frac{10}{3}-2\alpha+\gamma} \lesssim N^{\frac{19}{6}-2\alpha+\frac{\gamma}{2}}.
\end{equation}
Combining the two cases  proves \eqref{Eqn:remainderaim3} for this situation.

\gb
\subparagraph{(4) Type (C,D,C)}

This situation is similar to the above. Actually, we have

\[ 
  \mathscr{X}^{\Gamma}_k = \sum_{k_1, k_2, k_3} \mathrm{T}^{\text{b},
  m}_{k k_1 k_2 k_3}\ind_{B_{\Gamma}} \cdot \sum_{\substack{
    k_1', k'_3}} h_{k_1 k'_1}^{N_1, L_1} h_{k_3 k'_3}^{N_3, L_3}
   \frac{g_{k'_1} g_{k'_3}}{\jbb{k'_1}^{\frac{\alpha}{2}} \jbb{k'_3}^{\frac{\alpha}{2}}
  } \cdot \overline{(z_{N_2})_{k_2}}
\]
\noi
where $h^{N_j, L_j}_{k_j k_j'}$ satisfies
\eqref{Eqn:eee1} and $(z_{N_2})_{k_2}$ satisfies \eqref{Eqn:eee2}. Apply Proposition \ref{PROP:puretensor} and Proposition \ref{Prop:tensor} to get

\[
\begin{aligned}
    \|\mathscr{X}^{\Gamma}_k\|_{k}&\lesssim (N_1 N_3)^{-\frac{\alpha}{2}}\left\|\sum_{k_1, k_3} \mathrm{T}^{\text{b},
  m}_{k k_1 k_2 k_3}\ind_{B_{\Gamma}} \ \sum_{k_1', k_3'} h_{k_1 k'_1}^{N_1, L_1} h_{k_3 k'_3}^{N_3, L_3}g_{k_1'}g_{k_3'}
  \right\|_{k\to k_2}\cdot \|(z_{N_2})_{k_2}\|_{k_2}\\
  &\lesssim (N_1 N_3)^{-\frac{\alpha}{2}} N_2^{\frac{10}{3}-2\alpha+\gamma}\Big(|S^{\Gamma}_{k_2}|^{\frac{1}{2}}+ |S^{\Gamma}_{kk_1}|^{\frac{1}{2}}|S^{\Gamma}_{k_2k_3}|^{\frac{1}{2}}+|S^{\Gamma}_{kk_3}|^{\frac{1}{2}}|S^{\Gamma}_{k_1k_2}|^{\frac{1}{2}}+|S^{\Gamma}_{k}|^{\frac{1}{2}}\Big).
\end{aligned}
\]

Only the bound of $|S^{\Gamma}_{kk_3}|^{\frac{1}{2}}|S^{\Gamma}_{k_1k_2}|^{\frac{1}{2}}$ becomes worse, but

\[
    (N_1N_3)^{-\frac{\alpha}{2}}N_2^{\frac{10}{3}-2\alpha+\gamma}|S^{\Gamma}_{kk_3}|^{\frac{1}{2}}|S^{\Gamma}_{k_1k_2}|^{\frac{1}{2}}\lesssim N_1^{-\frac{\alpha}{2}}N_2^{\frac{23}{6}-2\alpha+\gamma}N_3^{\frac{1-\alpha}{2}},
\]

\noi
which is still enough for \eqref{Eqn:remainderaim3}. Other terms can be treated as above by reversing the place of $N_2$ and $ N_3$. We omit the details.

\gb
\paragraph{5.3.4.3}Estimates for $\mathscr{X}^{r}$.
\noi
We will show \eqref{Eqn:remainderaim4} here.

\gb
\subparagraph{(1) Type (C,C,C)}

Similarly, we write

\[
  \mathscr{X}_k^{r} = \frac{1}{2^{\frac{r}{2}}}\sum_{k_1, k_2, k_3} \mathrm{T}^{\text{b},
  r}_{k k_1 k_2 k_3} \cdot \sum_{
    k_1', k'_2, k_3'} h_{k_1 k'_1}^{N_1, L_1} \overline{h_{k_2 k'_2}^{N_2, L_2}}
  h_{k_3 k'_3}^{N_3, L_3} \frac{g_{k'_1} \overline{g_{k'_2}}g_{k'_3}}{\jbb{k'_1}^{\frac{\alpha}{2}} \jbb{k'_2}^{\frac{\alpha}{2}}
  \jbb{k'_3}^{\frac{\alpha}{2}}}
\]

\noi
where $h^{N_j, L_j}_{k_j k_j'}$ satisfies
\eqref{Eqn:eee1}. This situation also admits one pairing case $k_2'=k_3'$, but we consider the no-pairing case at first. Apply Proposition \ref{Prop:tensor} and Proposition \ref{PROP:puretensor} repeatedly to get

\[
\|\mathscr{X}^{r}_k\|_{k}
   \lesssim 2^{-\frac{r}{2}}(N_1 N_2 N_3)^{-\frac{\alpha}{2}}\|\mathrm{T}^{\text{b},
  r}_{k k_1 k_2 k_3}\|_{kk_1k_2k_3}\cdot\prod_{j=1}^3\|h_{k_j k'_j}^{N_j, L_j}\|_{k_j\to k_j'}\lesssim 2^{-\frac{r}{2}}(N_1 N_2 N_3)^{-\frac{\alpha}{2}} |S^{r}|^{\frac{1}{2}}.
\]

Then we use Lemma \ref{Gammacounting2} to derive

\[
    \|\mathscr{X}^{r}_k\|_{k}\lesssim 2^{-\frac{r}{2}} (N_1 N_2 N_3)^{-\frac{\alpha}{2}} \cdot 2^{\frac{r}{2}} N^{1-\frac{\alpha}{2}+\theta}N_2 N_3 \lesssim N^{3-2\alpha+\theta}
\]
\noi
which proves \eqref{Eqn:remainderaim4}. As for the pairing case, i.e. $ k_2'=k_3'$, we divide $\mathscr{X}^{r}_k$ into two parts as $\mathscr{X}^{\Gamma}_k$ and only consider the bad part

\[
    \mathscr{X}_k^{r,2} = \frac{1}{2^{\frac{r}{2}}}\sum_{k_1,k_2,k_3} \mathrm{T}^{\text{b},
  r}_{k k_1 k_2 k_3}\sum_{k_1',k_2'}  h_{k_1k_1'}^{N_1,L_1} \overline{h_{k_2k_2'}^{N_2,L_2}} h_{k_3k_2'}^{N_3,L_3} \frac{g_{k_1'}}{\jbb{k_1'}^{\frac{\alpha}{2}}} \left( \frac{1}{\jbb{k_2'}^{\alpha}} - \frac{1}{\jbb{k_2}^{\alpha}} \right) 
\]
\noi
where we assume $N_2=N_3$, $L_2\geq L_3$. By a similar computation as \eqref{gammaCCCcal}, we get

\[
    \|\mathscr{X}_k^{r,2}\|_k\lesssim 2^{-\frac{r}{2}}N_1^{-\frac{\alpha}{2}}N_2^{\frac{5-4\alpha}{2}+\frac{3\gamma_0}{4}+\kappa_0^{-\frac{1}{2}}} \cdot \Big(|S^r_{kk_1}|^\frac{1}{2}|S^{r}_{k_2k_3}|^\frac{1}{2} + |S^{r}_{k}|^\frac{1}{2}\Big).
\]

Then apply Lemma \ref{Gammacounting2} to get

\[
\begin{aligned}
    \|\mathscr{X}_k^{r,2}\|_k&\lesssim 2^{-\frac{r}{2}}N_1^{-\frac{\alpha}{2}}N_2^{\frac{5-4\alpha}{2}+\frac{3\gamma_0}{4}+\kappa_0^{-\frac{1}{2}}} \cdot 2^{\frac{r}{2}}\Big(N^{1-\frac{\alpha}{2}+\theta}N_2 + N_2^{\frac{7}{3}-\frac{\alpha}{2}}\Big)\\
    &\lesssim N^{1-\alpha+\theta}N_2^{\frac{7-4\alpha}{2}+\frac{3\gamma_0}{4}+\kappa_0^{-\frac{1}{2}}} + N^{-\frac{\alpha}{2}}N_2^{\frac{29}{6}-\frac{5\alpha}{2}+\frac{3\gamma_0}{4}+\kappa_0^{-\frac{1}{2}}}\\
    &\lesssim N^{\frac{9}{2}-3\alpha+\theta} + N^{-\frac{\alpha}{2}}.
\end{aligned}
\]

This gives \eqref{Eqn:remainderaim4} by our choice of parameters.

\gb
\subparagraph{(2) Type (C,D,D)}

Now we write

\[
  \mathscr{X}^{r}_k = \frac{1}{2^{\frac{r}{2}}}\sum_{k_1, k_2, k_3} \mathrm{T}^{\text{b},
  r}_{k k_1 k_2 k_3}  \cdot \sum_{
    k_1'} h_{k_1 k'_1}^{N_1, L_1} 
   \frac{g_{k'_1}}{\jbb{k'_1}^{\frac{\alpha}{2}}} \cdot \overline{(z_{N_2})_{k_2}}\cdot(z_{N_3})_{k_3}
\]
\noi
where $h^{N_1, L_1}_{k_1 k_1'}$ satisfies
\eqref{Eqn:eee1} and $(z_{N_2})_{k_2}$, $(z_{N_3})_{k_3}$ satisfy \eqref{Eqn:eee2}. Apply Proposition \ref{PROP:puretensor} , Proposition \ref{Prop:tensor}  and Lemma \ref{Gammacounting2} to get

\[
\begin{aligned}
    \|\mathscr{X}^{r}_k\|_{k}&\lesssim 2^{-\frac{r}{2}}N_1^{-\frac{\alpha}{2}}\left\|\sum_{k_1} \mathrm{T}^{\text{b},
  r}_{k k_1 k_2 k_3} \sum_{k_1'} h_{k_1 k'_1}^{N_1, L_1} g_{k_1'}
  \right\|_{k\to k_2k_3}\cdot \prod_{j=2}^3\|(z_{N_j})_{k_j})\|_{k_j}\\
  &\lesssim 2^{-\frac{r}{2}}N_1^{-\frac{\alpha}{2}} (N_2N_3)^{\frac{10}{3}-2\alpha+\gamma}\Big(|S^{r}_{k}|^{\frac{1}{2}}+ |S^{r}_{kk_1}|^{\frac{1}{2}}|S^{r}_{k_2k_3}|^{\frac{1}{2}}\Big)\\
  &\lesssim N_1^{-\frac{\alpha}{2}} (N_2N_3)^{\frac{10}{3}-2\alpha+\gamma}\cdot (N_{3}^{\frac{4}{3}-\frac{\alpha}{2}}N_{2})\wedge(N_2^{\frac{1}{2}}N_3) + N^{1-\alpha+\theta} N_{\med}^{\frac{23}{3}-4\alpha+2\gamma}.
\end{aligned}
\]

Then the first term need to be discussed as \eqref{leq1/2new} and \eqref{geq1/2new}, while the second term is enough for \eqref{Eqn:remainderaim4}.

\gb
\subparagraph{(3) Type (C,C,D) and Type (C,D,C)}

As discussed in the estimates of $\mathscr{X}^{\Gamma}$, the two situations are similar. So we only consider

\[ 
  \mathscr{X}^{r}_k = \frac{1}{2^{\frac{r}{2}}}\sum_{k_1, k_2, k_3} \mathrm{T}^{\text{b},
  r}_{k k_1 k_2 k_3} \cdot \sum_{
    k_1', k'_2} h_{k_1 k'_1}^{N_1, L_1} \overline{h_{k_2 k'_2}^{N_2, L_2}}
   \frac{g_{k'_1} \overline{g_{k'_2}}}{\jbb{k'_1}^{\frac{\alpha}{2}} \jbb{k'_2}^{\frac{\alpha}{2}}
  } \cdot (z_{N_3})_{k_3}
\]
\noi
where $h^{N_j, L_j}_{k_j k_j'}$ satisfies
\eqref{Eqn:eee1}, $(z_{N_3})_{k_3}$ satisfies \eqref{Eqn:eee2} and  $N_2\ll N_1$ ensures no pairing. So we apply Proposition \ref{PROP:puretensor} and Proposition \ref{Prop:tensor}  to get

\[
\begin{aligned}
    \|\mathscr{X}^{r}_k\|_{k}&\lesssim 2^{-\frac{r}{2}}(N_1 N_2)^{-\frac{\alpha}{2}}\left\|\sum_{k_1, k_2} \mathrm{T}^{\text{b},
  r}_{k k_1 k_2 k_3} \ \sum_{k_1', k_2'} h_{k_1 k'_1}^{N_1, L_1} \overline{h_{k_2 k'_2}^{N_2, L_2}}g_{k_1'}\overline{g_{k_2'}}
  \right\|_{k\to k_3}\cdot \|(z_{N_3})_{k_3}\|_{k_3}\\
  &\lesssim2^{-\frac{r}{2}} (N_1 N_2)^{-\frac{\alpha}{2}} N_3^{\frac{10}{3}-2\alpha+\gamma}\Big(|S^{r}_{k_3}|^{\frac{1}{2}}+ |S^{r}_{kk_1}|^{\frac{1}{2}}|S^{r}_{k_2k_3}|^{\frac{1}{2}}+|S^{r}_{kk_2}|^{\frac{1}{2}}|S^{r}_{k_1k_3}|^{\frac{1}{2}}+|S^{r}_{k}|^{\frac{1}{2}}\Big).
\end{aligned}
\]

Note that $N_1\geq N_2$, we have

\[
    \|\mathscr{X}^{r}_k\|_{k}\lesssim 2^{-\frac{r}{2}}(N_1 N_{2})^{-\frac{\alpha}{2}} N_{3}^{\frac{10}{3}-2\alpha+\gamma}\cdot 2^{-\frac{r}{2}}\Big( N^{1-\frac{\alpha}{2}+\theta}N_2+N_{2}N_{3}^{\frac{4}{3}-\frac{\alpha}{2}}\Big)\lesssim N^{2-\frac{3\alpha}{2}+\theta}N^{\frac{14}{3}-\frac{5\alpha}{2}}_{3} ,
\]
which gives \eqref{Eqn:remainderaim4}.

\noi
\subsubsection{Resonant}

Finally, we consider (f). By a similar reduction step, we only need to consider the following trilinear estimates:

\begin{equation}
  \label{Eqn:remainderaim5} 
  \|\mathscr{X}^{Q}_k\|_k \lesssim N^{\frac{10}{3}-2\alpha+ 
  \frac{2\gamma}{3} }.
\end{equation}

The trilinear form is defined as

\[ 
  \mathscr{X}^{Q}_k = (w_{N_1})_{k} \cdot   \overline{(w_{N_2})_{k}} \cdot (w_{N_3})_{k}
\]
\noi
where $w_{N_j}$ are either type (C) or (D) as above and $N_1=N_2=N_3=N$. Notice that for type(C), 

\[
\begin{aligned}
    \|(w_{N_j})_{k}\|_{\ell^{\infty}} &= \left\|\sum_{ k_j' } h_{k
    k_j'}^{N_j, L_j} \frac{g_{k_j'}}{\jbb{k_j'}^{\frac{\alpha}{2}}}\right\|_{\ell^{\infty}}\leq  \left\|\sum_{ k_j' } h_{k
    k_j'}^{N_j, L_j} \frac{g_{k_j'}}{\jbb{k_j'}^{\frac{\alpha}{2}}}\right\|_{\ell^2_k}\\
    &\lesssim \|h_{k
    k_j'}^{N_j, L_j}\|_{kk_j}N_j^{-\frac{\alpha}{2}}\lesssim N_j^{\frac{5-3\alpha}{2}+\gamma_0}L^{-\frac{\gamma_0}{4}}.
\end{aligned}
\]

As for type(D), we have the $\ell^{\infty}$ norm

\[
\|z_{N_j}\|_{\ell^{\infty}}\leq\|z_{N_j}\|_{\ell^2}\lesssim N_j^{\frac{10}{3}-2\alpha+\gamma}<1.
\]

To conclude, if there is at least one $w_{N_j}$ of type(D), say $w_{N_3}$, we may directly estimate

\[
    \|\mathscr{X}^{Q}_k\|_k \lesssim \prod_{j=1}^2\|w_{N_j}\|_{\ell^\infty}\cdot\|w_{N_3}\|_{\ell^2}\lesssim N^{\frac{10}{3}-2\alpha+\frac{2\gamma}{3}}.
\]
\noi
where we use the negative power from $\ell^\infty$ norm to decrease the total power. Otherwise we have the situation (C,C,C) i.e.

\[
    \mathscr{X}^{Q}_k = \sum_{k_1',k_2',k_3'}h_{kk_1'}^{N, L_1}\overline{h_{kk_2'}^{N, L_2}}h_{kk_3'}^{N, L_3}\frac{g_{k_1'}\overline{g_{k_2'}}g_{k_3'}}{\jbb{k_1'}^{\frac{\alpha}{2}}\jbb{k_2'}^{\frac{\alpha}{2}}\jbb{k_3'}^{\frac{\alpha}{2}}}.
\]
A direct estimate implies

\[
\begin{aligned}
    \|\mathscr{X}^{Q}_k\|_k &\lesssim \prod_{j=1}^2\left\|\sum_{ k_j' } h_{k
    k_j'}^{N, L_j} \frac{g_{k_j'}}{\jbb{k_j'}^{\frac{\alpha}{2}}}\right\|_{\ell^\infty}\cdot\left\|\sum_{ k_3' } h_{k
    k_3'}^{N, L_3} \frac{g_{k_3'}}{\jbb{k_3'}^{\frac{\alpha}{2}}}\right\|_{\ell^2}\\
    &\lesssim N^{\frac{15-9\alpha}{2}+3\gamma_0} \lesssim N^{\frac{10}{3}-2\alpha-\frac{5}{6}(3\alpha-5)+3\gamma_0}.    
\end{aligned}
\]
This bound is enough for \eqref{Eqn:remainderaim5} by our choice of parameters.

\section{Proof of the main results}\label{MainProof}

In this section we give a self-contained proof of Theorems \ref{LWP} and
\ref{GWP} using the estimates established in Sections
\ref{preliminaries}--\ref{esti}.  
The argument follows the globalization
scheme in \cite{DNY2024}, with the necessary modifications tailored to our situation. 
We include the details
to make these modifications explicit. We fix
\[
 s=\frac{\alpha-2}{2}-\varepsilon, \quad
 s_0=\frac{\alpha-2}{2}-\frac{\varepsilon}{2}>s,
 \qquad
 \beta=2\alpha-\frac{10}{3}-\gamma>0,
\]
where the parameters satisfy \eqref{parameters1}.  Thus \eqref{induct6}
becomes $\|z_N\|_{X^b}\leq N^{-\beta}$. Since $\alpha$,$\varepsilon$ are fixed throughout the proof, we omit the dependence on $\alpha$, $\varepsilon$ of sets and constants below for simplicity. Our main theorems follow from a countable intersection over $\varepsilon$. We usually denote $I=[-\tau,\tau]$ if not explicitly emphasized.

\subsection{Consequences of the local estimates}

We first present the forced stability result for truncated local solutions, which is useful for establishing the uniqueness and global dynamics. See also \cite[Proposition 5.5]{DNY2024}.

\begin{prop}[Stability]\label{prop:complete-stability}
Let $I=[t_0-\tau,t_0+\tau]$, and suppose that $v_N$ is the solution to \eqref{truncfNLSgauge} on $I$. Then $\tau^{-1}$-certainly,

\begin{enumerate}[(i)]
\item Suppose that $N\leq N'$, then
\begin{equation}\label{complete-commutator}
 \|v_N-\Pi_Nv_{N'}\|_{X^b(I)}\leq C N^{-\beta}.
\end{equation}

\item Suppose that $w=\Pi_Nw$ satisfies
\begin{equation}\label{complete-forced-equation}
 (i\partial_t-{\rm D}^{\alpha})w
 =\Pi_N\Nc_3(w)+\Pi_N\Qc_3(w)+e_N
\end{equation}
and
\begin{equation}\label{complete-forced-data}
 \|w(t_0)-v_N(t_0)\|_{L^2}
 +\|\eta_{2\tau}(t-t_0)\mathcal I_{t_0}e_N\|_{X^b}
 \leq A N^{-\beta},
\end{equation}
where
\[
 \mathcal I_{t_0}F(t)
 :=-i\int_{t_0}^t e^{-i(t-t'){\rm D}^{\alpha}}F(t')\,dt'.
\] Then
\begin{equation}\label{complete-forced-stability}
 \|w-v_N\|_{X^b(I)}\leq C_A N^{-\beta}.
\end{equation}
\end{enumerate}
\end{prop}

\begin{proof}
For (i), we write
\[\Pi_Nv_{N'}-v_N=\sum_{N<M\leq N'}\Pi_Ny_M=\sum_{N<M\leq N'}(\Pi_N\psi_{M,L_M}+\Pi_Nz_M),\] where $L_M$ is the largest $L$ satisfying $(M,L)\in\mathcal{K}_0$. The bound for the remainder $\Pi_Nz_M$ follows from Proposition \ref{local2}, so it suffices to bound $\Pi_N\psi_{M,L_M}$. We simply write $\psi=\psi_{M,L_M}$, then $\Pi_N\psi$ solves the equation

\begin{equation}
\begin{aligned}
 \Pi_N\psi(t)= &2\eta_{2\tau}(t-t_0)\mathcal{I}_{t_0}\Pi_N\mathcal{N}_{3}\big(\Pi_N\psi,v_{L_M},v_{L_M}\big) + 2\eta_{2\tau}(t-t_0)\mathcal{I}_{t_0}\Pi_N\mathcal{N}_{3}\big(\Pi_N^{\perp}\psi,v_{L_M},v_{L_M}\big)\\
 &+2\eta_{2\tau}(t-t_0)\mathcal{I}_{t_0}\Pi_N\mathcal{Q}_{3}\big(\Pi_N\psi,v_{L_M},v_{L_M}\big) + 2\eta_{2\tau}(t-t_0)\mathcal{I}_{t_0}\Pi_N\mathcal{Q}_{3}\big(\Pi_N^{\perp}\psi,v_{L_M},v_{L_M}\big),
\end{aligned}\notag
\end{equation}
The linear term disappears due to $N\leq\frac{M}{2}$. Now  we may control the first and third terms on the right hand side by \eqref{rao-OP}. The second and fourth terms are $\tau^{-1}$-certainly controlled by the variant $\Gamma$-condition counting as in Section \ref{Gammacondi}. The last term can be absorbed by the left hand side if we choose $\tau\ll 1$ and apply Proposition \ref{Prop:stb}. In the end we get that

\[\|\Pi_N\psi\|_{X^{b}}\lesssim {\tau}^\theta\|\Pi_N\psi\|_{X^{b}}+{\tau}^\theta M^{-\beta-\frac{\gamma}{4}},\] which proves (\ref{complete-commutator}).

For (ii), we may assume $N$ is large depending on $A$. The difference $z=w-v_N$ satisfies the equation
\[
(i\partial_t-{\rm D}^\alpha)z=\Pi_N[\mathcal{N}_{3}(z+v_N)-\mathcal{N}_{3}(v_N)]+
\Pi_N[\mathcal{Q}_{3}(z+v_N)-\mathcal{Q}_{3}(v_N)]+e_N
\]
on $J$, and $z_0=z(t_0)$ satisfies $\|z_0\|_{L^2}\leq AN^{-\beta}$. The fixed point map for $z$ is
\[
\begin{aligned}
 \Gamma(z)=\ & \eta(t-t_0)e^{-i(t-t_0){\rm D}^{\alpha}}z(t_0)
 +\eta_{2\tau}(t-t_0)\mathcal I_{t_0}\Pi_N
 \big(\Nc_3(v_N+z)-\Nc_3(v_N)\big)\\
 &+\eta_{2\tau}(t-t_0)\mathcal I_{t_0}\Pi_N
 \big(\Qc_3(v_N+z)-\Qc_3(v_N)\big)+\eta_{2\tau}(t-t_0)\mathcal I_{t_0}e_N\\.
\end{aligned}
\]
Every cubic difference contains a copy of $z$.  Terms with one high input
are bounded by \eqref{rao-OP}; terms with at least two high inputs are $\tau^{-1}$-certainly 
bounded as in Subsection \ref{2Hterms}; the low-high-low, output-truncation,
and resonant terms follow the estimates in Subsections \ref{LHLremainder},
\ref{Gammacondi}, and \eqref{Eqn:remainderaim5}, respectively.  Replacing
one copy of $z$ by $z-\widetilde z$ gives the corresponding difference
estimate. More precisely, the estimates proved in those subsections give,
uniformly for
$\|z\|_{X^b},\|\widetilde z\|_{X^b}\leq A_1N^{-\beta}$,
\begin{equation}\label{complete-cubic-lipschitz}
\begin{aligned}
 &\left\|\eta_{2\tau}\mathcal I_{t_0}\Pi_N
 \big[\Nc_3(v_N+z)-\Nc_3(v_N)
      +\Qc_3(v_N+z)-\Qc_3(v_N)\big]\right\|_{X^b}
 \leq c_{A_1}(\tau)\|z\|_{X^b},\\
 &\left\|\eta_{2\tau}\mathcal I_{t_0}\Pi_N
 \big[\Nc_3(v_N+z)-\Nc_3(v_N+\widetilde z)
      +\Qc_3(v_N+z)-\Qc_3(v_N+\widetilde z)\big]\right\|_{X^b}
 \leq c_{A_1}(\tau)\|z-\widetilde z\|_{X^b},
\end{aligned}
\end{equation}
where $c_{A_1}(\tau)\to0$ as $\tau\to0$, uniformly in $N$ and $t_0$.
Choose $A_1=2CA$ and then choose $\tau\ll 1$ so that
$c_{A_1}(\tau)\leq\frac14$.
Consequently,
\[
 \|\Gamma(z)\|_{X^b}
 \leq CA N^{-\beta}+c_{A_1}(\tau)\|z\|_{X^b},\qquad
 \|\Gamma(z)-\Gamma(\widetilde z)\|_{X^b}
 \leq c_{A_1}(\tau)\|z-\widetilde z\|_{X^b}.
\]
The map is therefore a contraction on the ball
$\{z:\|z\|_{X^b}\leq A_1N^{-\beta}\}$, which proves
\eqref{complete-forced-stability}.
\end{proof}

\begin{rem}\label{rem:complete-long-stability}
Proposition \ref{prop:complete-stability} may be iterated over finitely many
adjacent short intervals.  The resulting constant depends on the number of
intervals but is uniform in $N$.
\end{rem}

Let $\Psi_t^N$ be the flow of the truncated gauged system \eqref{truncfNLSgauge}, we can identify the following good sets:

\begin{prop}[Uniform local good sets]\label{prop:complete-goodsets}
For each $0<\tau\ll1$, there is a rotation-invariant Borel set
$E_\tau\subseteq H^s$ with $\mu_\alpha(E_{\tau}^c)\leq C_{\theta}e^{-c\tau^{-\theta}}$, such that for each dyadic $N$ and $E^N_\tau = \Pi_N E_{\tau}$,
\begin{equation}\label{complete-good-prob}
 \mu_{\alpha,N}((E_\tau^N)^c)
 +\rho_{\alpha,N}^{\circ}((E_\tau^N)^c)
 \leq C_\theta e^{-c\tau^{-\theta}},
\end{equation}
and Proposition \ref{local2}, Proposition \ref{prop:complete-stability} up to frequency $N$ hold for initial data on $E^N_{\tau}$. Moreover, the sets may be chosen so that
\begin{equation}\label{complete-good-size}
 \|v\|_{H^{s_0}}\leq \tau^{-C_0},
 \qquad\text{for every }v\in E_\tau^N,
\end{equation}
where $C_0$ is independent of $N$.
The sets are compatibly across cutoffs: if $v\in E_\tau^N$ and
$M\leq N$ is dyadic, then the solution $\Psi_t^M\Pi_Mv$ on
$[-\tau,\tau]$ has the decomposition \eqref{ansatzN}, with $N$ replaced by
$M$, and satisfies \eqref{induct1}--\eqref{induct3} and \eqref{induct6}. Proposition \ref{prop:complete-stability} also holds for $\Psi_t^M\Pi_Mv$. The same statement holds on an interval centered at $t_0$ if
$\Psi_{t_0}^Nv\in E_\tau^N$.
\end{prop}

\begin{proof}
Proposition \ref{local2} is already a simultaneous statement in all
frequency parameters: on one event it gives the induction bounds for every
dyadic cutoff and every dyadic scale.
The exceptional probabilities in Proposition \ref{local2} are summed only
over the dyadic induction scale, giving
$C_\theta e^{-c\tau^{-\theta}}$ uniformly in $N$.  And Proposition \ref{prop:complete-stability} holds $\tau^{-1}$-certainly, which is also uniform in the dyadic scale. The definition of their available sets implies the
stated cutoff compatibility. The conditions defining the event are
unchanged by a constant phase rotation, so $E_\tau^N$ is rotation
invariant. By Fernique's theorem, after choosing $C_0$ sufficiently large,
the complement of \eqref{complete-good-size} has Gaussian measure at most
$C_\theta e^{-c\tau^{-\theta}}$, uniformly in $N$.  Intersecting with this
event preserves \eqref{complete-good-prob}. Time translation gives the
centered version. Finally, Proposition
\ref{PROP:Gibbs} and H\"older's
inequality give, for $p>1$,
\[
 \rho_{\alpha,N}^{\circ}((E_\tau^N)^c)
 \leq
 \|Z_{\alpha,N}^{-1}R_N\|_{L^p(\mu_{\alpha,N})}
 \mu_{\alpha,N}((E_\tau^N)^c)^{1-\frac1p}.
\]
Changing $\theta$ proves \eqref{complete-good-prob}.
\end{proof}

To recover the Wick-ordered system from the gauged one, we need the convergence rate of the phase as follows.

\begin{lem}[Renormalized mass]\label{lem:complete-mass}
Fix
\begin{equation}\label{complete-phase-exponent}
 \frac{2-\alpha}{2}<\beta_1<
 \min\left\{\alpha-1,\, \beta-\frac{2-\alpha}{2}\right\}.
\end{equation}
There is a rotation-invariant Borel set  $\Sigma_{\rm mass}\subset H^s$ of
full $\mu_\alpha$- and $\rho_\alpha$-measure such that
\begin{equation}\label{complete-mass-limit}
 X(u):=\lim_{N\to\infty}
 \left(\frac1{(2\pi)^2}\|\Pi_Nu\|_{L^2}^2-\sigma_N\right)
\end{equation}
exists and
\begin{equation}\label{complete-mass-rate}
 \left|
 \frac1{(2\pi)^2}\|\Pi_Nu\|_{L^2}^2-\sigma_N-X(u)
 \right|
 \leq C(u)N^{-\beta_1}.
\end{equation}
More precisely, one may write $\Sigma_{\rm mass}=\bigcup_{B=1}^{\infty}\Sigma_{\rm mass}(B)$, where $\Sigma_{\rm mass}(B)$ is Borel. The constant in
\eqref{complete-mass-rate} and $|X(u)|$, as well as the bound
\begin{equation}\label{complete-L2-growth}
 \|\Pi_Nu\|_{L^2}\leq B N^{\frac{2-\alpha}{2}},
\end{equation}
is uniform for $u\in\Sigma_{\rm mass}(B)$.
\end{lem}

\begin{proof}
Under $\mu_\alpha$, the partial sums in \eqref{complete-mass-limit} are  $X_N=\sum_{\jb{k}\leq N}
 \frac{|g_k|^2-1}{\jbb{k}^{\alpha}}$. Since
\[
 \mathbb E|X-X_N|^2
 \lesssim\sum_{\jb{k}>N}\jbb{k}^{-2\alpha}
 \lesssim N^{2-2\alpha},
\]
the Wiener chaos estimate in Lemma \ref{LEM:wc}, followed by
Chebyshev's inequality and Borel--Cantelli, gives
$|X-X_N|\lesssim N^{1-\alpha+\theta}$ for dyadic $N$. Kolmogorov's maximal
inequality, applied to the independent centered variables
$(|g_k|^2-1)\jbb{k}^{-\alpha}$ in each dyadic shell, gives the same bound
for every integer $N$. This implies
\eqref{complete-mass-rate} by our choice \eqref{complete-phase-exponent}.  Define
$\Sigma_{\rm mass}(B)$ using the countable conditions
\[
 |X_N(u)-X_M(u)|\leq B N^{1-\alpha+\theta}
 \quad(M\geq N,\ M,N\text{ dyadic}, \ B\in \mathbb{N}),\qquad |X(u)|\leq B.
\]
These sets are Borel. Since
\[
 \frac1{(2\pi)^2}\|\Pi_Nu\|_{L^2}^2
 =\sigma_N+X(u)+\mathcal{O}(BN^{1-\alpha+\theta})
 \lesssim_BN^{2-\alpha},
\]
enlarging $B$ also gives \eqref{complete-L2-growth}.  Absolute continuity
$\rho_\alpha\ll\mu_\alpha$ implies the full measure of $\Sigma_{\rm mass}$ under $\rho_\alpha$.
\end{proof}

\begin{lem}[Total-variation approximation]\label{lem:complete-TV}
With $\rho_{\alpha,N}$ as in Remark \ref{truncatedGibbs},
\begin{equation}\label{complete-TV}
 \|\rho_{\alpha,N}-\rho_\alpha\|_{\rm TV}\longrightarrow0.
\end{equation}
\end{lem}
\begin{proof}
On the full Gaussian space,
\[
 d\rho_{\alpha,N}=Z_{\alpha,N}^{-1}R_N\,d\mu_\alpha,
 \qquad
 d\rho_\alpha=Z_\alpha^{-1}R\,d\mu_\alpha.
\]
Proposition \ref{PROP:Gibbs} gives $R_N\to R$ in $L^1(\mu_\alpha)$ and
$Z_{\alpha,N}\to Z_\alpha$.  The densities therefore converge in
$L^1(\mu_\alpha)$, which is \eqref{complete-TV}.
\end{proof}

\subsection{Local existence and uniqueness}

We will first show the convergence of solutions to the truncated gauged systems \eqref{truncfNLSgauge}. Indeed, we obtain the explicit rate of convergence:

\begin{prop}[Quantitative local Cauchy estimate]
\label{prop:complete-local-cauchy}
Let $v\in E_{\tau}$ and $\Psi^N_t\Pi_Nv$ be the solution to \eqref{truncfNLSgauge}, and let $N\leq N'$ be dyadic. Then 
\begin{equation}\label{complete-local-cauchy}
 \|\Psi^N_t\Pi_Nv-\Psi^{N'}_t\Pi_{N'}v\|_
 {C(I;H^s)}
 \leq C(\tau)N^{-\frac{\varepsilon}{2}}.
\end{equation}
\end{prop}

\begin{proof}
Write the difference as the sum of the dyadic pieces $y_M$ with
$N<M\leq N'$ and use \eqref{ansatzN}. For the free part, set
$s_0=\frac{\alpha-2}{2}-\frac{\varepsilon}{2}$. By \eqref{complete-good-size},
\[
\begin{aligned}
 \left\|\sum_{N<M\leq N'}
 e^{-it{\rm D}^{\alpha}}\Delta_Mv\right\|_{C(I;H^s)}
 &\leq N^{s-s_0}\|v\|_{H^{s_0}} \leq \tau^{-C_0}N^{-\frac{\varepsilon}{2}}.
\end{aligned}
\]

For the remainder, $s<0$, so the embedding
$X^b\hookrightarrow C(I;L^2)\hookrightarrow C(I;H^s)$ and
\eqref{induct6} give
\[
 \left\|\sum_{N<M\leq N'}z_M\right\|_{C(I;H^s)}
 \lesssim\sum_{M>N}M^{-\beta}
 \lesssim N^{-\beta}
 \lesssim N^{-\frac{\varepsilon}{2}}.
\]

Finally, \eqref{RAOstructure}, \eqref{matrices}, and \eqref{induct2} give
\[
 \|\zeta_{M,L}\|_{X^b}
 \lesssim
 M^{\frac{5-3\alpha}{2}+\gamma_0}
 L^{-\frac{\gamma_0}{4}}.
\]
The weighted off-diagonal bound \eqref{induct3} allows the output to be
decomposed into dyadic annuli around $|k|\sim M$; the contributions away
from this annulus are summable because $\kappa_0\gg1$. Hence
\[
 \left\|\sum_{\substack{N<M\leq N'\\(M,L)\in\mathcal K_0}}
 \zeta_{M,L}\right\|_{C(I;H^s)}\lesssim
 \sum_{M>N}M^s
 \sum_{L:(M,L)\in\mathcal K_0}\|\zeta_{M,L}\|_{X^b}\lesssim
 \sum_{M>N} M^{\frac{3-2\alpha}{2}+\gamma_0-\varepsilon}\lesssim N^{-\frac{\varepsilon}{2}},
\]
where all sums are dyadic and the last inequality follows from
\eqref{parameters1}. Combining the three bounds proves
\eqref{complete-local-cauchy}.
\end{proof}

To prove uniqueness, we state the definition of the structured solution
class as follows.

\begin{df}\label{structured-class}
For $u_0\in H^s$, let $\mathcal U_{u_0}(I)$ be the structured solution
class consisting of distributional solutions $v\in C(I; H^s)$ to the gauged equation \eqref{Eqn:gaugedlimit}
with $v(0)=u_0$, such that, for every dyadic $N$, $v_N^\sharp=\Pi_Nv$
satisfies
\begin{equation}\label{complete-admissible-defect}
\begin{aligned}
 &(i\partial_t-{\rm D}^{\alpha})v_N^\sharp
 -\Pi_N\Nc_3(v_N^\sharp)-\Pi_N\Qc_3(v_N^\sharp)=e_N,\\
 &\|\eta_{2\tau}\mathcal I e_N\|_{X^b(I)}
 \leq C_vN^{-\beta}.
\end{aligned}
\end{equation}    

\end{df}

\begin{prop}[Local limit and uniqueness]\label{prop:complete-local}
Outside a set of probability at most $C_\theta e^{-c\tau^{-\theta}}$, the
gauged solutions $\Psi_t^N\Pi_Nu^\omega$ converge in $C(I;H^s)$ to a limit
$v:=\Psi_tu^\omega$.  This limit is the unique solution in
$\mathcal {U}_{u^\omega}(I)$.
\end{prop}

\begin{proof}
The convergence $v_N=\Psi^N_t\Pi_N u^\omega \to v$ is clear by Proposition \ref{prop:complete-local-cauchy}. Note that $E_{\tau}\subseteq H^s$ with $\mu_\alpha(E_\tau^c)< C_\theta e^{-c\tau^{-\theta}}$ corresponds to a set $\Omega_{\tau}\subseteq \Omega$ with $\mathbb{P}(\Omega_\tau^c)< C_\theta e^{-c\tau^{-\theta}}$ by $u^\omega$. Moreover, for every test function $\varphi\in C_c^\infty(I\times\mathbb T^2)$ and admissible $\omega$, the equation
\[
 \Pi_N(\Nc_3(v_N)+\Qc_3(v_N))
 =(i\partial_t-{\rm D}^{\alpha})v_N.
\]
and the convergence of $v_N$ in $C(I;H^s)$ imply
\[
 \left\langle\Pi_N(\Nc_3(v_N)+\Qc_3(v_N)),\varphi\right\rangle
 \longrightarrow
 \left\langle(i\partial_t-{\rm D}^{\alpha})v,\varphi\right\rangle.
\]
This limit defines the renormalized gauged nonlinearity and proves that the
limit solves the gauged equation \eqref{Eqn:gaugedlimit} in distributions.

Next, we show that the limit $v$ belongs to $\mathcal{U}_{u^\omega}(I)$ with high probability. Note that $v$ admits \eqref{ansatz1} together with
\eqref{induct1}--\eqref{induct3} and \eqref{induct6} on $E_\tau$.  Expanding the
high-frequency tails shows the desired error bound, by similar multilinear estimates used in Section \ref{esti} and the variant of $\Gamma$-condition, which remove an exceptional set of probability $C_\theta e^{-c\tau^{-\theta}}$. Indeed, the error in
\eqref{complete-admissible-defect} is
\[
 e_N=\Pi_N\big(\Nc_3(v)-\Nc_3(\Pi_Nv)\big)
     +\Pi_N\big(\Qc_3(v)-\Qc_3(\Pi_Nv)\big),
\]
where each term contains at least one frequency larger than $N$. Then $v\in \mathcal U_{u^\omega}(I)$ $\tau^{-1}$-certainly. By shrinking the admissible good set $E_\tau$, we have $v\in \mathcal{U}_{u^\omega}(I)$.

For uniqueness, take an arbitrary solution $\tilde{v}\in \mathcal{U}_{u^\omega}(I)$, and put $w_N=\Pi_N\tilde{v}$.
 \eqref{complete-admissible-defect} and Proposition
\ref{prop:complete-stability}(ii) give
\[
 \|w_N-\Psi_t^N\Pi_Nu^\omega\|_{X^b(I)}
 \lesssim_vN^{-\beta}.
\]
Passing to the limit and the embedding $X^b\hookrightarrow C(I;H^s)$ give $v=\tilde{v}$. In particular, there is only one solution in the class $\mathcal{U}_{u^\omega}(I)$ for each $u^\omega$. By definition, the limit $v$ is the unique solution with initial data $u^\omega$ in the class $\mathcal{U}_{u^\omega}(I)$.
\end{proof}

\begin{rem}\label{shrink-def}
To give the uniqueness in global setting, here we shrink the good set $E_\tau$ defined in Proposition \ref{prop:complete-goodsets} to make $\Psi_tv\in \mathcal{U}_{u^\omega}(I)$ for $v\in E_\tau$.
\end{rem}

Let $\Phi^N_t$ be the flow of the truncated Wick-ordered system. The finite-dimensional inverse gauge relation is $\Phi_t^N\Pi_Nu
 =e^{-2it(m_N(u)-\sigma_N)}\Psi_t^N\Pi_Nu$. Consequently, Lemma \ref{lem:complete-mass} gives on $E_\tau\cap \Sigma_{\rm mass}$
\begin{equation}\label{complete-inverse-gauge}
 \Phi_tu=e^{-2itX(u)}\Psi_tu.
\end{equation}
Moreover, \eqref{complete-local-cauchy} and
\eqref{complete-mass-rate} imply that, for every $0<\varepsilon_{\rm loc}<\min\left\{\frac{\varepsilon}{2},\beta_1\right\}$, the Wick ordered approximants satisfy
\begin{equation}\label{complete-Wick-local-cauchy}
 \|\Phi_t^N\Pi_Nu^\omega-\Phi_t^{N'}\Pi_{N'}u^\omega\|_{C(I;H^s)}
 \leq C(\omega,\tau)N^{-\varepsilon_{\rm loc}}
 \qquad(N\leq N').
\end{equation}
Indeed, the gauged approximants are uniformly bounded in $C(I;H^s)$ by
\eqref{complete-local-cauchy}, while
\[
 \sup_{t\in I}
 \left|e^{-2it(m_N-\sigma_N)}
       -e^{-2it(m_{N'}-\sigma_{N'})}\right|
 \lesssim_\tau N^{-\beta_1}
\]
by \eqref{complete-mass-rate}.
The finite-dimensional Wick ordered solutions therefore converge to the
right-hand side of \eqref{complete-inverse-gauge}. Passing to the limit in
their Duhamel formulas shows that this limit solves \eqref{WfNLS}. This
proves Theorem \ref{LWP}, including uniqueness in the inverse image of
$\mathcal U_{u^\omega}(I)$ under \eqref{complete-inverse-gauge}.

\subsection{Construction of the global flow}

The following construction is adapted from the finite-dimensional
globalization scheme in \cite[Proposition 6.1]{DNY2024}.
Let $D,T,K,A$ range over fixed countable sets of positive dyadic numbers,
with $K$ an integer and $K\gg T\gg D$.  Define
\begin{equation}\label{complete-F}
 F^N_{T,K}
 :=\bigcap_{|j|\leq K}
 \left(\Psi^N_{\frac{jT}{K}}\right)^{-1}
 E^N_{\frac{T}{K}},
\end{equation}
and
\begin{equation}\label{complete-G}
\begin{aligned}
 G^N_{T,K,A,D}
 :=\big\{v\in\Pi_NH^s:\ &\exists\,t_0\in[-D,D],\
 v'\in F^N_{T,K},\ v''\in\Pi_NL^2,\\
 &\Psi^N_{t_0}v=v'+v'',\quad
 \|v''\|_{L^2}\leq AN^{-\beta}\big\}.
\end{aligned}
\end{equation}
The existential time may be restricted to rational $t_0$ by continuity, so
the sets are Borel.  Put
\begin{equation}\label{complete-Sigma}
\Sigma_0
=\bigcup_D\ \bigcap_{T\gg D}\ \bigcup_{K\gg T,\ A\geq1}
\limsup_{N\to\infty}\Pi_N^{-1}G^N_{T,K,A,D},
\qquad
\Sigma=\Sigma_0\cap\Sigma_{\rm mass}.
\end{equation}

\begin{lem}\label{lem:complete-fullmeasure}
The set $\Sigma$ is rotation invariant, Borel, and
$\rho_\alpha(\Sigma^c)=0$.
\end{lem}

\begin{proof}
The rotation invariance and Borel property are implied by the definition.  Finite-dimensional invariance, \eqref{complete-good-prob}, and the union
bound give
\[
 \rho_{\alpha,N}^{\circ}(F^N_{T,K})
 \geq1-(2K+1)C_\theta e^{-c(K/T)^\theta}.
\]
Lemma \ref{lem:complete-TV} and Fatou's lemma imply
\[
 \rho_\alpha\left(\limsup_N\Pi_N^{-1}F^N_{T,K}\right)
 \geq1-(2K+1)C_\theta e^{-c(K/T)^\theta}.
\]
For fixed $T$, the right side tends to one as $K\to\infty$.  Since
$F^N_{T,K}\subset G^N_{T,K,A,D}$, \eqref{complete-Sigma} has full measure.
\end{proof}

\begin{prop}[Global Cauchy estimate]\label{prop:complete-global-cauchy}
For $u_{\rm in}\in\Sigma$ and every $T>0$,
\begin{equation}\label{complete-global-cauchy}
 \sup_{|t|\leq T}
 \|\Psi_t^N\Pi_Nu_{\rm in}
 -\Psi_t^{N'}\Pi_{N'}u_{\rm in}\|_{H^s}
 \leq C(u_{\rm in},T)N^{-\frac{\varepsilon}{2}}
\end{equation}
whenever $N\leq N'$.
\end{prop}

\begin{proof}
Choose $D,T_1,K,A$ in \eqref{complete-Sigma}, with
$T_1>2T+2D$, and an arbitrarily large admissible $\overline N$ such that
$\Pi_{\overline N}u_{\rm in}\in
G^{\overline N}_{T_1,K,A,D}$.  Then
\[
 \Psi^{\overline N}_{t_0}\Pi_{\overline N}u_{\rm in}
 =v'+v'',\quad
 v'\in F^{\overline N}_{T_1,K},\quad
 \|v''\|_{L^2}\leq A\overline N^{-\beta}.
\]
At each grid time $jT_1/K$, the upper-cutoff evolution of $v'$ lies in
$E^{\overline N}_{T_1/K}$. By the cutoff compatibility in Proposition
\ref{prop:complete-goodsets}, all lower-cutoff reference solutions satisfy
the same local bounds on the corresponding grid interval. Apply
\eqref{complete-commutator} on the next grid interval and
\eqref{complete-forced-stability} to propagate the accumulated error.
Induction in $j$ gives
\begin{equation}\label{complete-grid-commutator}
 \sup_{|t|\leq T_1}
 \|\Pi_N\Psi_t^{\overline N}v'
 -\Psi_t^N\Pi_Nv'\|_{L^2}
 \leq C_{T_1,K}N^{-\beta}.
\end{equation}
Applying Proposition \ref{prop:complete-local-cauchy} and Proposition \ref{prop:complete-stability} (i) on each grid interval, \eqref{complete-grid-commutator}
gives
\begin{equation}\label{complete-vprime-cauchy}
 \sup_{|t|\leq T_1}
 \|\Psi_t^N\Pi_Nv'-\Psi_t^{N'}\Pi_{N'}v'\|_{H^s}
 \leq C_{T_1,K}N^{-\frac{\varepsilon}{2}}.
\end{equation}

Then we change $v'$ to $u_{\rm in}$. Applying stability backward from
$t_0$ first gives
\begin{equation}\label{perturb-u}
\left\|\Pi_{\overline{N}}u_{\rm in}-\Psi^{\overline{N}}_{-t_0}v'\right\|_{L^2}\leq C_{T_1,K,A,D}\overline{N}^{-\beta}   
\end{equation}
Since $\overline N^{-\beta}\leq N^{-\beta}$, we have $\left\|\Pi_Nu_{\rm in}
 -\Pi_N\Psi_{-t_0}^{\overline N}v'\right\|_{L^2}
 \leq C_{T_1,K,A,D}\overline N^{-\beta}.$ The grid commutator estimate compares
$\Pi_N\Psi_{-t_0}^{\overline N}v'$ with
$\Psi_{-t_0}^N\Pi_Nv'$ up to $C_{T_1,K}N^{-\beta}$. Forced stability for
the lower-cutoff equation, now run forward from time $0$, therefore yields
\begin{equation}\label{complete-link}
 \sup_{|t|\leq T}
 \|\Psi_t^N\Pi_Nu_{\rm in}
 -\Psi_{t-t_0}^N\Pi_Nv'\|_{L^2}
 \leq C_{T_1,K,A,D}N^{-\beta}.
\end{equation}
Combining \eqref{complete-vprime-cauchy} with \eqref{complete-link} proves
\eqref{complete-global-cauchy} for
$N\leq N'\leq\overline N$.  Given fixed $N\leq N'$, we may choose an admissible large $\overline N\geq N'$, which finishes the global Cauchy estimates.
\end{proof}

The gauged limit $\Psi_tu_{\rm in}$ therefore exists in
$C([-T,T];H^s)$ for every $T$ and $u_{\rm in}\in \Sigma$.  By Lemma \ref{lem:complete-mass}, we have
\begin{equation}\label{global-inverse-gauge}
 \Phi_tu_{\rm in}
 =\lim_{N\to\infty}\Phi_t^N\Pi_Nu_{\rm in}
 =e^{-2itX(u_{\rm in})}\Psi_tu_{\rm in}.    
\end{equation}
Passing to the limit in the Duhamel formula gives the Wick ordered equation. Local uniqueness may be iterated on the grid intervals, giving global
uniqueness in the propagated structured class, which can be defined piecewise by gluing the local structured solution class $\mathcal{U}_{u_{\rm in}}(I)$.

\subsection{Group property}

We first verify the preservation of $\Sigma_0$. More precisely, $\Phi_{t_1}u_{\rm in}\in\Sigma_0$ for any $t_1\in\mathbb{R}$ and $u_{\rm in}\in\Sigma_0$. Since $\Sigma_0$ is rotation invariant, it suffices to show $\Psi_{t_1}u_{\rm in}\in \Sigma_0$. We still assume the decomposition at $t_0$ for fixed $D,T,K,A,\overline{N}$. Proposition \ref{prop:complete-stability} implies \eqref{perturb-u} and similarly 
 \begin{equation}\label{continue}
    \left\|\Psi_{t_1}^{\overline{N}}\Pi_{\overline{N}}u_{\rm in}-\Psi^{\overline{N}}_{t_1-t_0}v'\right\|_{L^2}\leq C_{T,K,A,D}\overline{N}^{-\beta}.
 \end{equation}
Fix a lower cutoff $N$ and send the admissible upper cutoff
$\overline N\to\infty$ in
\eqref{complete-grid-commutator} and \eqref{complete-link}. This gives the
stronger low-frequency comparison for $u_{\rm in}\in \Sigma$
\begin{equation}\label{complete-finite-limit}
 \sup_{|t|\leq T}
 \|\Psi_t^N\Pi_Nu_{\rm in}-\Pi_N\Psi_tu_{\rm in}\|_{L^2}
 \leq C(u_{\rm in},T)N^{-\beta}.
\end{equation}
Hence, we have
\begin{equation}\label{continue2}
   \left\|\Pi_{\overline{N}}\Psi_{t_1}u_{\rm in}-\Psi^{\overline{N}}_{t_1-t_0}v'\right\|_{L^2}\leq C_{T,K,A,D}\overline{N}^{-\beta}. 
\end{equation}
Then apply Proposition \ref{prop:complete-stability} (ii) to get $\left\|\Psi_{t_0-t_1}^{\overline{N}}\Pi_{\overline{N}}\Psi_{t_1}u_{\rm in}-v'\right\|_{L^2}\leq C_{T,K,A,D}\overline{N}^{-\beta}$, which shows that $\Pi_{\overline{N}}\Psi_{t_1}u_{\rm in}\in G^{\overline{N}}_{T,K,B,D_1}$ for the fixed $T,K,\overline{N}$ and $B=B(T,K,A,D)$, $D_1\gg D+|t_1|$. By definition, $u_{\rm in}\in\Sigma_0$ implies that $\Psi_{t_1}u_{\rm in}\in\Sigma_0$. 

Moreover, \eqref{complete-mass-rate} and the mass comparison below show that $\Psi_{t_1}u_{\rm in}\in\Sigma_{\rm mass}$ whenever $u_{\rm in}\in\Sigma$. Indeed, the finite-dimensional mass conservation and \eqref{complete-finite-limit} imply
\begin{equation}\label{m_N-difference}
 |m_N(\Psi_{t_1}u_{\rm in})-m_N(u_{\rm in})|
 \lesssim N^{-\beta}
 \big(\|\Pi_N\Psi_{t_1}u_{\rm in}\|_{L^2}
     +\|\Psi_{t_1}^N\Pi_Nu_{\rm in}\|_{L^2}\big)\lesssim
 N^{-\beta+\frac{2-\alpha}{2}}
 \longrightarrow0,  
\end{equation}
where the exponent is negative for $\alpha>\frac{29}{15}$. This ends the proof of $\Psi_{t_1}(\Sigma)=\Sigma$ and $\Phi_{t_1}(\Sigma)=\Sigma$. In particular, both $\Psi_{t_1+t_2}u_{\rm in}$ and $\Psi_{t_2}\Psi_{t_1}u_{\rm in}$ belong to $\Sigma$ for any $t_1,t_2$.

Furthermore, for fixed $t_1,t_2$, we may continue from \eqref{continue} and \eqref{continue2} to derive that both $\Psi^{\overline{N}}_{t_2}\Psi_{t_1}^{\overline{N}}\Pi_{\overline{N}}u_{\rm in}$ and $\Psi^{\overline{N}}_{t_2}\Pi_{\overline{N}}\Psi_{t_1}u_{\rm in}$ lie in the $\mathcal{O}_{u_{\rm in},t_1,t_2}(1)\overline{N}^{-\beta}$-neighborhood of $\Psi^{\overline{N}}_{t_1+t_2-t_0}v'$ in $L^2$ by stability. Then the group property of finite-dimensional flow map gives,
\[
 \|\Psi_{t_1+t_2}^N\Pi_Nu_{\rm in}
 -\Psi_{t_2}^N\Pi_N\Psi_{t_1}u_{\rm in}\|_{L^2}
 \leq C(u_{\rm in},t_1,t_2)N^{-\beta}.
\]
Passing to the limit yields $\Psi_{t_2}\Psi_{t_1}u_{\rm in}=\Psi_{t_1+t_2}u_{\rm in}$. Recall \eqref{complete-inverse-gauge}, note that the phase functional is conserved. More precisely, $X(\Psi_{t_1}u_{\rm in})=X(u_{\rm in})$. In fact, by \eqref{complete-phase-exponent}, \eqref{m_N-difference} together with
\eqref{complete-mass-rate} gives
\[
 |m_N(\Psi_tu_{\rm in})-\sigma_N-X(u_{\rm in})|\lesssim_{u_{\rm in},t}N^{-\beta_1},
\]
after decreasing $\beta_1$ slightly. Then the group property of $\Psi_t$ and \eqref{global-inverse-gauge} gives $\Phi_{t_2}\Phi_{t_1}u_{\rm in}=\Phi_{t_1+t_2}u_{\rm in}$.

\subsection{Invariance of the Gibbs measure}

The map $\Phi_t$ is Borel on $\Sigma$, being a pointwise limit of continuous
finite-dimensional maps.  We first justify the reduction to sets with
uniform constants, which is the analogue of
\cite[Proposition 6.2]{DNY2024}.  For fixed $t$, let $\mathfrak C_t(u)$ be the least
integer which bounds all constants needed in
\eqref{complete-finite-limit}, the iterated stability estimates on
$[-|t|-1,|t|+1]$, and the mass bounds
\eqref{complete-mass-rate}--\eqref{complete-L2-growth}.  The construction
from the countable good sets shows that $\mathfrak C_t$ is Borel and finite
on $\Sigma$.  Set
\[
 S_m(t)=\{u\in\Sigma:\mathfrak C_t(u)\leq m\}.
\]
Then $S_m(t)\uparrow\Sigma$.  Let
$\nu_t=(\Phi_{-t})_\#\rho_\alpha$.  Since $\Phi_t(\Sigma)=\Sigma$ and
$\rho_\alpha(\Sigma)=1$, we also have $\nu_t(\Sigma)=1$, and hence
$\nu_t(S_m(t))\to1$.

It is therefore enough to prove the following estimate
\begin{equation}\label{complete-invariance-compact}
 \rho_\alpha(\Phi_tE)\leq\rho_\alpha(E)
\end{equation}
for compact $E\subset S_m(t)$.  Indeed, if it holds for every such compact
set, inner regularity gives
\[
 \nu_t(A\cap S_m(t))\leq\rho_\alpha(A\cap S_m(t))
\]
for every Borel set $A$, by applying the compact estimate to compact subsets of
$A\cap S_m(t)$ and taking the supremum.  Letting $m\to\infty$ yields
$\nu_t(A)\leq\rho_\alpha(A)$.  Thus the estimates for compact sets with uniform constants imply the desired inequality for every Borel set.

We now fix compact $E\subset S_m(t)$.  All constants below are uniform.
The proof of Proposition \ref{prop:complete-global-cauchy}, with those
uniform constants, gives uniform convergence on $E$ in $H^s$ of the
continuous maps $\Phi_t^N\Pi_N$. Hence $\Phi_t\big|_E$ is continuous and
$\Phi_tE$ is compact. By \eqref{continue2} and Proposition \ref{prop:complete-stability}, we have $\left\| \Psi^N_{-t}\Pi_N\Psi_{t}u_{\rm in} -\Psi^N_{-t_0}v' \right\|_{L^2}\leq C_{E,t}N^{-\beta}$. Combining this with \eqref{perturb-u}, we obtain $\|\Psi_{-t}^N\Pi_N\Psi_tu-\Pi_Nu\|_{L^2}\leq C_{E,t}N^{-\beta}$. By \eqref{complete-mass-rate}, we can replace $\Psi_t$ and $\Psi^N_t$ with $\Phi_t$ and $\Phi^N_t$. More precisely, we obtain
\[
 \|\Phi_{-t}^N\Pi_N\Phi_tu-\Pi_Nu\|_{L^2}
 \leq C_{E,t}N^{-\beta_2},
\]
where $0<\beta_2<
 \min\left\{\beta,\beta_1-\frac{2-\alpha}{2}\right\}$. Denoting by $\mathscr B_{L^2}(R)$ the ball of radius $R$ centered at $0$ in $L^2$ , we therefore have the relation
\begin{equation}\label{complete-set-inclusion}
 \Pi_N\Phi_tE
 \subseteq
 \Phi_t^N\big(\Pi_NE+\mathscr B_{L^2}(C_{E,t}N^{-\beta_2})\big).
\end{equation}

Let $A_N=\Pi_N^{-1}
 \big(\Pi_NE+\mathscr B_{L^2}(C_{E,t}N^{-\beta_2})\big)$. By finite-dimensional invariance, \eqref{complete-set-inclusion}, and
Lemma \ref{lem:complete-TV},
\[
 \rho_\alpha(\Phi_tE)
 \leq\limsup_{N\to\infty}
 \rho_{\alpha,N}^{\circ}
 \big(\Pi_NE+\mathscr B_{L^2}(C_{E,t}N^{-\beta_2})\big)=\limsup_{N\to\infty}\rho_\alpha(A_N).
\]
Indeed, the first inequality follows from
$\Phi_tE\subseteq\Pi_N^{-1}(\Pi_N\Phi_tE)$, followed by
\eqref{complete-set-inclusion}; in both displayed comparisons the error is
$o_{N\to\infty}(1)$ by \eqref{complete-TV}.
It suffices to prove
\begin{equation}\label{complete-limsup-compact}
 \limsup_{N\to\infty}A_N\subseteq E.
\end{equation}
Indeed, if $u\in A_N$ for infinitely many $N$, choose $u_N\in E$ with
$\|\Pi_N(u-u_N)\|_{L^2}\leq C_{E,t}N^{-\beta_2}$.  Compactness gives, after a
subsequence, $u_N\to v\in E$ in $H^s$.  For every fixed Fourier mode,
the $L^2$ estimate and the $H^s$ convergence imply
$\widehat u(k)=\widehat v(k)$, so $u=v\in E$.

The reverse Fatou inequality for the bounded indicators
$\mathbf1_{A_N}$ now gives
\[
 \limsup_N\rho_\alpha(A_N)
 \leq\rho_\alpha(\limsup_NA_N)
 \leq\rho_\alpha(E).
\]
This proves \eqref{complete-invariance-compact}.  Replacing $t$ by $-t$ and
using the group property gives the reverse inequality.  The reduction above
then yields equality for every Borel set, completing the proof of Theorem
\ref{GWP}.

\bibliographystyle{plain}
\bibliography{mybib}

\begin{thebibliography}{10}

\bibitem{B94}
Jean Bourgain.
\newblock Periodic nonlinear {S}chr\"odinger equation and invariant measures.
\newblock {\em Comm. Math. Phys.}, 166(1):1--26, 1994.

\bibitem{B96}
Jean Bourgain.
\newblock Invariant measures for the {$2$}d-defocusing nonlinear {S}chr\"odinger equation.
\newblock {\em Comm. Math. Phys.}, 176(2):421--445, 1996.

\bibitem{B97}
Jean Bourgain.
\newblock Invariant measures for the {G}ross-{P}iatevskii equation.
\newblock {\em J. Math. Pures Appl. (9)}, 76(8):649--702, 1997.

\bibitem{BBNLS2D}
Jean Bourgain and Aynur Bulut.
\newblock Almost sure global well posedness for the radial nonlinear {S}chr\"odinger equation on the unit ball {I}: the 2{D} case.
\newblock {\em Ann. Inst. H. Poincar\'e{} C Anal. Non Lin\'eaire}, 31(6):1267--1288, 2014.

\bibitem{BBNLS3D}
Jean Bourgain and Aynur Bulut.
\newblock Almost sure global well-posedness for the radial nonlinear {S}chr\"odinger equation on the unit ball {II}: the 3{D} case.
\newblock {\em J. Eur. Math. Soc. (JEMS)}, 16(6):1289--1325, 2014.

\bibitem{BBNLW3D}
Jean Bourgain and Aynur Bulut.
\newblock Invariant {G}ibbs measure evolution for the radial nonlinear wave equation on the 3{D} ball.
\newblock {\em J. Funct. Anal.}, 266(4):2319--2340, 2014.

\bibitem{Bring23}
Bjoern Bringmann.
\newblock Invariant {G}ibbs measures for the three-dimensional wave equation with a {H}artree nonlinearity {II}: dynamics.
\newblock {\em J. Eur. Math. Soc. (JEMS)}, 26(6):1933--2089, 2024.

\bibitem{PLWick}
Giuseppe Da~Prato and Luciano Tubaro.
\newblock Wick powers in stochastic {PDE}s: an introduction.
\newblock In {\em Quantum and stochastic mathematical physics}, volume 377 of {\em Springer Proc. Math. Stat.}, pages 1--15. Springer, Cham, [2023] \copyright 2023.

\bibitem{DWWZ}
Yangkendi Deng, Han Wang, Yuzhao Wang, and Zehua Zhao.
\newblock On restricted-type {S}trichartz estimates and the applications, 2025.
\newblock arxiv:2508.18827.

\bibitem{DNY2021}
Yu~Deng, Andrea~R. Nahmod, and Haitian Yue.
\newblock Invariant {G}ibbs measure and global strong solutions for the {H}artree {NLS} equation in dimension three.
\newblock {\em J. Math. Phys.}, 62(3):Paper No. 031514, 39, 2021.

\bibitem{DNY2022}
Yu~Deng, Andrea~R. Nahmod, and Haitian Yue.
\newblock {Random tensors, propagation of randomness, and nonlinear dispersive equations}.
\newblock {\em Invent. Math.}, 228(2):539--686, 2022.

\bibitem{DNY2024}
Yu~Deng, Andrea~R. Nahmod, and Haitian Yue.
\newblock {Invariant {G}ibbs measures and global strong solutions for nonlinear {S}chr\"odinger equations in dimension two}.
\newblock {\em Ann. of Math. (2)}, 200(2):399--486, 2024.

\bibitem{FOSW2021}
Chenjie Fan, Yumeng Ou, Gigliola Staffilani, and Hong Wang.
\newblock 2{D}-defocusing nonlinear {S}chr\"odinger equation with random data on irrational tori.
\newblock {\em Stoch. Partial Differ. Equ. Anal. Comput.}, 9(1):142--206, 2021.

\bibitem{LO2023fwave}
Luigi Forcella and Oana Pocovnicu.
\newblock {Invariant Gibbs dynamics for two-dimensional fractional wave equations in negative Sobolev spaces}, 2025.
\newblock arxiv:2306.07857.

\bibitem{GlimmJaffe}
James Glimm and Arthur Jaffe.
\newblock {\em Quantum physics: A functional integral point of view}.
\newblock Springer-Verlag, New York, second edition, 1987.

\bibitem{GIP15}
Massimiliano Gubinelli, Peter Imkeller, and Nicolas Perkowski.
\newblock Paracontrolled distributions and singular {PDE}s.
\newblock {\em Forum Math. Pi}, 3:e6, 75, 2015.

\bibitem{GIP16}
Massimiliano Gubinelli, Peter Imkeller, and Nicolas Perkowski.
\newblock A {F}ourier analytic approach to pathwise stochastic integration.
\newblock {\em Electron. J. Probab.}, 21:Paper No. 2, 37, 2016.

\bibitem{Hairer13}
Martin Hairer.
\newblock Solving the {KPZ} equation.
\newblock {\em Ann. of Math. (2)}, 178(2):559--664, 2013.

\bibitem{Hairer14}
Martin Hairer.
\newblock Singular stochastic {PDE}s.
\newblock In {\em Proceedings of the {I}nternational {C}ongress of {M}athematicians---{S}eoul 2014. {V}ol. {IV}}, pages 49--73. Kyung Moon Sa, Seoul, 2014.

\bibitem{Hairer14reg}
Martin Hairer.
\newblock A theory of regularity structures.
\newblock {\em Invent. Math.}, 198(2):269--504, 2014.

\bibitem{Hairer15}
Martin Hairer.
\newblock Introduction to regularity structures.
\newblock {\em Braz. J. Probab. Stat.}, 29(2):175--210, 2015.

\bibitem{Count23}
Ralph Howard and Ognian Trifonov.
\newblock Bounding the number of lattice points near a convex curve by curvature.
\newblock {\em Funct. Approx. Comment. Math.}, 69(2):219--245, 2023.

\bibitem{KLS2012limit}
Kay Kirkpatrick, Enno Lenzmann, and Gigliola Staffilani.
\newblock On the continuum limit for discrete {NLS} with long-range lattice interactions.
\newblock {\em Comm. Math. Phys.}, 317(3):563--591, 2013.

\bibitem{Kup16}
Antti Kupiainen.
\newblock Renormalization group and stochastic {PDE}s.
\newblock {\em Ann. Henri Poincar\'e}, 17(3):497--535, 2016.

\bibitem{FQM}
Nick Laskin.
\newblock {\em Fractional quantum mechanics}.
\newblock World Scientific Publishing Co. Pte. Ltd., Hackensack, NJ, 2018.

\bibitem{Laskin00}
Nikolai Laskin.
\newblock Fractional quantum mechanics and {L}\'evy path integrals.
\newblock {\em Phys. Lett. A}, 268(4-6):298--305, 2000.

\bibitem{LRS}
Joel~L. Lebowitz, Harvey~A. Rose, and Eugene~R. Speer.
\newblock Statistical mechanics of the nonlinear {S}chr\"odinger equation.
\newblock {\em J. Statist. Phys.}, 50(3-4):657--687, 1988.

\bibitem{Wang2021}
Rui Liang and Yuzhao Wang.
\newblock Gibbs measure for the focusing fractional {NLS} on the torus.
\newblock {\em SIAM J. Math. Anal.}, 54(6):6096--6118, 2022.

\bibitem{Wang2024}
Rui Liang and Yuzhao Wang.
\newblock Gibbs dynamics for fractional nonlinear {S}chr\"odinger equations with weak dispersion.
\newblock {\em Comm. Math. Phys.}, 405(10):Paper No. 250, 69, 2024.

\bibitem{LTW23}
Ruoyuan Liu, Nikolay Tzvetkov, and Yuzhao Wang.
\newblock Existence, uniqueness, and universality of global dynamics for the fractional hyperbolic $\phi^4_3$-model, 2024.
\newblock arxiv:2311.00543.

\bibitem{OhPL21}
Tadahiro Oh, Philippe Sosoe, and Leonardo Tolomeo.
\newblock Optimal integrability threshold for {G}ibbs measures associated with focusing {NLS} on the torus.
\newblock {\em Invent. Math.}, 227(3):1323--1429, 2022.

\bibitem{Oh2015APA}
Tadahiro Oh and Laurent Thomann.
\newblock A pedestrian approach to the invariant {G}ibbs measures for the 2-{$d$} defocusing nonlinear {S}chr\"odinger equations.
\newblock {\em Stoch. Partial Differ. Equ. Anal. Comput.}, 6(3):397--445, 2018.

\bibitem{OTW2020}
Tadahiro Oh, Nikolay Tzvetkov, and Yuzhao Wang.
\newblock Solving the 4{NLS} with white noise initial data.
\newblock {\em Forum Math. Sigma}, 8:Paper No. e48, 63, 2020.

\bibitem{Simon.B}
Barry Simon.
\newblock {\em The {$P(\phi )\sb{2}$} {E}uclidean (quantum) field theory}.
\newblock Princeton Series in Physics. Princeton University Press, Princeton, NJ, 1974.

\bibitem{SS99}
Catherine Sulem and Pierre-Louis Sulem.
\newblock {\em The nonlinear {S}chr\"odinger equation: Self-focusing and wave collapse}, volume 139 of {\em Applied Mathematical Sciences}.
\newblock Springer-Verlag, New York, 1999.

\bibitem{Sun2019GibbsMD}
Chenmin Sun and Nikolay Tzvetkov.
\newblock Gibbs measure dynamics for the fractional {NLS}.
\newblock {\em SIAM J. Math. Anal.}, 52(5):4638--4704, 2020.

\bibitem{Sun2021}
Chenmin Sun and Nikolay Tzvetkov.
\newblock Refined probabilistic global well-posedness for the weakly dispersive {NLS}.
\newblock {\em Nonlinear Anal.}, 213:Paper No. 112530, 91, 2021.

\bibitem{weakuniversality}
Chenmin Sun, Nikolay Tzvetkov, and Weijun Xu.
\newblock {Weak universality results for a class of nonlinear wave equations}.
\newblock Online first, 2025.

\bibitem{MX22}
Mouhamadou Sy and Xueying Yu.
\newblock Global well-posedness and long-time behavior of the fractional {NLS}.
\newblock {\em Stoch. Partial Differ. Equ. Anal. Comput.}, 10(4):1261--1317, 2022.

\bibitem{Tzvetkov2006InvariantMF}
Nikolay Tzvetkov.
\newblock Invariant measures for the nonlinear {S}chr\"odinger equation on the disc.
\newblock {\em Dyn. Partial Differ. Equ.}, 3(2):111--160, 2006.

\bibitem{Tzvetkov2008InvariantMF}
Nikolay Tzvetkov.
\newblock Invariant measures for the defocusing nonlinear {S}chr\"odinger equation.
\newblock {\em Ann. Inst. Fourier (Grenoble)}, 58(7):2543--2604, 2008.

\bibitem{Wang2026fNLSsobolev}
Jiajun Wang.
\newblock {Polynomial growth of Sobolev norms of solutions of the fractional NLS equation on {$\mathbb{T}^d$}}, 2026.
\newblock arxiv:2603.24906.

\end{thebibliography}

\end{document}